\documentclass[12pt]{amsart}
\usepackage[headings]{fullpage}
\usepackage{longtable,array}
\newcolumntype{L}{>{$}l<{$}}
\usepackage{listings}
\usepackage{rotating}
\usepackage{amssymb}
\usepackage{xcolor}
\usepackage{hyperref}
\theoremstyle{plain}
\allowdisplaybreaks
\newtheorem{theorem}{Theorem}[section]

\newtheorem{prop}[theorem]{Proposition}

\theoremstyle{remark}
\newtheorem{remark}[theorem]{Remark}
\newtheorem{exmp}{Example}
\theoremstyle{definition}
\newtheorem{defn}{Definition}[section]
\newcommand{\PP}{\mathbb P}
\newcommand{\lra}{\longrightarrow}
\newcommand\CC{{\mathbb C}}

\newcommand\ZZ{{\mathbb Z}}
\newcommand\QQ{{\mathbb Q}}
\newcommand\mt{{\mathcal T}}

\numberwithin{table}{section}
\begin{document}
\leavevmode
\def\wl#1{\hbox to 10mm{\hfill$\ifcase#1\or I_{2}\or I_{4}\or I_{0}^{*}\or I_{2}^{*}\fi$\hfill}}

\lstset{language=Magma,breaklines=false,breakatwhitespace=true,basicstyle=\tiny, comment=[s]{/*}{*/},commentstyle=\tiny\emph, showstringspaces=false}

\title[Classification of double octic Calabi--Yau threefolds II]
{Classification of double octic Calabi--Yau threefolds defined by an
  arrangement of eight planes II}
        \author{S{\l}awomir Cynk}
\address{Faculty of Mathematics and Computer Science, Jagiellonian University,
	{\L}ojasiewicza 6, 30-348 Krak\'ow, Poland
}
\email{slawomir.cynk@uj.edu.pl}
\author{Beata Kocel--Cynk}
\address{Department of Applied Mathematics\\
	Faculty of Computer Science and Telecommunications\\
	Cracow University of Technology\\
	Warszawska~24\\
	31-155~Krak\'ow\\
	POLAND}
\email{beata.kocel-cynk@pk.edu.pl}
\date{}


\maketitle
\begin{abstract}
We present a complete classification of all arrangements of eight planes in projective three‑space that give rise to double octic Calabi--Yau threefolds. Building on earlier work, we determine all 455 combinatorial types and describe the projective automorphism groups associated with each arrangement. These automorphisms allow us to identify geometrically distinguished subfamilies and elements, and to construct a rich collection of elliptic and K3 fibrations arising naturally from the singularities of the arrangements.

As applications, we exhibit two one‑parameter families of Calabi--Yau threefolds that each possess three pairwise non‑equivalent maximally unipotent monodromy points---phenomena that are rare in known examples. We further construct two Calabi--Yau threefolds whose third cohomology decomposes, via complex multiplication, into complementary two‑dimensional Hodge structures of types (1,0,1,0) and (0,1,0,1). This decomposition yields new instances of modularity not previously observed in the Calabi--Yau setting.

\end{abstract}

\section*{Introduction}

In this paper we continue the study of double octic Calabi--Yau threefolds
initiated in \cite{CK-C}. A Calabi--Yau threefold is a complex projective
manifold $X$ of dimension $3$ with trivial canonical divisor $K_X = 0$
and vanishing first cohomology of the structure sheaf,
$H^1(\mathcal O_X)=0$.

We study Calabi--Yau threefolds arising as double coverings of the
projective space $\PP^3$ branched along an arrangement of eight planes.
If the arrangement $D$ satisfies mild conditions
(no six planes intersect and no four planes contain a common line),
then the double covering
\[
\pi \colon Y \to \PP^3
\]
branched along $D$ admits a resolution of singularities $X$ which is a
Calabi--Yau threefold, called a \emph{double octic} \cite{CSz}.

Double octics form a relatively small but very suitable class of
Calabi--Yau threefolds for explicit study.
Their invariants (such as the topological Euler characteristic and
Hodge numbers) can be computed easily using computer algebra
\cite{CvS}.
Recently M. Oczko gave a general method to construct a semistable degeneration of a one-parameter family of double octics.
This class is rich enough to provide examples illustrating numerous
interesting phenomena:
\begin{itemize}
	\item examples of one-parameter families of Calabi--Yau threefolds
	without a point of maximal unipotent monodromy, and an example of a
	family with three non-equivalent MUM points \cite{CvS};
	\item a Calabi--Yau threefold in characteristic $3$ that is not
	liftable to characteristic $0$, and a Calabi--Yau threefold in
	characteristic $5$ with obstructed infinitesimal deformations; the
	latter provides a counterexample to the Bogomolov--Tian--Todorov
	theorem in positive characteristic \cite{CvS3};
	\item an example of a Hilbert modular Calabi--Yau threefold with parallel weights [4,4],, a rigid example defined over $\QQ$\cite{CScvS},
	\item  an example of a Hilbert modular Calabi--Yau threefold with weights [4,2] and [2,4] and a real multiplication explaining the decomposition of the Galois representation into Galois-conjugate two-dimensional pieces (and the lack of symmetry between them) \cite{CScvS};
	\item an example of a Calabi--Yau threefold with trivial monodromy and
	no smooth filling, showing that the theorem of Kulikov, Persson, and
	Pinkham does not generalize to dimension $3$ \cite{CvS5};
	\item an example of a Calabi--Yau threefold over a $p$-adic field with unramified $l$-adic cohomology yet having bad reduction which provides a counterexample to the N\'eron--Ogg--Shafarevich criterion in dimension 3 \cite{CO}.
\end{itemize}

C.~Meyer carried out an extensive computer search for double octic
Calabi--Yau threefolds in \cite{Meyer}. His method was based on studying
arrangements of eight planes with small integer coefficients. This
search produced $450$ types of octic arrangements, distinguished by the
types of subarrangements of six planes, which are listed in
\cite[App.~A]{Meyer}.

In \cite{CK-C} we proposed a combinatorial approach to the study of
double octic Calabi--Yau threefolds and used it to classify arrangements
of eight planes defining Calabi--Yau threefolds with Hodge number
$h^{1,2} \le 1$.

In the present paper we apply a similar method to obtain a complete
classification of arrangements of eight planes defining a
Calabi--Yau threefold. We show that there exist $455$ combinatorial
types of octic arrangements. For all types we study equations of octic
arrangements with a given incidence table. For $454$ types there exists
a family parametrized by a quasi-projective rational variety such that
every octic arrangement with a given incidence table is projectively
equivalent to an element of this family. There exists one minimal
incidence table (No.~453) for which there are two projectively
non-equivalent arrangements. These two arrangements are given by Galois
conjugate sequences of eight linear forms with coefficients in the
quadratic field $\QQ(\sqrt{5})$.

In Section~\ref{sec:projtransf} we describe birational transformations of families of
double octic Calabi--Yau threefolds induced by projective
transformations of $\PP^3$ corresponding to permutations of the eight
planes preserving intersections. Using invariant permutations of an
octic arrangement, we define in Section~\ref{sec:dist} geometrically distinguished
subfamilies and distinguished elements of families of double octics
corresponding to jumps in the group of projective self-maps. A quotient
of a Calabi--Yau threefold by a finite group of automorphisms preserving
the canonical form admits a crepant resolution of singularities.

As an application, we construct two new families of Calabi--Yau
threefolds with three non-equivalent points of maximal unipotent
monodromy. We also find two double octic Calabi--Yau threefolds admitting
birational transformations acting on the middle cohomology (of
dimension $4$) as complex multiplication. In particular, this action
splits the associated Galois representation (over a quadratic field)
into a direct sum of two complementary two-dimensional Hodge structure of types (1,0,1,0) and (0,1,0,1), leading to a new instances of modularity .not previously observed in the Calabi--Yau setting.

In Section~\ref{sec:fibr} we describe a large number of fibrations on double octic
Calabi--Yau threefolds. Elliptic and K3 fibrations on a Calabi--Yau manifold carry a lot of information about its arithmetic and geometry.  These fibrations help illuminate the birational geometry of double octics and often lead to identifications between seemingly unrelated families. In particular, explicit descriptions of fibrations play an
important role in most of the listed constructions involving double octics.

In Section~\ref{sec:EqsData} we list equations of 455 families of octic arrangements and Magma \cite{Magma} code computing additional information:
incidence table, singularities, the Euler characteristic, Hodge numbers.

\section{Double octic Calabi--Yau threefolds}
In this section we recall basic information on Calabi--Yau
threefolds constructed as a resolution of singularities of a double
covering of $\PP^{3}$ branched along a union of eight planes.
\begin{defn}\label{def:oa}
An \emph{octic arrangement} is a union of eight planes
 $D=P_{1}\cup\dots\cup P_{8}$ in
$\PP^{3}$ satisfying the following two conditions
\begin{itemize}
  \itemsep=2mm
\item the intersection $P_{i_{1}}\cap\dots\cap P_{i_{6}}$ of any six is empty
  ($1\le i_{1}< i_{2}<\dots< i_{6}\le 8$),
\item the intersection $P_{i_{1}}\cap P_{i_{2}}\cap P_{i_{3}}\cap
  P_{i_{4}}$ of any four does not contain a line i.e. it is the  empty
  set or consists of one
  point ($1\le i_{1}< i_{2}< i_{3}< i_{4}\le 8$).
\end{itemize}
\end{defn}
The surface $D$  is singular at the intersection points of the
components. There are eight types of singularities of octic arrangements
\begin{itemize}
\item [$l_{2}$] --  double line,
\item [$l_{3}$] --  triple line,
\item [$p_{3}$] --  triple point (not on a triple line),
\item [$p^{0}_{4}$] --  fourfold point not lying on a triple line,
\item [$p^{1}_{4}$] --  fourfold point lying on  one triple line,
\item [$p^{0}_{5}$] --  fivefold point not lying on a triple line,
\item [$p^{1}_{5}$] --  fivefold point lying on  one triple line,
\item [$p^{2}_{5}$] --  fivefold point lying on  two triple lines.
\end{itemize}
We shall denote the numbers of multiple lines and points by the same symbols
$l_{2}$, $l_{3},p_{3}$, $p_{4}^{0}$, $p_{4}^{1}$, $p_{5}^{0}$, $p_{5}^{1}$, $p_{5}^{2}$.

\begin{theorem}[\mbox{\cite[Thm.~2.1]{CSz}}]
  For an octic arrangement $D$ there exists a resolution of
singularities $X$ of the double cover of $\PP^{3}$ branched along $D$
that is a smooth Calabi--Yau threefold.
\end{theorem}
Double cover of $\PP^{3}$ branched along an octic arrangement can be defined in the total space of the line bundle $\mathcal O_{\PP^{3}}(4)$ (cf. \cite{Clemens}), we shall use an alternative embedding as a hypersurface in the weighted projective space $\PP(1,1,1,1,4)$
\[Y:=\{(x,y,z,t,u)\in\PP(1,1,1,1,4): u^{2}=F(x,y,z,t)\}\subset \PP(1,1,1,1,4),\]
where $F$ is the equation of the octic arrangement $D$.

We shall briefly recall the resolution of singularities introduced in \cite{CSz}.
For any octic arrangement $D=P_{1}\cup\dots\cup D_{8}$ there exists a sequence of blow-ups
\[\tilde\PP=\PP_{(k)}\xrightarrow{\ \ \sigma_{k}\ \ }
\PP_{(k-1)}\xrightarrow{\ \ \sigma_{k-1}\ \ }\dots
\xrightarrow{\ \ \sigma_{2}\ \ }\PP_{(1)}\xrightarrow{\ \ \sigma_{1}\ \ }\PP_{(0)}=\PP^{3}
\]
and a sequence of reduced, even divisors $D^{i}\subset \PP_{(i)}$ such that
\begin{itemize}
	\item every intersection of components of $D_{(i)}$ is smooth and irreducible,
	\item the center $C_{i}$ of the blow-up $\sigma_{i}:\PP_{i}\lra \PP_{i-1}$
	is a smooth curve contained in 2 or 3 components of $D^{(i)}$ or a point contained in  4 or 5 components of $D_{i}$,
	\item $\tilde D^{i-1}\le D^{i}\le \sigma^{*}D^{i-1}$,
	\item $D^{0}=D$,
	\item $D^{*}:=D^{k}$ is a non-singular divisor (i.e. irreducible components of $D^{*}$ are pairwise disjoint).
\end{itemize}
Then the double covering $X$ of $\tilde\PP$ branched along $D^{*}$ is a crepant resolution of singularities of $Y$. In particular, $X$ is a Calabi--Yau threefold.

The center of a blow-up $\sigma_{i}$ is chosen among prime components of the singular subscheme of the divisor $D^{i-1}$ (intersections of components of $D^{i-1}$) in the following order: fivefold points, triple lines, fourfold-points and double lines. In particular we never blow-up triple points nor fourfold points that belong to a triple curve ($p_{4}^{1}$ type). This algorithm is described in details in \cite{CSz}. Ingals-Logan \cite{IL} generalized this construction to an arbitrary dimension.

This resolution of singularities is not unique, it depends on the order of double lines intersecting at a triple point. Different resolutions differs by a finite number of flops. In particular they have equal topological invariants: the Euler characteristic and Hodge numbers.

The Euler characteristic of a double octic $X$ is given by a simple formula involving only the arrangement invariants (\cite[Thm.~2.1]{CSz})
\[e(X)=40+4p_{4}^{0}+3p_{4}^{1}+16p_{5}^{0}+18p_{5}^{1}+20p_{5}^{2}+l_{3}.\]
The Hodge numbers of $X$ are more difficult to compute:  the octic  arrangements Nos. 241 and 242 (\cite{Meyer}) have equal arrangement invariants ($p_{3}=16$, $p_{4}^{0}=10$, $p_{4}^{1}=p_{5}^{0}=p_{5}^{1}=p_{5}^{2}=l_{3}=0$) and different Hodge numbers.
The Hodge number $h^{1,2}(X)$ equals the dimension of the Kuranishi space of $X$.

The Kuranishi space (universal deformation) of a double octic $X$ is
given by the space of equisingular deformations (i.e. family of octic
arrangements preserving the types of singularities) modulo trivial
deformations (induced by projective automorphisms of
$\PP^{3}$). The dimension of the Kuranishi space equals the Hodge number
$h^{1,2}(X)$ and can be computed with a computer algebra system via
the equisingular ideal (cf. \cite{CvS}). More precisely, let $f_{i}$
be the equation of the plane
$D_{i}$, $i=1,\dots,8$, and let $F=f_{1}\cdot\ldots\cdot f_{8}$.
The \emph{equisingular ideal} of $D$ is defined by
\[I_{\rm eq}=\bigcap_{C}(I_{C}^{\operatorname{mult}_{C}(D)}+J_{F})
\]
the intersection being taken over all double and triple lines and
multiple points of $D$, where $J_{F}$ is the homogeneous ideal
generated by the partial derivatives  of $F$. Then
\[h^{1,2}(X)=\dim_{\CC}(I_{\rm eq}/J_{F})_{8}.\]

\section{Classification}

\label{sec:cl}
In this section we apply the algorithm introduced in
\cite{CK-C} to perform a complete classification of all octic arrangements satisfying the Calabi--Yau condition (definition~\ref{def:oa}), up to projective transformations.

This algorithm is based on a combinatorial encoding of an octic arrangement in its incidence table.

\begin{defn}
The \emph{incidence table} of an octic arrangement
$D=P_{1}\cup\dots\cup P_{8}$ is  the lexicographically sorted list of all quadruples $1\le
i_{1}<\dots<i_{4}\le8$, such that the planes
$P_{i_{1}},\dots, P_{i_{4}}$ intersect.
\end{defn}
The lexicographically sorted incidence table is independent of the coordinates in
$\PP^{3}$, however it depends on the order of the arrangement
planes. A permutation of arrangement planes may change the incidence
table.
\begin{defn}
The \emph{minimal incidence table} of an octic arrangement is the
minimum of incidence tables over all permutations of its components.
\end{defn}
The minimal incidence table of an octic arrangement is independent of
the coordinate system in $\PP^{3}$ and the order of eight planes.
In \cite{CK-C} we introduced an algorithm to classify minimal
incidence tables of arrangements of eight planes satisfying
definition~\ref{def:oa}. We study an incidence table of an arrangement
$D=P_{1}\cup\dots\cup P_{8}$ together with two auxiliary lists:
\begin{itemize}
	\item $\mathcal T_{3}$: the  list  of triplets
	\(1\le i_{1}<i_{2}<i_{3}\le8\) such that the intersection of planes
	$P_{i_{1}}, P_{i_{2}}, P_{i_{3}}$ contains a line,
	\item  $\mathcal
	T_{5}$: the list of quintuples $1\le
	i_{1}<\dots<i_{5}\le8$, such that the planes
	$P_{i_{1}},\dots, P_{i_{5}}$ intersect.
\end{itemize}
If an arrangement of eight planes satisfies condition
(definition~\ref{def:oa}) then the lists $\mathcal T_{3}$ and $\mathcal T_{5}$
can be recovered from $\mathcal T$ by the following two properties
\begin{itemize}
  \itemsep=2mm
\item [(L1)] a triplet $T$ belongs to $\mathcal T_{3}$ iff
every quadruple containing $T$ belongs to $\mathcal T$,
\item [(L2)] a quintuple $T$ belongs to $\mathcal T_{5}$ iff every quadruple
  contained in $T$ belongs to~$\mathcal T$.
\end{itemize}
Moreover, the lists $\mathcal T_{3}$, $\mathcal T$ and $\mathcal
T_{5}$ satisfy the following properties:
\begin{itemize}
  \itemsep=2mm
\item [(L3)] If $\#(Q_{1}\cap Q_{2})=3$ for $Q_{1},Q_{2}\in\mathcal T$
then  $Q_{1}\cap Q_{2}\in \mathcal T_{3}$ or $Q_{1}\cup
Q_{2}\in\mathcal T_{5}$.
\item [(L4)] If  $Q_{1},Q_{2}\in\mathcal T_{3}$, $Q_{1}\not=Q_{2}$
  then $\#(Q_{1}\cap Q_{2})\le1$,
\item [(L5)] If  $Q_{1},Q_{2}\in\mathcal T_{5}$, $Q_{1}\not=Q_{2}$
  then $\#(Q_{1}\cap Q_{2})\le3$,
\item [(L6)] $\#\mathcal T_{3}\le 4$ and $\#\mathcal T_{5}\le 4$,
\end{itemize}
Finally, if $\mathcal T$ is a minimal incidence table of an octic arrangement then $\mathcal T$ satisfies also
\begin{itemize}
\item [(L7)] $\mathcal T$ is minimal among $\mathcal T^{g}$ for $g\in\Sigma_{8}$,
\end{itemize}
where $\mathcal T^{g}$ is obtained from $\mathcal T$ by an application of a permutation $g\in\Sigma_{8}$ to each entry of every quadruple in $\mathcal T$ (and sorting first every quadruple and then the list of quadruples).
\begin{prop}
	There are 515 triplets $[\mathcal T_{3}, \mathcal T, \mathcal
	T_{5}]$ satisfying $(L1)-(L7)$.
\end{prop}
\begin{proof}
We shall recursively construct a list $\mathcal L_{k}$ of triplets $[\mathcal T_{3}, \mathcal
T, \mathcal T_{5}]$ such that the union $\bigcup_{k\ge0}\mathcal L_{k}$ consists of all triplets satisfying $(L1)-(L7)$.

We start with the list
$\mathcal L_{0}$ that contains a single element with empty lists
$\mt_{3}$, $\mt$, $\mt_{5}$. Note that this element corresponds to an octic arrangement equal the union of eight planes in general position.

Assume that we have constructed a list $\mathcal L_{k-1}$, for some $k\ge1$.
Define
\[\mathcal M:=[[\mt_{3},\mt\cup \{Q\},\mt_{5}]:
[\mt_{3},\mt,\mt_{5}]\in\mathcal L_{k-1}, Q\in S_{4}\setminus
\mt],\]
where $S_{k}$ denotes the set of ordered $k$-tuples $1\le
i_{1}<\dots<i_{k}\le 8$ of digits from the set $\{1,\dots,8\}$.

Repeat the following steps on all elements $[\mt_{3},\mt,\mt_{5}]\in
\mathcal M$
\begin{itemize}
	\itemsep=2mm
	\item [A1:] If $\#(Q_{1}\cap Q_{2})=3$ for $Q_{1},Q_{2}\in\mt$ then replace
	the entry $[\mt_{3},\mt,\mt_{5}]$ with the following two entries $[\mt_{3}\cup\{Q_{1}\cap
	Q_{2}\},\mt,\mt_{5}]$ and $[\mt_{3},\mt,\mt_{5}\cup\{Q_{1}\cup
	Q_{2}\}]$.
	\item [A2:] If $\#(Q_{1}\cap Q_{2})=2$ for $Q_{1},Q_{2}\in\mathcal T_{3}$
	then remove the element $[\mt_{3},\mt,\mt_{5}]$ from  $\mathcal M$.
	\item [A2:] for each $Q\in\mt_{3}$ replace $[\mt_{3},\mt,\mt_{5}]$ with
	$[\mt_{3},\mt\cup\{Q'\in S_{4}:Q'\supset Q\},\mt_{5}]$.
	\item [A3:] for each $Q\in\mt_{5}$ replace $[\mt_{3},\mt,\mt_{5}]$ with
	$[\mt_{3},\mt\cup\{Q'\in S_{4}:Q'\subset Q\},\mt_{5}]$.
	\item [A4:] If $\mt^{g}<\mt$ for a permutation $g\in\Sigma_{8}$ then replace
	$[\mt_{3},\mt,\mt_{5}]$ with  $[\mt_{3}^{g},\mt^{g},\mt_{5}^{g}]$.
\end{itemize}
until the list $\mathcal M$ stabilizes. Then define $\mathcal
L_{k}:=\mathcal M$.

Running this algorithm in Magma we verify that
\[\mathcal L_{15}=\emptyset \quad\text{ and }\quad
\#\left( \bigcup_{k=0}^{14}\mathcal L_{k}\right)=515.\]
\end{proof}

\begin{prop}\label{prop:it455}
	There are 455 minimal incidence tables of arrangements of
	eight planes satisfying (definition~\ref{def:oa})
\end{prop}
\begin{proof}
	An element of the incidence table of an octic arrangement $D=P_{1}\cup\dots\cup P_{8}$ corresponds bijectively to a maximal minor of the $8\times4$ matrix of coefficients of equations of planes $P_{1},\dots,P_{8}$. Hence, arrangements with a given incidence table are parametrized by a quasi-projective subset of $(\mathbb P^{3})^{8}$. A list of quadruples from $S_{4}$ is an incidence table when the quasi-projective variety, assigned to that list in an obvious way, is non-empty.
	However, this approach involves polynomials in 32 variables and ignores the action of the automorphisms group.

	To remedy these issues, we fix a quintuple $Q=\{i_{1},\dots,i_{5}\}\subset\{1,\dots,8\}$ such that $Q_{1}\not\subset
	Q$ for any quadruple $Q_{1}\in\mathcal T$.
	If $\mathcal T$ is an incidence table of an octic arrangement $D$, then this condition implies that the planes $P_{i_{1}},\dots,P_{i_{5}}$ are in general position.
	Denote by $1\le j_{1}<j_{2}<j_{3}\le 8$ the
	elements of the set $\{1,\dots,8\}\setminus\{i_{1},\dots,i_{5}\}$, and by $F_{i}$ the equation of $P_{i}$ (linear in coordinates $x,y,z,t$ on $\PP^{3}$).

	Up to an automorphism of $\PP^{3}$  we can assume that
	\begin{eqnarray*}
		&&F_{i_{1}}=x,\;\; F_{i_{2}}=y,\;\; F_{i_{3}}=z,\;\; F_{i_{4}}=t,\;\;
		F_{i_{5}}=x+y+z+t,\\
		&&F_{j_{k}}=A_{k}x+B_{k}y+C_{k}z+P_{k}t,\;\; k=1,2,3.
	\end{eqnarray*}

	Let $M$ be the $8\times4$ matrix of coefficients of the linear forms
	$F_{1},\dots,F_{8}$, denote by $I$ the ideal generated by degree four
	minors of $M$ corresponding to quadruples from $\mathcal T$ and by $J$
	the principal ideal defined by the product of remaining non-zero minors. A list
	of quadruples $\mathcal T$ is an incidence table iff the
	quasi-projective set $V_{\mathcal T}:=V(I)\setminus V(J)\subset \PP^{3}\times\PP^{3}\times \PP^{3} $ is non-empty.
	In that case octic arrangements with the minimal incidence table $\mathcal T$ up to  projective equivalence correspond bijectively to points of the quasi-projective set  $V_{\mathcal T}$.

	A quintuple $Q$ such that $Q$ does not contain any element of $\mathcal T$ exists for any $\mathcal T\in\mathcal L$ with the following unique exception
	\[
	[1234, 1256, 1278, 1357, 1368, 1458, 1467, 2358, 2367, 2457, 2468, 3456, 3478, 5678]
	\]
	We shall prove that $\mathcal T$ is not an incidence table of any octic arrangement, assume to the contrary, that $\mathcal T $ is an incidence table of an octic arrangement $D=P_{1}\cup\dots\cup P_{8}$.
	Observe, that the only incidences among planes $P_{1},\dots,P_{6}$ are
	$1234$, $1256$ and $3456$, so the arrangements of six planes $P_{1}\cup\dots\cup P_{6}$ is projectively equivalent to the union of six planes containing faces of a cube.

	Consequently, we can change  coordinates in $\PP^{3}$ to get
	\[F_{1}=x,\;\; F_{2}=x-t,\;\; F_{3}=y,\;\; F_{4}=y-t,\;\; F_{5}=z,\;\;
	F_{6}=t,\]
	where $F_{i}$ is an equation of the plane $P_{i}$. If we write the
	remaining equations as
	\[F_{7}=A_{1}x+B_{1}y+C_{1}z+P_{1}t,\;\;
	F_{8}=A_{2}x+B_{2}y+C_{2}z+P_{2}t\]
	then minors corresponding to incidences equal
	\begin{eqnarray*}
	 0, 0, -B_{1}C_{2} + C_{1}B_{2}, P_{1}, C_{2} + P_{2}, B_{2} + P_{2}, B_{1} + C_{1} + P_{1}, A_{2} + P_{2}, A_{1} + C_{1} + P_{1}, \\
	 A_{1} + B_{1} + P_{1}, A_{2} + B_{2} + C_{2} + P_{2}, 0, A_{1}C_{2} - C_{1}A_{2}, -A_{1}B_{2} + B_{1}A_{2}.
	\end{eqnarray*}
	We get a system of equations with only trivial solution
	\[A_{1}=B_{1}=C_{1}=P_{1}=A_{2}=B_{2}=C_{2}=P_{2}=0.\]
The ideal $I$ generated by these polynomials
equal
\[\langle2A_{1},A_{1}+B_{1},A_{1}+C_{1},P_{1},2A_{2},A_{2}+B_{2},A_{2}+C_{2},A_{2}+P_{2}\rangle = \langle0\rangle\]
and consequently for this $\mathcal T$ we get $V_{\mathcal T}=\emptyset$.

Using MAGMA we compute the quasi-projective set $V_{\mathcal T}$ and verify that for 455 elements of $\mathcal L$ this set is non-empty and $\mathcal T$ is an incidence table of an octic arrangement.

\end{proof}
For 451 minimal incidence tables the saturated ideal $\operatorname{sat}(I,J)$
is radical and hence the set $V(I)\setminus V(J)$ is
irreducible. Manipulating equations we verified that in fact
$V(I)\setminus V(J)$ is rational and found a rational
parametrization. In the remaining four cases the set $V(I)\setminus
V(J)$ decomposes over a quadratic field into two Galois conjugate
components. These four incidence tables were denoted by E1, E2, E3 and E4
in \cite{CK-C}, here we shall denote them by 451, 452, 453 and 454
respectively.

\begin{prop}[\mbox{\cite[Thm.~7.3]{CK-C}}]
	Octic arrangement No. 451 is unique up-to a projective
	transformation, it is given by an arrangement of eight planes with
	coefficients in $\QQ[\sqrt{-3}]$ and is isomorphic over $\QQ[i]$ to
	its Galois conjugate.

	Octic arrangement No. 452 is unique up-to a projective
	transformation, it is given by an arrangement of eight planes with
	coefficients in $\QQ[\sqrt{-3}]$ and is isomorphic over $\QQ$ to
	its Galois conjugate.

	There are two non-isomorphic octic arrangements No. 453, they are
	given by two Galois conjugate arrangement of eight planes with
	coefficients in $\QQ[\sqrt{5}]$.

	There exists a family  over $\PP^{1}\setminus\{0, \infty, 1,
	\frac12(\sqrt{-3}+1)\}$ that contains all octic arrangements No. 454
	up to isomorphism defined over $\QQ[\sqrt{-3}]$. Every element of this
	family is isomorphic to an element of the Galois conjugate family.
\end{prop}

C. Meyer in his PhD thesis run an extensive computer
search of octic arrangements given by eight planes with integer
coefficients and compiled table with numerical data of 450
non-equivalent examples (\cite[Appendix A]{Meyer}).
We determined minimal incidence tables of these 450 octic arrangements
using sample equations that we get from C. Meyer.
Finally, the octic arrangement forgotten by C. Meyer we call 455.

	\def\arraystretch{1.4}
\begin{table}[h]
	\centering
	\begin{tabular}{|r||r|r|r|r|r|r|r|r|r|r|}
		\hline
		\rule[-1ex]{0pt}{2.5ex}  no.& $p_{3}$ &$p^{0}_{4}$  &$p^{1}_{4}$  & $p_{5}^{0}$ & $p_{5}^{1}$ & $p_{5}^{2}$ & $l_{3}$ & $h^{1,1}$ & $h^{1,2}$ & $\chi$ \\
		\hline
		\rule[-1ex]{0pt}{2.5ex} 451 & 13 & 6 & 3 & 0 & 1 & 0 & 1 & 46 & 0 & 92 \\
		\hline
		\rule[-1ex]{0pt}{2.5ex} 452 & 20 & 9 & 0 & 0 & 0 & 0 & 0 & 38 & 0 & 76 \\
		\hline
		\rule[-1ex]{0pt}{2.5ex} 453 & 20 & 9 & 0 & 0 & 0 & 0 & 0 & 38 & 0 & 76 \\
		\hline
		\rule[-1ex]{0pt}{2.5ex} 454 & 24 & 8 & 0 & 0 & 0 & 0 & 0 & 37 & 1 & 72 \\
		\hline
		\rule[-1ex]{0pt}{2.5ex} 455 & 28 & 7 & 0 & 0 & 0 & 0 & 0 & 36 & 2 & 68 \\
		\hline
	\end{tabular}
	\caption{\rule[-4mm]{0cm}{1cm}Forgotten arrangements}
\end{table}

C. Meyer proved that there are ten types of arrangements of six planes (\cite[Lem.~4.4]{Meyer}) and then used the ordered list of types of subarrangements of six planes to distinguish octic arrangements. The types of subarrangements of six planes for an octic arrangement can be computed  from the incidence tables, we have checked that (after appropriate renumbering of eight planes) the ordered sequence of types of subarrangements of six planes for the arrangement No. 308 and the forgotten arrangement No. 455 coincide.

In section~\ref{sec:EqsData} we collect equations of 455 families of octic arrangements, to keep the length of this paper reasonable we decided to not include further data. The extra information - (minimal) incidence table, invariant permutations, singularities - can be easily recovered from the equations using the Magma scripts that we also included in section \ref{sec:EqsData}.

\subsection*{Arrangements in a positive characteristic}
In the proof of Prop.~\ref{prop:it455} we have considered
$\mathcal T=	[1234, 1256, 1278, 1357, 1368, 1458, 1467, 2358, 2367, 2457, 2468, 3456, 3478, 5678]$ for which the ideal $I$ in the ring $\mathbb C[A_{1},B_{1},C_{1},P_{1},A_{2},B_{2},C_{2},P_{2}]$ defined by incidence minors
is trivial. However, the ideal $I_{\ZZ}$ defined by the same minors in the ring
of polynomials with integral coefficients
$\ZZ[A_{1},B_{1},C_{1},P_{1},A_{2},B_{2},C_{2},P_{2}]$ equals
\[I_{\ZZ}=\langle2A_{1},A_{1}+B_{1},A_{1}+C_{1},P_{1},2A_{2},A_{2}+B_{2},A_{2}+C_{2},A_{2}+P_{2}\rangle \not= \langle0\rangle\]
is not trivial. More precisely, the quotient ring equals
\[\ZZ[A_{1},B_{1},C_{1},P_{1},A_{2},B_{2},C_{2},P_{2}]/I_{\ZZ}\cong\ZZ[A_{1},A_{2}]/(2A_{1},2A_{2})\cong \mathbb F_{2}[A_{1},A_{2}]\cong \mathbb F_{2}[A_{1}]\otimes_{\ZZ}\mathbb F_{2}[A_{2}].\]
Conseqently, the ideal $I_{\ZZ}$ defines a single point in $\PP_{\mathbb F_{2}}^{3}\times\PP_{\mathbb F_{2}}^{3}$.

As a consequence, $\mathcal T$ is an incidence table of an octic arrangement in characteristic 2
\[x(x-t)y(y-t)z(z-t)(x+y+z)(x+y+z+t)=0,\]
but it is not an incidence table of an octic arrangement in any other characteristic.

We carried out similar computations for 60 elements in $\mathcal L$ for which $\mathcal T$ is not an incidence table for an octic arrangement in $\PP_{\CC}^{3}$ and found 15 incidence tables of octic arrangements in characteristic 2 and one incidence table of an octic arrangement in characteristic 3.

Octic arrangements in characteristic 2 that do not lift to characteristic 0
\begin{eqnarray*}
	&&xyz(x+y+z)t(A_{3}x+A_{2}y+A_{3}t)(A_{0}A_{3}x+A_{1}A_{2}z+A_{1}A_{3}t)\times\\
	&&\rule{10cm}{0cm}(A_{0}x+A_{2}y+A_{2}z+A_{3}t)=0\\
	&&xyz(x+y+z)t(x+y+t)(x+z+t)(A_{0}x+A_{1}y+A_{2}z+A_{3}t)=0\\
	&&xyz(x+y+z)(A_{0}x+A_{1}y+A_{2}z)t(x+y+t)(x+z+t)=0\\
	&&xy(x+y)zt(x+z+t)(A_{0}y+A_{1}z-A_{0}t)(A_{0}x+A_{0}y+A_{1}z+A_{0}t)=0\\
	&&xyz(x+y+z)t(A_{2}x+A_{0}y+A_{2}t)(A_{0}A_{2}x+A_{0}A_{1}z+A_{1}A_{2}t)(A_{0}y+A_{0}z+A_{2}t)=0\\
&&xyz(x+y+z)t(A_{2}x+A_{0}y+A_{2}t)(A_{0}x+A_{1}z+A_{0}t)(A_{0}^{2}y+A_{1}A_{2}z+A_{0}A_{2}t)=0\\
	&&xy(x+z+t)(A_{0}x+A_{0}y+A_{2}z+A_{2}t)z(A_{0}x+A_{1}y+A_{2}z)t(A_{0}x+A_{1}y+A_{2}t)=0\\
	&&xy(x+y)z(A_{0}x+A_{1}y+A_{0}z)t(A_{0}x+A_{1}y+A_{0}t)(x+z+t)=0\\
	&&xy(x+y)z(x+z)t(y-z+t)(A_{0}x+A_{0}y+A_{0}z+A_{1}t)=0\\
	&&xyz(x+y+z)(A_{0}x+(A_{0}-A_{1})y+A_{1}z)t(x+y+t)(x+z+t)=0\\
	&&xy(x+y)zt(x+z+t)(A_{0}y+(A_{1}-A_{0}z)-A_{0}z)(A_{0}x+A_{0}y+A_{1}z+A_{0}t)=0\\
	&&xy(x+y)z(x+z)t(x+t)(y+z+t)=0\\
	&&xyz(x+y+z)t(A_{1}x+A_{0}y+A_{1}t)(A_{1}x+A_{0}z+A_{1}t)(A_{0}y+A_{0}z+A_{1}t)=0\\
	&&xyz(x+y+z)t(A_{1}x+A_{0}y+A_{1}t)((A_{1}-A_{0})x+(A_{1}-A_{0})z-A_{1}t)\times\\
	&&\rule{8.6cm}{0cm}(A_{0}x+A_{0}y+(A_{0}-A_{1}z)+A_{1}t)=0\\
	&&xyz(x+y+z)t(x+y+t)(x+z+t)(y+z+t)=0
\end{eqnarray*}
The geometry of double octics in characteristic 2 substantially differs from the case of characteristic 0. We discuss in more detail the simplest example given by the last equation from the above list. In characteristic zero this equation defines an octic arrangement with the minimal incidence table
\begin{eqnarray*}
	&&[[1,2,3,4], [1,2,5,6], [1,2,7,8], [1,3,5,7], [1,3,6,8], [1,4,5,8], [2,3,5,8], [2,3,6,7],\\ &&\phantom{[}[2,4,5,7], [3,4,5,6]]
\end{eqnarray*}

Singularities of this octic arrangement consist of 10 fourfold points of type $p_{4}^{0}$, 16 (isolated) triple points and 28 double lines. Singular locus of the double octic
\[\{u^{2} = xyz(x+y+z)t(x+y+t)(x+z+t)(y+z+t)\}\subset \PP_{\CC}(1,1,1,1,4)
\]
coincide (in characteristic different from 2) with singularities of the octic arrangement. The resolution of singularities (\cite{CSz}) can be performed by first resolving the octic arrangement (blow-up of all fourfold points followed by blow--ups of double lines) and the taking the double covering.

The same equation in characteristic 2 defines an octic arrangement with the minimal incidence table
\begin{eqnarray*}
	&&[[1, 2, 3, 4], [1, 2, 5, 6],[1, 2, 7, 8], [1, 3, 5, 7], [1, 3, 6, 8], [1, 4, 5, 8], [1, 4, 6, 7], [2, 3, 5, 8],\\
	&&\phantom{[}[2, 3, 6, 7], [2, 4, 5, 7], [2, 4, 6, 8], [3, 4, 5, 6], [3, 4, 7, 8], [5, 6, 7, 8]]
\end{eqnarray*}

Singularities of the octic arrangement consists of 14 fourfold points (of type $p_{4}^{0}$) and 28 double lines. However, this time singular locus of the double octic
\[\{u^{2} = xyz(x+y+z)t(x+y+t)(x+z+t)(y+z+t)\}\subset \PP_{\mathbb F_{2}}(1,1,1,1,4)
\]
is larger, it contains additional 7 double lines that are contained in the branch locus. More precisely, the 14 fourfold points fall in 7 pairs of ``opposite'' points that do not belong to a common arrangement plane (i.e. given by two disjoint quadruples from the incidence table). A line joining a pair of opposite fourfold point is a double line.

For instance the disjoint quadruples $[1,2,3,4]$ and $[5,6,7,8]$ define the fourfold points
\begin{eqnarray*}
&&(0,0,0,1):\quad x=y=z=x+y+z=0,\\
&&(1,1,1,0):\quad t=x+y+t=x+z+t=y+z+t=0.
\end{eqnarray*}
The double octic  $Y_{2}\subset\PP_{\mathbb F_{2}}$ is singular along the line $\{x=y=z\}\subset \PP_{\mathbb F_{2}}^{3}$ joining these points.

After blowing-up of fourfold points double lines become pair-wise disjoint, blowing-up all double lines we get a (nonsingular) Calabi--Yau manifold $X_{2}$.
This Calabi--Yau manifold has very rich geometry. It has a large group of birational isomorphisms induced by projective transformations of $\PP^{3}$ (cf. Section~\ref{sec:projtransf}) and fibrations induced by  double lines, fourfold points, ``opposite'' pairs of fourfold points, and pairs of skew double lines  (see Section \ref{sec:fibr}).

A double octic Calabi--Yau manifold in characteristic 2 is a purely inseparable double cover of a rational threefold (blow-ups of $\PP^{3}$ at points and rational curves). Consequently, double octics in characteristic 2 have completely different properties then in other characteristics. Most notably, the resolution of singularities introduced in \cite{CScvS} does not work in characteristic~2. Fibers of above-mentioned fibrations on $X_{2}$ are also purely inseparable coverings of $\PP^{1}$ or blow-ups of $\PP^{2}$. The Calabi--Yau manifold $X_{2}$ has several fibrations by quasi-elliptic curves, it is 2 purely inseparable double covering of the fiber product if two rational, quasi-elliptic surfaces (cf. \cite{Hirokado2}).

From the above discussion it follows that a resolution of singularities of a double octic in characteristic 2 cannot be realized, in general,  by a sequence of admissible blow-ups. Consequently, known formulas for the Euler characteristic, Hodge and Betti numbers and the description of the infinitesimal deformation space etc. may not work in characteristic 2.

The following sequence
\begin{eqnarray*}
	&&[[1,2,3,4],
	[1,2,5,6],
	[1,2,7,8],
	[1,3,5,7],
	[1,3,6,8],
	[1,4,5,8],
	[2,3,5,8],
	[2,3,6,7],\\
	&&\phantom{[}[2,4,5,7],
	[3,4,5,6],
	[4,6,7,8]]
\end{eqnarray*}
is the only minimal incidence table of an octic arrangement in characteristic 3, but not of an arrangement in characteristic 0 (and in fact in any characteristic $\not=3$). The unique, up to an isomorphism of the projective space $\PP_{\mathbb F_{3}}^{3}$, octic arrangement with this incidence table is given by the same equation
\[xyz(x+y+z)t(x+y+t)(x+z+t)(y+z+t)=0.\]
Contrary to the situation in characteristic 2, the double octic
\[\{u^{2} = xyz(x+y+z)t(x+y+t)(x+z+t)(y+z+t)\}\subset \PP_{\mathbb F_{3}}(1,1,1,1,4)\]
has a resolution of singularities $X_{3}$ by a sequence of admissible blow-ups, which is a smooth Calabi--Yau threefold in characteristic 3 non-liftable to characteristic 0 (\cite{CvS3}).
From non-liftability of $X_{3}$ it follows (\cite{Lam}) that the third Betti number of $X_{3}$ vanishes.

Partial resolutions of singularities of the two above double octics may be considered as fibers of the same  $\operatorname{Spec}(\mathbb Z)$-scheme. Denote by $\mathcal Y$ the subscheme of weighted projective space $\PP_{\operatorname{Spec}\ZZ}(1,1,1,1,4)$ defined by the same equation
\[\mathcal Y=\{u^{2} = xyz(x+y+z)t(x+y+t)(x+z+t)(y+z+t)\}\subset \PP_{{\operatorname{Spec}\ZZ}}(1,1,1,1,4)\]
Generic fiber of this scheme is a projective variety over $\QQ$ with 28 double lines, 10 fourfold points (of type $p_{4}^{0}$) and 16 isolated triple points.
Let $\mathcal X$ be a partial resolution of $\mathcal Y$ obtained by blow-ups of closures of fourfold points in the generic fiber followed by blow-ups of closures of double lines. This partial resolution is a resolution of singularities of the generic fiber ``propagated'' to the closed fibers. Consequently, the generic fiber $Y_{0}$ of $\mathcal Y$ is a smooth Calabi--Yau manifold over $\QQ$. The fibers $Y_{p}$, for $p\in \operatorname{Spec}(\ZZ)$, $p>3$, are also smooth. In fact the scheme $\mathcal X|\left((\operatorname{Spec}(\ZZ)\setminus\{2,3\}\right
)$ is a smooth $\operatorname{Spec}(\ZZ)\setminus\{2,3\}$-scheme.
The fibers $X_{2}$ and $X_{3}$ are singular. By \cite{CvS3} $X_{3}$ is a nodal variety with two nodes admitting a projective small resolution. Singular locus of the variety $X_{2}$ consists of seven double lines, four of them have a pinch-point. More precise description of geometric and arithmetic of singular fibers can be read off from a semistable degeneration. An algorithm for computing a semistable  degeneration of a family of double octic Calabi--Yau threefolds was developed in \cite{Oczko}.

\section{Projective transformations of double octics}
\label{sec:projtransf}

Let $D=D_{1}\cup\dots\cup D_{8}$ and $D'=D_{1}'\cup\dots\cup D_{8}'$
be octic arrangements with defining equations $F(x:y:z:t)$
and $F'(x:y:z:t)$ respectively. The double coverings of
$\PP$ branched along $D$ and $D'$ can be given as hypersurfaces in the
weighted projective space $\mathbb P(1,1,1,1,4)$
\[Y:=\{(x:y:z:t:u)\in \mathbb P(1,1,1,1,4)\colon
  u^{2}=F(x:y:z:t)\} \subset\mathbb P(1,1,1,1,4),\]
and
\[Y':=\{(x:y:z:t:u)\in \mathbb P(1,1,1,1,4)\colon
  u^{2}=F'(x:y:z:t)\} \subset\mathbb P(1,1,1,1,4).\]
Any projective transformation $\Phi:\PP^3\lra\PP^3$
such that $\Phi(D')=D$ induces a projective transformation $\tilde
\Phi:Y'\lra Y$  by
\[\tilde\Phi(x:y:z:t:u)=(\Phi(x:y:z:t):cu),\]
where $c\in\CC^{*}$ is a non-zero complex number satisfying
$F\circ\Phi=c^{2}F'$ (the map $\tilde\Phi$ is determined up-to the sheet
inverting automorphism of $Y$ or $Y'$).
By the universal property of the blow-up the map $\tilde\Phi$ lifts to an automorphism $\Psi:X'\lra X$ of Calabi--Yau resolutions of
singularities $X$ and $X'$ of varieties $Y$ and $Y'$.

As $\Phi(D')=D$ there exists a permutation $\sigma\in\Sigma_{8}$ such
that $\Phi (P_{i}')=P_{\sigma(i)}$. Planes
$P_{i_{1}}',\dots,P_{i_{4}}'$ intersect iff planes
$P_{\sigma(i_{1})},\dots,P_{\sigma(i_{4})}$ intersect. Equivalently,
the application of the permutation $\sigma$ to the incidence table of $D'$
coincide with the incidence table of $D$. In particular, the
arrangements $D$ and $D'$ have the same minimal incidence table $\mathcal
T$ and the permutation $\sigma$ preserves $\mt$.

Conversely, an invariant permutation $\sigma$ of the minimal incidence table
determines a projective transformation $\phi:\PP^3\lra\PP^3$ such that
$\Phi(P_{i}')=P_{\sigma(i)}$.
\begin{prop}\label{p:uniq}
	Assume that $ D=P_1\cup\dots\cup P_8 $, $ D'=P_1'\cup\dots\cup P_8' $ are octic arrangements. For any permutation $\sigma\in\Sigma_8$ there exists at most one projective automorphism $\Phi\in\operatorname{\mathbb PGL}(3)$ such that $\Phi(P_i')=P_{\sigma(i)}$ for $i=1,\dots,8$.
\end{prop}
\begin{proof}
Denote by $F_i(X,Y,Z,T)=a_{i,1}X+a_{i,2}Y+a_{i,3}Z+a_{i,4}T$ (resp. $F_i'(x,y,z,t)=a_{i,1}'x+a_{i,2}'y+a_{i,3}'z+a_{i,4}t'$) a linear form defining $P_i$ (resp. $P_i'$), $ i=1,\dots,8 $.
Renumbering the planes $ P_i $ and $ P_i' $ and changing the coordinates in $\PP^3$ we can assume that $\sigma=\operatorname{Id}$ and
\[ F_1=X, F_2=Y, F_3=Z, F_4=T. \]
Since $ \Phi(P_{i}')=P_{i} $ there exists a non-zero constant $\lambda_{i}$ such that $F_{i}\circ \Phi=\lambda_{i}F_{i}'$. Linear forms $F_{1},F_{2},F_{3},F_{4}$ are linearly independent, which implies that the forms $F_{1}',F_{2}',F_{3}',F_{4}'$ are also linearly independent.
Moreover,
\[ \Phi=(\lambda_{1}F_{1}',\dots,\lambda_{4}F_{4}') \]
and
\begin{equation}
\begin{aligned}\label{leq}
a_{5,1}\lambda_{1}F_{1}'+ a_{5,2}\lambda_{2}F_{2}'+ a_{5,3}\lambda_{3}F_{3}'+ a_{5,4}\lambda_{4}F_{4}' &= \lambda_{5}F_{5}'\\
a_{6,1}\lambda_{1}F_{1}'+ a_{6,2}\lambda_{2}F_{2}'+ a_{6,3}\lambda_{3}F_{3}'+ a_{6,4}\lambda_{4}F_{4}' &= \lambda_{6}F_{6}'\\
a_{7,1}\lambda_{1}F_{1}'+ a_{7,2}\lambda_{2}F_{2}'+ a_{7,3}\lambda_{3}F_{3}'+ a_{7,4}\lambda_{4}F_{4}' &= \lambda_{7}F_{7}'\\
a_{8,1}\lambda_{1}F_{1}'+ a_{8,2}\lambda_{2}F_{2}'+ a_{8,3}\lambda_{3}F_{3}'+ a_{8,4}\lambda_{4}F_{4}' &= \lambda_{8}F_{8}'\\\end{aligned}
\end{equation}

Our goal is to show, that the numbers $\lambda_{1},\dots,\lambda_{8}$ are determined uniquely by the above equations, up-to multiplication by a non-zero constant. If it is not the case there exists a non-zero solution $\lambda_{1},\dots,\lambda_{8}$ such that one of the numbers $\lambda_{i}$ equals zero.

  We shall consider case-by case the number $k:=\#\{i\in\{1,\dots,4\}: \lambda_{i}=0\}$. Without loss of generality we can assume that $\lambda_{i}=0$ for $i=1,\dots,k$.
  As the numbers $\lambda_{1},\dots,\lambda_{4}$ determine uniquely $\lambda_{5},\dots,\lambda_{8}$ the case $k=4$ is impossible.
  The linear forms $F_{i}'$ are non-proportional hence in the case $k=3$ we get $a_{5,4}=a_{6,4}=a_{7,4}=a_{8,4}=0$ and consequently the planes $P_{1}$, $P_{2}$, $P_{3}$, $P_{5}$, $P_{6}$, $P_{7}$ and $P_{8}$ contain the point $(0,0,0,1)$ contradicting the definition of an octic arrangement.

  If $k=2$ the left-hand sides of \eqref{leq} are linear combinations of $F_{3}'$ and $F_{4}'$, as at most one of the planes $P_{5}', P_{6}', P_{7}', P_{8}'$
contains the line $P_{3}'\cap P_{4}'$ at least three of the numbers $\lambda_{5}, \lambda_{6}, \lambda_{7}, \lambda_{8} $ are equal zero, assume $\lambda_{5} = \lambda_{6} = \lambda_{7}=0$. Then $a_{5,3} = a_{5,4} = a_{6,3} = a_{6,4} = a_{7,3} = a_{7,4}=0$ and the planes $P_{1}, P_{2}, P_{5}, P_{6}, P_{7}$ contain the line $X=Y=0$, which contradicts assumptions.

Since $F_{i}\not=a_{i,1}X$ for $i=5,\dots,8$ and the linear forms $F_{1}'$, $F_{2}'$, $F_{3}'$, $F_{4}'$ are linearly independent
in the case of $k=1$  we get $\lambda_{i}\not=0$ for $i=5,\dots,8$, consequently the planes $F_{2}', F_{3}', F_{4}', F_{5}', F_{6}', F_{7}', F_{8}'$ intersect, which is impossible.

Finally, $k=0$ implies that $\lambda_{j}=0$ for some $i\in\{5,\dots,8\}$, without loss of generality we can assume that $\lambda_{5}=0$. As $\lambda_{i}\not=0$ for $i\in\{1,\dots,4\}$ and $F_{5}\not=0$ first equation in \eqref{leq} implies that $F_{1}',\dots,F_{4}'$ are linearly dependent, which contradicts our assumptions.
\end{proof}

Using MAGMA code presented in the Appendix B we computed groups of invariant permutations of all octic arrangements listed in Appendix A. Using standard MAGMA functions (\texttt{IdentifyGroup}, \texttt{GroupName}, \texttt{FewGenerators}) we identified computed groups in the table of small groups, found their name and small set of permutations generating the group.
Altogether there are 34 isomorphisms types of groups of symmetric permutations

\noindent
 $\langle\operatorname{Id}\rangle$,
 $C_2$,
 $C_3$,
 $C_4$,
 $C_2^2$,
 $S_3$,
 $C_6$,
 $C_8$,
 $D_4$,
 $C_2^3$,
 $D_6$,
 $D_4\times C_2$,
 $S_4$,
 $S_3\times C_2^2$,
 $C_2^2\wr C_2$,
 $C_8\rtimes C_2^2$,
 $C_2^2\times D_4$,
 $S_3^2$,
 $S_3\times D_4$,
 $C_2\times S_4$,
 $C_2\wr C_2^2$,
 $S_3\wr C_2$,
 $C_2\times S_3^2$,
 $C_2^2\times S_4$,
 $C_2\wr D_4$,
 $S_3\times S_4$,
 $C_\times S_3\wr C_2$,
 $C_2^2\rtimes S_4\rtimes C_2$,
 $C_2^3\rtimes S_4$,
 $C_2\wr A_4.C_2$,
 $A_4^2.C_2^2$,
 $S_3\times S_5$,
 $A_4^2.D_4$,
 $S_8$.

For an invariant permutation $\sigma$ of an octic arrangement $D$ we can reverse the procedure in the proof of prop.~\ref{p:uniq} to compute the transformation $\Phi$ of the projective space $\PP^{3}$ and the transformation $\phi$ of the parameter space using  elementary linear algebra.

\begin{exmp}\label{ex280}
	The	arrangement no. 280 has a unique non-identity invariant permutation $(1,2)(3,5)(4,6)(7,8)$. We are looking for projective transformations $\Phi:\PP^{3}\lra \PP^{3}$ and $\phi:\PP^{2}\lra\PP^{2}$ such that
	\[F_{\sigma(i)}^{A}\circ \Phi=\lambda_{i}F_{i}^{\phi(A)},\]
	where $(F_{1}^{A},\dots,F_{8}^{A})$ are the equations (properly ordered) of planes defining the arrangement no. 280 corresponding to parameter $A$.
	Since
	\[(F_{1}^{A},\dots,F_{8}^{A})=\left(x, y, z, A_{0}x + A_{1}y - A_{0}z, z + t, A_{0}x + A_{2}y + A_{2}z + A_{2}t, x + t, y + t\right)\]
	we get the following system of equations (to simplify notation we introduce $(B_{0},B_{1},B_{2}) = \phi(A_{0},A_{1},A_{2})$)
	\begin{eqnarray*}
		\Phi_{2}&=&\lambda_{1}x\\
		\Phi_{1}&=&\lambda_{2}y\\
		\Phi_{3}+\Phi_{4}&=&\lambda_{3}z\\
		A_{0}\Phi_{1}+A_{2}\Phi_{2}+A_{2}\Phi_{3}+A_{2}\Phi_{4} &=& \lambda_{4}(B_{0}x+B_{1}y-B_{0}z)\\
		\Phi_{3}&=&\lambda_{5}(z+t)\\
		A_{0}\Phi_{1}+A_{1}\Phi_{2}-A_{0}\Phi_{3} &=& \lambda_{6}(B_{0}x+B_{2}y+B_{2}z+B_{2}t)\\
		\Phi_{2}+\Phi_{4}&=&\lambda_{7}(x+t)\\
		\Phi_{1}+\Phi_{4}&=&\lambda_{8}(y+t)
	\end{eqnarray*}
Solving this system of equations we get
$\lambda_{2}=\lambda_{1}$, $\lambda_{3}=-\lambda_{1}$, $\lambda_{5}=-\lambda_{1}$,  $\lambda_{7}=\lambda_{1}$, $\lambda_{8}=\lambda_{1}$. Normalizing $\lambda_{1}=1$ we conclude
\[\Phi(x,y,z,t)=(y,x,-z-t,t)\]
and \[\phi(A_{0},A_{1},A_{2})=(A_{1}A_{2},A_{0}A_{1},A_{0}A_{2}).\]
\end{exmp}
Using a Magma implementation of a similar algorithm we have computed over 3600 transformations: we have completely omitted the two most generic Arrangements No. 449 and 450 (depending on 8 and 9 parameters and with groups of invariant permutations isomorphic to $S_{4}^{2}$ and $S_{8}$ respectively). For 23 Arrangements with large groups of Symmetries we have considered only small sets of generators.

\section{Geometrically distinguished subfamilies and elements}
\label{sec:dist}
Let $G$ be a finite group acting fiber-wise on a family $\pi:\mathcal X\lra B$ of Calabi--Yau threefolds: every element  $\Phi:\mathcal X\lra\mathcal X$ of $G$ is an automorphism of $\mathcal X$  satisfying $\pi\circ\Phi=\pi$. Then there is an automorphism $\phi:B\lra B$ such that for any element $b\in B$ the map $\Phi$ restricts to $\Phi_{b}:\mathcal X_{b}\lra \mathcal X_{\phi(b)}$.
Automorphisms $\phi:B\lra B$ form a group $g$ acting on $B$.
We call the restriction $\mathcal X|B_{1}$ of the family $\mathcal X$
to an irreducible component $B_{1}$ of the fixed locus $\operatorname{Fix}(g)$ of $g$ a \emph{geometrically distinguished subfamily} of $\mathcal X$. If $B_{1}=\{b_{0}\}$ is an isolated point, we call $\mathcal X_{b_{0}}$ a \emph{distinguished element} of the family $\mathcal X$.

In this situation restriction of the local system $R^{3}\pi_{*}\mathbb C_{\mathcal Y}$ to $B_{1}$ contains subsystem $\mathcal R$ such that for an element $b\in B_{1}$ the fiber $\mathcal R_{b}$ corresponds to the fixed locus of the action induced by $G$ on $H^{3}(\mathcal X_{b})$. In the case of a geometrically distinguished element $\mathcal X_{b_{0 }}$ we get a submotive in $H^{3}(\mathcal X_{b_{0}})$.

For computation of distinguished subfamilies of a family $\mathcal X$ of double octic Calabi--Yau threefolds we use the following strategy. For each transformation $\Phi$ of $\mathcal X$ (corresponding to an invariant permutation of an octic Arrangement) we compute irreducible components of the fixed locus of corresponding transformation $\phi$ of the parameter space. Denote by $\mathcal F$ the set of all irreducible components of all transformations $\phi$ for the family $\mathcal X$. Then we add to $\mathcal F$ irreducible components of intersections of elements of $\mathcal F$ until the process stabilizes.

Let $\mathcal M$ be the $8\times4$ matrix  of homogeneous polynomials in $\CC[A_{0},\dots,A_{k}]$ defining the family $\mathcal X$.
Let $W$ be the zero-locus of the product of non-zero degree four minors of $\mathcal M$. Then $W$ is the discriminant of the family $\mathcal X$: $(a_{0}:\dots:a_{k})\in W$ iff the eight planes defined by linear form $\mathcal(a_{0}:\dots:a_{k}) (x,y,z,t)^{T}$ have incidence table strictly larger than the incidence table of a generic element of $\mathcal X$.
Consequently, $W$ corresponds to configurations of eight planes that either fail to be an octic arrangement, or define an octic arrangement of type different that for a generic element of $\mathcal X$, hence we remove from $\mathcal F$ varieties which are contained in $W$.

In the context of modularity and Mirror Symmetry of particular interest are geometrically distinguished elements and one parameter families. Since this construction yields a large number of examples we just list a small number of examples demonstrating different situations.

\subsection{Geometrically distinguished elements in families of double octic Calabi--Yau manifolds.}

	C. Meyer made a list (\cite[p. 59]{Meyer}) of double octics $X$ with $h^{1,2}(X)=1$ and modular in the sense that the $L$-series  factors
	\[L(X,s)=L(f,s)L(g,s-1)\]
	into a product $L$-series of a weight four modular form $f$ and a Tate twist of weight two modular form $g$, modularity of all examples from this list except one was proved in \cite{CM}.

	In \cite{CK-C} we have observed that several elements from this list are distinguished in their universal deformation by an action of an involution preserving canonical form.
	D. Burek in \cite{Burek}  constructed crepant resolutions of quotients of these double octics by the corresponding action and consequently obtained  rigid Calabi--Yau threefolds realizing  associated weight four modular forms.
	All modular examples from Meyer's list, except for Arrangements Nos. 8, 154 and 275 are geometrically distinguished.

\begin{exmp}
	Consider an octic arrangement No. 6
	\[D_{A}=\{(x:y:z:t)\in\PP^{3}: tx(x+t)y(y+t)z(z+t)(A_{0}x+A_{1}y+A_{2}z)=0\},\]
	and denote by $X_{A}$ corresponding double octic Calabi--Yau threefold. C. Meyer conjectured (\cite[p. 61]{Meyer}) that the element $X_{(1,1,1)}$ of this family is modular, more precisely that its $L$-series equals
	\[L(X_{(1,1,1)})=L(f,s)L(g,s-1)^{2},\]
	where $f$ is a modular form of weight four and level 96 (96.4.a.b \cite{LMFDB}) and $g$ is a weight two modular form of level 32 (32.2.a.a). Moreover, Meyer checked that the coefficients of $p^{-s}$ of both sides agrees for prime numbers $5\le p\le97$.

	The group of invariant permutations of this octic arrangement is isomorphic to the symmetric group $\Sigma_3$ and contains permutations
	\[\operatorname{Id}, (2,4)(3,5), (2,6)(3,7), (4,6)(5,7), (2,4,6)(3,5,7), (2,6,4)(3,7,5)\]
	that lifts to

	$\displaystyle (24)(35),$\phantom{12} \quad $(x,y,z,t)\mapsto (y,x,z,t)$, \quad $(A_0, A_1, A_2)\mapsto(A_1, A_0, A_2)$

	$\displaystyle (26)(37),$ \phantom{12}\quad $(x,y,z,t)\mapsto (z,y,x,t)$, \quad $(A_0, A_1, A_2)\mapsto(A_2, A_1, A_0)$

	$\displaystyle (46)(57),$ \phantom{12}\quad $(x,y,z,t)\mapsto (x,z,y,t)$, \quad $(A_0, A_1, A_2)\mapsto(A_0, A_2, A_1)$

	$\displaystyle (246)(357),$ \quad $(x,y,z,t)\mapsto (y,z,x,t)$, \quad $(A_0, A_1, A_2)\mapsto(A_2, A_0, A_1)$

	$\displaystyle (264)(375),$ \quad $(x,y,z,t)\mapsto (z,x,y,t)$, \quad $(A_0, A_1, A_2)\mapsto(A_1, A_2, A_0)$

Consequently, the family contains three geometrically distinguished subfamilies $X_{(A_{0},A_{0},A_{1})}$, $X_{(A_{0},A_{1},A_{0})}$, $X_{(A_{1},A_{0},A_{0})}$ that intersects at the distinguished element
$X_{(1,1,1)}$. A. Czarnecki \cite{czarnecki} used the above symmetries to give a proof of modularity of $X_{(1,1,1)}$
(conjectured by Meyer \cite[p. 61]{Meyer}.

\end{exmp}

\subsection{A Calabi--Yau threefold with Complex Multiplication}

The one parameter family $\mathcal X$ of double octics for the Arrangement No. 2 is defined as a resolution of the singular double cover (degree 8 hypersurface in weighted projective space)
\[
\{(x,y,z,t,u)\in\PP(1,1,1,1,4): u^{2} = xy(A_{0} x + A_{1} y)(x + z)z (y + t) t (x + y + z + t)\}.
\]
Its group of invariant permutations is isomorphic to the dihedral group $D_{4}$ of order 8 generated by $(1,2)(4,6)(5,7)$ and $(1,4,6,2)(2,5,7,8)$. Generators of the invariant permutations group induces the following transformations of $\mathcal X$
\begin{eqnarray*}
	&&(1,2)(4,6)(5,7),\ \  \Phi: (x,y,z,t)\mapsto(y,x,t,z),\ \  \phi:(A_{0},A_{1}) \mapsto (A_{1},A_{0})\\
	&&(1,4,6,2)(2,5,7,8),\ \  \Phi: (x,y,z,t)\mapsto(A_{1}y,-A_{0}(y+t), -A_{0}x-A_{1}y, A_{0}(x+y+z+t)),\\  &&\rule{3.9cm}{0cm}\phi:(A_{0},A_{1}) \mapsto (A_{0},A_{1	})
\end{eqnarray*}
There are three non-trivial transformations that acts on the parameter space as identity, and four that acts as $\phi:(A_{0},A_{1}) \mapsto (A_{1},A_{0})$.
Consequently, we get two fixed points $(A_{0},A_{1}) = (1,1)$ and $(A_{0},A_{1}) = (1,-1)$.
At the point $A_{0}=1,A_{1}=-1$ we get the arrangement No. 1, the only distinguished element of family $\mathcal X$ is the double octic $X_{(1:-1)}$ given by
\[
 u^{2} = xy(x - y)(x + z)z (y + t) t (x + y + z + t).
\]
The map
\[
(x,y,z,t,u)\mapsto (y,x,z,t,z,-iu)
\]
defines an automorphism $\Phi$ of $X_{(1:-1)}$ which acts as multiplication by (-1) on $H^{1}\mathcal T_{X_{(1:-1)}}$ and satisfies $\Phi^{*}\omega_{X_{(1:-1)}} = i \omega_{X_{(1:-1)}}$.	The action of Complex Multiplication $\Phi$ decomposes the restriction of the Galois action on $H^{3}_{\text{\'et}}(X,\mathbb Q_{l})$ to the Galois group
$\operatorname{Gal}(\bar{\mathbb Q}/\mathbb Q[i])$ in the case of a prime $p$ split in the field $\mathbb Q[i]$
or $\operatorname{Gal}(\bar{\mathbb Q}/\mathbb Q[\sqrt{p}])$ for an innert prime $p$.
Consequently the Frobenius polynomial

\[F_{p} := \det(1-t\cdot \operatorname{Frob}^{*}_{p}|H^{3}_{\text{\'et}}(X,\mathbb Q_{l}))\]
decomposes over an appropriate quadratic field.

Counting points over $\mathbb F_{p^{2}}$ we derived Frobenius polynomials for primes $p<100$

\long\def\ppp#1||{\parbox{13cm}{\rule[4mm]{0mm}{1pt}$#1$\rule[-2mm]{0mm}{1pt}}}
\bgroup
\tiny
\begin{longtable}[t]{|r|r|r|L|}
	\hline $p$&$a_{p}$&$a_{p^{2}}$&F_{p}\\\hline
	\endhead
	\hline 2 & 0 & -6 & {X}^{4}+3\,{X}^{2}+64\\
	\hline 3 & 0 & -12 &\ppp {X}^{4}+6\,{X}^{2}+729\\
	({X}^{2}+4\,\sqrt {3}X+27)\times ({X}^{2}-4\,\sqrt {3}X+27)|| \\
	\hline 5 & -4 & 276 &\ppp {X}^{4}+4\,{X}^{3}-130\,{X}^{2}+500\,X+15625\\
	({X}^{2}-4\,iX+2\,X-75-100\,i)\times ({X}^{2}+4\,iX+2\,X-75+100\,i)|| \\
	\hline 7 & 0 & -476 &\ppp {X}^{4}+238\,{X}^{2}+117649\\
	({X}^{2}-8\,\sqrt {7}X+343)\times ({X}^{2}+8\,\sqrt {7}X+343)|| \\
	\hline 11 & 0 & -4972 &\ppp {X}^{4}+2486\,{X}^{2}+1771561\\
	({X}^{2}-4\,\sqrt {11}X+1331)\times ({X}^{2}+4\,\sqrt {11}X+1331)|| \\
	\hline 13 & -84 & -1420 &\ppp {X}^{4}+84\,{X}^{3}+4238\,{X}^{2}+184548\,X+4826809\\
	({X}^{2}+28\,iX+42\,X+845+2028\,i)\times ({X}^{2}-28\,iX+42\,X+845-2028\,i)|| \\
	\hline 17 & 36 & 7620 &\ppp {X}^{4}-36\,{X}^{3}-3162\,{X}^{2}-176868\,X+24137569\\
	({X}^{2}-72\,iX-18\,X-4335+2312\,i)\times ({X}^{2}+72\,iX-18\,X-4335-2312\,i)|| \\
	\hline 19 & 0 & -21964 &\ppp {X}^{4}+10982\,{X}^{2}+47045881\\
	({X}^{2}+12\,\sqrt {19}X+6859)\times ({X}^{2}-12\,\sqrt {19}X+6859)|| \\
	\hline 23 & 0 & 24932 &\ppp {X}^{4}-12466\,{X}^{2}+148035889\\
	({X}^{2}+40\,\sqrt {23}X+12167)\times ({X}^{2}-40\,\sqrt {23}X+12167)|| \\
	\hline 29 & 140 & -62412 &\ppp {X}^{4}-140\,{X}^{3}+41006\,{X}^{2}-3414460\,X+594823321\\
	({X}^{2}-28\,iX-70\,X+17661+16820\,i)\times ({X}^{2}+28\,iX-70\,X+17661-16820\,i)|| \\
	\hline 31 & 0 & -55676 &\ppp {X}^{4}+27838\,{X}^{2}+887503681\\
	({X}^{2}+32\,\sqrt {31}X+29791)\times ({X}^{2}-32\,\sqrt {31}X+29791)|| \\
	\hline 37 & 60 & 128660 &\ppp {X}^{4}-60\,{X}^{3}-62530\,{X}^{2}-3039180\,X+2565726409\\
	({X}^{2}-180\,iX-30\,X-47915+16428\,i)\times ({X}^{2}+180\,iX-30\,X-47915-16428\,i)|| \\
	\hline 41 & -140 & -56988 &\ppp {X}^{4}+140\,{X}^{3}+38294\,{X}^{2}+9648940\,X+4750104241\\
	({X}^{2}+56\,iX+70\,X+15129+67240\,i)\times ({X}^{2}-56\,iX+70\,X+15129-67240\,i)|| \\
	\hline 43 & 0 & -250604 &\ppp {X}^{4}+125302\,{X}^{2}+6321363049\\
	({X}^{2}+28\,\sqrt {43}X+79507)\times ({X}^{2}-28\,\sqrt {43}X+79507)|| \\
	\hline 47 & 0 & -391228 &\ppp {X}^{4}+195614\,{X}^{2}+10779215329\\
	({X}^{2}-16\,\sqrt {47}X+103823)\times ({X}^{2}+16\,\sqrt {47}X+103823)|| \\
	\hline 53 & 924 & -113580 &\ppp {X}^{4}-924\,{X}^{3}+483678\,{X}^{2}-137562348\,X+22164361129\\
	({X}^{2}+132\,iX-462\,X+126405-78652\,i)\times
	({X}^{2}-132\,iX-462\,X+126405+78652\,i)|| \\
	\hline 59 & 0 & -502444 &\ppp {X}^{4}+251222\,{X}^{2}+42180533641\\
	({X}^{2}+52\,\sqrt {59}X+205379)\times ({X}^{2}-52\,\sqrt {59}X+205379)|| \\
	\hline 61 & -820 & 15796 &\ppp {X}^{4}+820\,{X}^{3}+328302\,{X}^{2}+186124420\,X+51520374361\\
	({X}^{2}-492\,iX+410\,X-40931-223260\,i)\times ({X}^{2}+492\,iX+410\,X-40931+223260\,i)|| \\
	\hline 67 & 0 & -1200908 &\ppp {X}^{4}+600454\,{X}^{2}+90458382169\\
	({X}^{2}-4\,\sqrt {67}X+300763)\times ({X}^{2}+4\,\sqrt {67}X+300763)|| \\
	\hline 71 & 0 & -986332 &\ppp {X}^{4}+493166\,{X}^{2}+128100283921\\
	({X}^{2}-56\,\sqrt {71}X+357911)\times ({X}^{2}+56\,\sqrt {71}X+357911)|| \\
	\hline 73 & -396 & 693220 &\ppp {X}^{4}+396\,{X}^{3}-268202\,{X}^{2}+154050732\,X+151334226289\\
	({X}^{2}-528\,iX+198\,X-293095-255792\,i)\times ({X}^{2}+528\,iX+198\,X-293095+255792\,i)|| \\
	\hline 79 & 0 & -1931708 &\ppp {X}^{4}+965854\,{X}^{2}+243087455521\\
	({X}^{2}-16\,\sqrt {79}X+493039)\times ({X}^{2}+16\,\sqrt {79}X+493039)|| \\
	\hline 83 & 0 & 966452 &\ppp {X}^{4}-483226\,{X}^{2}+326940373369\\
	({X}^{2}-140\,\sqrt {83}X+571787)\times ({X}^{2}+140\,\sqrt {83}X+571787)|| \\
	\hline 89 & -300 & 1165476 &\ppp {X}^{4}+300\,{X}^{3}-537738\,{X}^{2}+211490700\,X+496981290961\\
	({X}^{2}+240\,iX+150\,X-308919+633680\,i)\times ({X}^{2}-240\,iX+150\,X-308919-633680\,i)|| \\
	\hline 97 & -252 & -2420860 &\ppp {X}^{4}+252\,{X}^{3}+1242182\,{X}^{2}+229993596\,X+832972004929\\
	({X}^{2}-56\,iX+126\,X+611585-677448\,i)\times ({X}^{2}+56\,iX+126\,X+611585+677448\,i)|| \\
	\hline
\end{longtable}
\egroup

\bigskip

Frobenius polynomials for $X_{(1:-1)}$ are tensor products of the characteristic polynomials $\chi(f_{2})$ and $\chi(f_{3})$ for the weight 2 modular form of level 32 (32A1, 32.2.a.a in \cite{LMFDB}) with CM by the Dirichlet character $\chi_{4}(3,\bullet) = \left(\frac{-1}{\bullet}\right)$
\[f_{2}:=q - 2q^5 - 3q^9 + 6q^{13} + 2q^{17} - q^{25} - 10q^{29} - 2q^{37} + 10q^{41} + 6q^{45} - 7q^{49}+ 14q^{53}- 10q^{61}- 12q^{65} + o(q^{73})\]
and the modular form of weight 3 level 32 with coefficients in $\mathbb Z[4i]$ (32.3.c.a in \cite{LMFDB})
\begin{eqnarray*}
	f_{3}=q+4\,i{q}^{3}+2\,{q}^{5}-8\,i{q}^{7}-7\,{q}^{9}-4\,i{q}^{11}-14\,{q}^{13}+
	8\,i{q}^{15}+18\,{q}^{17}-12\,i{q}^{19}+32\,{q}^{21}+ o \left( {q}^{21} \right)
\end{eqnarray*}
\begin{longtable}{L|L|Lp{1.5cm}L|L|L}
	p&\chi_{p}(f_{2})&\chi_{p}(f_{3})&&		p&\chi_{p}(f_{2})&\chi_{p}(f_{3})\\
	\cline{1-3}\cline{5-7}\endhead
	3& X^{2}+3 & X^{2} + 4 i X - 9&&	43& X^{2}+43 & X^{2}+28 i X -1849 \\
	5& X^{2}+2 X +5 & X^{2}-2 X +25 &&47& X^{2}+47 & X^{2}+16 i X -2209 \\
	7& X^{2}+7 & X^{2} - 8 i X - 49 &&53& X^{2}-14 X +53 & X^{2}-66 X +2809 \\
	11& X^{2}+11 & X^{2} - 4 i X - 121 &&59& X^{2}+59 & X^{2}-52 i X -3481 \\
	13& X^{2}-6 X +13 & X^{2}+14 X +169 &&61& X^{2}+10 X +61 & X^{2}-82 X +3721 \\
	17& X^{2}-2 X +17 & X^{2}-18 X +289 &&
67& X^{2}+67 & X^{2}+4 i X - 4489 \\
	19& X^{2}+19 & X^{2} - 12 i X - 361 &&
71& X^{2}+71 & X^{2}+56 i X - 5041 \\
	23& X^{2}+23 & X^{2} + 40 i X -529 &&73& X^{2}+6 X +73 & X^{2}-66 X +5329 \\
	29& X^{2}+10 X +29 & X^{2}+14 X +841
&&79& X^{2}+79 & X^{2}-16 i X - 6241 \\
	31& X^{2}+31 & X^{2} - 32 i X +961
&&83& X^{2}+83 & X^{2}-140 i X - 6889 \\
	37& X^{2}+2 X +37 & X^{2}+30 X +1369
&&89& X^{2}-10 X +89 & X^{2}+30 X +7921 \\
	41& X^{2}-10 X +41 & X^{2}+14 X +1681
&&97& X^{2}-18 X +97 & X^{2}+14 X +9409 \\

\end{longtable}

Characteristic polynomial of the Galois action on $H^{3}_{\text{\'et}}(X)$ equals the tensor product of the characteristic polynomials of the Galois actions defined by $f_{2}$ and $f_{3}$
\[F_{p} = \chi_{p}(f_{2}) \otimes \chi_{p}(f_{3})\]
In fact, the form $f_{3}$ is defined only up to complex conjugate and the complex conjugate $\bar f_{3}$ is the quadratic twist of $f_{3}$ by $-1$
\[\bar f_{3} = f_{3} \otimes \chi_{4}(3,\bullet)\]
and consequently the tensor product $f_{2}\otimes f_{3}$ is equal to its complex conjugate, i.e. it has integral coefficients.

As the Complex Multiplication decomposes the Galois representation on $H^{3}_{\text{\'et}}(X_{(1:-1)})$ into the $\pm i$-eigenspaces of dimension 2, we believe that the modularity in this case can be proved using the approach presented in \cite{CScvS}.
The Galois representation for the modular form $f_{3}$ has a motivic realization (\cite{Scholl}), hence this example reminds the Borcea-Voisin Calabi--Yau.

We have similar numerical evidences of modularity for Arr. No. 155 at $A_{0}=1, A_{1}=-1$ with the equation
\[u^{2} = xy(x + y)z(x + z - t)t(x - y + 3z - t)(x + y + z + t)\]

\bgroup
\tiny
\begin{longtable}[t]{|r|r|r|L|}
	\hline $p$&$a_{p}$&$a_{p^{2}}$&F_{p}\\\endhead
	\hline 2 & 0 & -34 &\ppp {X}^{4}+17\,{X}^{2}+64\\
	({X}^{2}-iX+8)\times ({X}^{2}+iX+8)|| \\
	\hline 3 & 0 & -54 &\ppp {X}^{4}+27\,{X}^{2}+729\\
	(X^{2}+3\,\sqrt {3}X+27)\times (X^{2}-3\sqrt {3}X+27)|| \\
	\hline 5 & -16 & -204 &\ppp {X}^{4}+16\,{X}^{3}+230\,{X}^{2}+2000\,X+15625\\
	({X}^{2}-4\,iX+8\,X+75-100\,i)\times ({X}^{2}+4\,iX+8\,X+75+100\,i)|| \\
	\hline 7 & 0 & -1148 &\ppp {X}^{4}+574\,{X}^{2}+117649\\
	({X}^{2}+4\,\sqrt {7}X+343)\times ({X}^{2}-4\,\sqrt {7}X+343)|| \\
	\hline 11 & 0 & 308 &\ppp {X}^{4}-154\,{X}^{2}+1771561\\
	({X}^{2}-16\,\sqrt {11}X+1331)\times ({X}^{2}+16\,\sqrt {11}X+1331)|| \\
	\hline 13 & 12 & -3340 &\ppp {X}^{4}-12\,{X}^{3}+1742\,{X}^{2}-26364\,X+4826809\\
	({X}^{2}+4\,iX-6\,X+845-2028\,i)\times ({X}^{2}-4\,iX-6\,X+845+2028\,i)|| \\
	\hline 17 & -192 & -60 &\ppp {X}^{4}+192\,{X}^{3}+18462\,{X}^{2}+943296\,X+24137569\\
	({X}^{2}+24\,iX+96\,X+4335+2312\,i)\times ({X}^{2}-24\,iX+96\,X+4335-2312\,i)|| \\
	\hline 19 & 0 & -5548 &\ppp {X}^{4}+2774\,{X}^{2}+47045881\\
	({X}^{2}-24\,\sqrt {19}X+6859)\times ({X}^{2}+24\,\sqrt {19}X+6859)|| \\
	\hline 23 & 0 & -1564 &\ppp {X}^{4}+782\,{X}^{2}+148035889\\
	({X}^{2}+32\,\sqrt {23}X+12167)\times ({X}^{2}-32\,\sqrt {23}X+12167)|| \\
	\hline 29 & 176 & -10668 &\ppp {X}^{4}-176\,{X}^{3}+20822\,{X}^{2}-4292464\,X+594823321\\
	({X}^{2}+220\,iX-88\,X-17661-16820\,i)\times ({X}^{2}-220\,iX-88\,X-17661+16820\,i)|| \\
	\hline 31 & 0 & 48484 &\ppp {X}^{4}-24242\,{X}^{2}+887503681\\
	({X}^{2}-52\,\sqrt {31}X+29791)\times ({X}^{2}+52\,\sqrt {31}X+29791)|| \\
	\hline 37 & -36 & 168980 &\ppp {X}^{4}+36\,{X}^{3}-83842\,{X}^{2}+1823508\,X+2565726409\\
	({X}^{2}+108\,iX+18\,X-47915+16428\,i)\times ({X}^{2}-108\,iX+18\,X-47915-16428\,i)|| \\
	\hline 41 & 64 & 59364 &\ppp {X}^{4}-64\,{X}^{3}-27634\,{X}^{2}-4410944\,X+4750104241\\
	({X}^{2}-40\,iX-32\,X-15129+67240\,i)\times ({X}^{2}+40\,iX-32\,X-15129-67240\,i)|| \\
	\hline 43 & 0 & -48332 &\ppp {X}^{4}+24166\,{X}^{2}+6321363049\\
	({X}^{2}-56\,\sqrt {43}X+79507)\times ({X}^{2}+56\,\sqrt {43}X+79507)|| \\
	\hline 47 & 0 & -319036 &\ppp {X}^{4}+159518\,{X}^{2}+10779215329\\
	({X}^{2}+32\,\sqrt {47}X+103823)\times ({X}^{2}-32\,\sqrt {47}X+103823)|| \\
	\hline 53 & -144 & 388980 &\ppp {X}^{4}+144\,{X}^{3}-184122\,{X}^{2}+21438288\,X+22164361129\\
({X}^{2}-252\,iX+72\,X-126405-78652\,i)\times ({X}^{2}+252\,iX+72\,X-126405+78652\,i)|| \\
	\hline 59 & 0 & -700684 &\ppp {X}^{4}+350342\,{X}^{2}+42180533641\\
({X}^{2}+32\,\sqrt {59}X+205379)\times ({X}^{2}-32\,\sqrt {59}X+205379)|| \\
	\hline 61 & 620 & 79156 &\ppp {X}^{4}-620\,{X}^{3}+152622\,{X}^{2}-140728220\,X+51520374361\\
({X}^{2}+372\,iX-310\,X-40931-223260\,i)\times ({X}^{2}-372\,iX-310\,X-40931+223260\,i)|| \\
	\hline 67 & 0 & -345452 &\ppp {X}^{4}+172726\,{X}^{2}+90458382169\\
	({X}^{2}-80\,\sqrt {67}X+300763)\times ({X}^{2}+80\,\sqrt {67}X+300763)|| \\
	\hline 71 & 0 & 894884 &\ppp {X}^{4}-447442\,{X}^{2}+128100283921\\
	({X}^{2}-128\,\sqrt {71}X+357911)\times ({X}^{2}+128\,\sqrt {71}X+357911)|| \\
	\hline 73 & -396 & 693220 &\ppp {X}^{4}+396\,{X}^{3}-268202\,{X}^{2}+154050732\,X+151334226289\\
	({X}^{2}+528\,iX+198\,X-293095+255792\,i)\times ({X}^{2}-528\,iX+198\,X-293095-255792\,i)|| \\
	\hline 79 & 0 & -1908956 &\ppp {X}^{4}+954478\,{X}^{2}+243087455521\\
	({X}^{2}-20\,\sqrt {79}X+493039)\times ({X}^{2}+20\,\sqrt {79}X+493039)|| \\
	\hline 83 & 0 & -2244652 &\ppp {X}^{4}+1122326\,{X}^{2}+326940373369\\
	({X}^{2}-16\,\sqrt {83}X+571787)\times ({X}^{2}+16\,\sqrt {83}X+571787)|| \\
	\hline 89 & 2304 & 381732 &\ppp {X}^{4}-2304\,{X}^{3}+2463342\,{X}^{2}-1624248576\,X+496981290961\\
	({X}^{2}+720\,iX-1152\,X+308919-633680\,i)\times ({X}^{2}-720\,iX-1152\,X+308919+633680\,i)|| \\
	\hline 97 & -1692 & -1297660 &\ppp {X}^{4}+1692\,{X}^{3}+2080262\,{X}^{2}+1544242716\,X+832972004929\\
	({X}^{2}-376\,iX+846\,X+611585-677448\,i)\times ({X}^{2}+376\,iX+846\,X+611585+677448\,i)|| \\
	\hline
\end{longtable}
\egroup

\medskip

In this case the Frobenius polynomials for primes $2\le p\le97$ are the tensor products of Frobenius polynomials of the modular form of weight 2 and level 288 (288.2.a.a)
\[f_{2} = q - 4  q^5 - 6  q^{13} - 8  q^{17} + 11  q^{25} + 4  q^{29} - 2  q^{37} + 8  q^{41} - 7  q^{49} + 4  q^{53} - 10  q^{61} + 24  q^{65} + 6  q^{73} + 32  q^{85} - 16  q^{89}  + o(q^{96})\]
and weight 3 form of level (288.3.g.c)
\[f_{3} = q + 4  q^5 + 4i  q^7 + 16i  q^{11} - 2  q^{13} + 24  q^{17} - 24i  q^{19} + 32i  q^{23} - 9  q^{25} + 44  q^{29} + 52i  q^{31} + 16i  q^{35} + 18  q^{37} + 8  q^{41} +o(q^{42})\]

\begin{longtable}{L|L|Lp{1.5cm}L|L|L}

	p&\chi_{p}(f_{2})&\chi_{p}(f_{3})&&		p&\chi_{p}(f_{2})&\chi_{p}(f_{3})\\
	\cline{1-3}\cline{5-7}\endhead
	3& X^{2}+3 & X^{2} - 9 &&
	43& X^{2}+43 & X^{2} + 56 i X -1849 \\
	5& X^{2}+4 X +5 & X^{2}-4 X +25 &&	47& X^{2}+47 & X^{2}-32 i X -2209 \\
	7& X^{2}+7 & X^{2}-4 i X - 49 &&53& X^{2}-4 X +53 & X^{2}+36 X +2809 \\
	11& X^{2}+11 & X^{2}-16 i X - 121 &&59& X^{2}+59 & X^{2}+32 i X -3481 \\
	13& X^{2}+6 X +13 & X^{2}+2 X +169 &&61& X^{2}+10 X +61 & X^{2}+62 X +3721\\
	17& X^{2}+2 X +17 & X^{2}-24 X +289 &&67& X^{2}+67 & X^{2}+80 i X - 4489\\
	19& X^{2}+19 & X^{2}+24 i X - 361 &&71& X^{2}+71 & X^{2}+128 i X - 5041\\
	23& X^{2}+23 & X^{2}-32 i X -529 &&73& X^{2}-6 X +73 & X^{2}+66 X +5329\\
	29& X^{2}-4 X +29 & X^{2}-44 X +841&&79& X^{2}+79 & X^{2}+20 i X - 6241 \\
	31& X^{2}+31 & X^{2} - 52 i X +961 &&83& X^{2}+83 & X^{2}+16 i X - 6889 \\
	37& X^{2}+2 X +37 & X^{2} -18 X +1369 &&89& X^{2}+16 X +89 & X^{2} + 144 X +7921\\
	41& X^{2}-8 X +41 & X^{2} -8 X +1681 &&97& X^{2}+18 X +97 & X^{2} - 94 X +9409 \\
\end{longtable}

\subsection{Calabi--Yau threefold over a number field}
The one parameter family of double octics for Arrangement No. 148, defined by the singular double cover
\begin{multline*}
	u^{2}=\{(x,y,z,t,u)\in\PP(1,1,1,1,4):
	u^{2} = xy(A_{0}x + A_{1}y)z(A_{0}x + A_{2}z + A_{0}t)t\\
	\times((-A_{0} + A_{1})x + A_{1}z + A_{3}t)(x + y + z + t)\}.
\end{multline*}
It has eight invariant permutations (the dihedral group $D_{4}$). This three-dimensional family contains two distinguished points for $A=(2i,i+1,i-1,1)$ and $A=(-2i,-i+1,-i-1,1)$.

Projective transformation
\begin{multline*}
\begin{array}{c}
\Phi:\left(	x,y,z,t\right)\longmapsto
\biggl( A_{3}(A_{1}-A_{3})x, (A_{1} - A_{3}) (A_{0} - A_{1} + A_{3}) y, - A_{3} (A_{0} - A_{1} + A_{3}) z,\\
(A_{0} A_{1} - A_{0} A_{3} - A_{1}^{2}  + A_{1} A_{3}) x + A_{1} (A_{0} - A_{1} + A_{3}) z + (A_{1} - A_{3}) (A_{0} - A_{1} + A_{3}) t
\biggr)
\end{array}
\end{multline*}
defines isomorphism between the fiber over $A=(A_{0},A_{1}.A_{2},A_{3})$ and quadratic twist by
\[
(A_{0} - A_{1} + A_{3})^{5}  (A_{3} - A_{1})^{3}  A_{3}^{2}
\]
of the fiber at
\begin{multline*}
	\phi(A_{0},A_{1},A_{2},A_{3}) =
	\biggl((A_{1} - A_{3}) A_{0} A_{3}, A_{1} (A_{0} - A_{1} + A_{3}) (A_{1} - A_{3}),\\ A_{3} (A_{0} A_{1} - A_{2} A_{3}),  (A_{0} - A_{1} + A_{3}) (A_{1} - A_{3})^{3}\biggr)
\end{multline*}
In particular, it gives an isomorphism over $\QQ[i]$ between the quadratic twist $X'$ by $1+i$ of the geometrically distinguished fiber at $(2i,1+i,-1+i,1)$ and its complex conjugate. Consequently, Calabi--Yau manifold $X'$ is defined over $\mathbb Q$.
For any prime $p<100$ with $p\equiv 1\pmod4$ we have directly computed the number of points of reduction of singular double covering $Y'$ of $\mathbb P^{3}$ defining $X'$ modulo $p$
\[
\begin{array}{c|*{11}{|c}}
	p&5&13&17&29&37&41&53&61&73&89&97\rule[-2mm]{0mm}{7mm}\\\hline
	\#(Y'_{p})&142&2382&5354&25838&51182& 71098&152110&230222& 391626& 713418& 923690\rule[-2mm]{0mm}{7mm}
\end{array}
\]

For primes $p\equiv 3\pmod 4$, $\sqrt{-1}\not\in\mathbb F_{p}$, so we can count points only over $\mathbb F_{p^{2}}$. Using the method of \cite{czarnecki} we have computed for primes $p<100$

\[
\begin{array}{|c|*{5}{|c}}\rule[-2mm]{0mm}{7mm}
	p&3&7&11&19&23\\\hline\rule[-2mm]{0mm}{7mm}
	\#(Y'_{p^{2}})&906&119994&1795082&47200138&148357562 \\\hline\rule[-2mm]{0mm}{7mm}
	 p&31&43&47&59&67\\\cline{1-6}\rule[-2mm]{0mm}{7mm}
	 \#(Y'_{p^{2}})&888648186&6325127562&10784332794&42192140042&90479547402\\\hline\rule[-2mm]{0mm}{7mm}
	 p&71&79&\multicolumn{1}{c|}{83}\\\cline{1-4}\rule[-2mm]{0mm}{7mm}
	 \#(Y'_{p^{2}})&128125886138&243128591866&\multicolumn{1}{c|}{326991195402}
\end{array}
\]

These numbers are congruent to $1-a_{p}$ modulo $p$ (resp. $1-a_{p}^{2}$ modulo $p^{2}$), where $a_{p}$ are coefficients of the modular form $128/4$ (the form 128.4.a.a in \cite{LMFDB})
\[f_{4}(q)=q - 2  q^3 - 6  q^5 + 20  q^7 - 23  q^9 - 14  q^{11} - 54  q^{13} + 12  q^{15} - 66  q^{17} - 162  q^{19} - 40  q^{21} + 172  q^{23} + o(q^{24}).
\]
Consequently we believe that the $L$-function of this modular form is a factor of $L(X')$ and a crepant resolution of the quotient of $X'$ by a group of its automorphisms is a rigid Calabi--Yau threefold of level $128=2^{8}$ (which by the Serre bound is the highest possible power 2 factor of the level of a modular form of a rigid Calabi--Yau threefold \cite{Serre}).

\subsection{Geometrically distinguished one-parameter sub-families of double octic Calabi--Yau manifolds.}
For a distinguished family we can have three types of behaviors: it may happen that the differential operator computed for a family fails to be a CY(3)-operators, it can be an operator already present in the database \cite{CYDB}, or it is a new CY(3)-operator. The following examples show, that all three possibilities actually appears.
\begin{exmp}
	Invariant permutation $(1,2)(7,8)$ of Arr. No. 383 defines the birational transformation
	\[ \mathcal X \lra \mathcal X, \]
	where $\phi([A_{0}:A_{1}:A_{2}:A_{3}]) = [A_{1}A_{3}:A_{0}A_{3}:A_{0}A_{2}:A_{0}A_{1}]$.
	The fixed locus $\operatorname{Fix}(\phi)$ consists of 7 components, one of them is the line $\{A_{0}+A_{1}=A_{0}+A_{3}=0\}$ and defines a geometrically distinguished family
	\(\mathcal X_{[A_{0}:-A_{0}:A_{2}:-A_{0}]} \) of double octic Calabi--Yau 3-folds
	 for the octic arrangements
	 \[xyz(A_{0}x-A_{0}y+A_{2}z)t(A_{0}x-A_{0}y-A_{2}t)(x+y-z-t)(x+y+z+t)=0.\]
	 The Picard-Fuchs operator of this family equals
	 \[\Theta^{4}-t^{2}(\Theta+1)^{4}\]
	 with the Riemann-Symbol
	 \[\left\{\begin{array}{*{4}c}
	 	0& 1& -1& \infty \\
	 	\hline0 & 0 & 0 & 1 \\
	 	0 & 1 & 1 & 1 \\
	 	0 & 1 & 1 & 1 \\
	 	0 & 2 & 2 & 1 \\
	 \end{array}\right\}.\]
	 It is a pullback by the map $t\mapsto \frac{1}{256}t^{2}$ of the hypergeometric operator
	 \[\Theta^4-2^{4} x\left((2\Theta+1)^4\right)\]
	 which is the operator 1.3 in the online database of CY(3)-operators \cite{CYDB} and
	 AESZ 3 in the seminal paper \cite{AESZ}. This is the Picard-Fuchs operator of the mirror of complete intersections $V_{2,2,2,2}$ of four quadrics in $\PP^{7}$.
\end{exmp}
\begin{exmp}
	Two-dimensional family of double octic Calabi--Yau threefolds defined by the arrangement No. 280 has an automorphism $\Phi(x,y,z,t) = (y,x,-z-t,t)$, with $\phi(A_{0}:A_{1}:A_{2}) = (A_{1}A_{2}:A_{0}A_{1}:A_{0}A_{2})$ (cf. Example \ref{ex280}).
	The fixed-locus of $\phi$ consists of an isolated point $(1:0:0)$ and a conic $A_{0}^{2}=A_{1}A_{2}$. Consequently we get a distinguished subfamily
	\[xyz(sx + y - s^{2}z) (z + t) (sx + s^{2}y + s^{2}z + s^{2}t) (x + t) (y + )t\]
	  Using the Magma package by P. Lairez \cite{Lairez} we have computed an order 4 ordinary differential operator annihilating invariant period integrals
	\begin{center}
		\(\displaystyle {\Theta}^{4}\)
		\mbox{\(\displaystyle\; - \; 2  \cdot  7 t(1488\,{\Theta}^{4}+1452\,{\Theta}^{3}+1125\,{\Theta}^{2}+399\,\Theta+56)\)}
		\mbox{\(\displaystyle\; + \; 2 ^{2} 7 t^{2}(766392\,{\Theta}^{4}+1184952\,{\Theta}^{3}+1010797\,{\Theta}^{2}+454076\,\Theta+83776)\)}
		\mbox{\(\displaystyle\; - \; 2 ^{4} 7 t^{3}(12943616\,{\Theta}^{4}+28354200\,{\Theta}^{3}+30710572\,{\Theta}^{2}+16054731\,\Theta+3215254)\)}
		\mbox{\(\displaystyle\; + \; 2 ^{6} 7 t^{4}(105973188\,{\Theta}^{4}+333359304\,{\Theta}^{3}+436182381\,{\Theta}^{2}+261265857\,\Theta+57189166)\)}
		\mbox{\(\displaystyle\; - \; 2 ^{11} 127  \cdot  7 t^{5}(\Theta+1)(390972\,{\Theta}^{3}+1350660\,{\Theta}^{2}+1486781\,\Theta+460439)\)}
		\mbox{\(\displaystyle\; + \; 2 ^{14} 23 ^{2} 127 ^{2} 7 t^{6}(\Theta+1)(\Theta+2)(2\,\Theta+1)(2\,\Theta+5)\)}
	\end{center}
with the following Riemann Symbol
	\[\left\{\begin{array}{*{9}c}
	0& {\frac{1}{64}}& -\frac1{32}& -{\frac{1}{64}}& -{\frac{1}{128}}+{\frac {\sqrt {5}}{128}}& -{\frac{1}{128}}-{\frac {\sqrt {5}}{128}}& -{\frac{13}{384}}+{\frac {\sqrt {145}}{384}}& -{\frac{13}{384}}-{\frac {\sqrt {145}}{384}}& \infty \\[1mm]
	\hline0 & 0 & 0 & 0 & 0 & 0 & 0 & 0 & 1 \\
	0 & 1 & 1 & 1 & 1 & 1 & 1 & 1 & 3/2 \\
	0 & 1 & 1 & 1 & 1 & 1 & 3 & 3 & 3/2 \\
	0 & 2 & 2 & 2 & 2 & 2 & 4 & 4 & 2 \\
	\end{array}\right\}\]
	and instanton numbers equal
	\[n_{1}=176, n_{2}=-8330, n_{3}=720112, n_{4}=-81101933, n_{5}=10606760080,\dots\]
	There is no Calabi--Yau operator with these instanton numbers in the online database of CY(3) operators \cite{CYDB}.
\end{exmp}

We found more new examples CY(3) operators, we list here two operators with three non-isomorphic MUM points. These examples are specially interesting, because previously only two CY(3) operators with more than two MUM points were known (cf. \cite{vS}).

\begin{exmp}\leavevmode\samepage
Two-dimensional family of double octic Calabi--Yau threefolds defined by arrangement No. 283 has an automorphism
\begin{eqnarray*}
\Phi:&(x,y,z,t)&\mapsto((A_{0}+A_{2})x+A_{0}y+A_{0}z, (A_{0}+A_{2})t, -(A_{0}+A_{2})(x+y+z+t), A_{2}y)\\
\phi:&(A_{0},A_{1},A_{2})&\mapsto(A_{0}A_{1}, A_{0}A_{2}A_{2}^2, A_{2}A_{1})
\end{eqnarray*}
The fixed locus of the transformation consists of four isolated points and the conic
\[\{A_{0}A_{2}-A_{1}^2+A_{2}^2=0\}\subset\PP^{2}.\]
Substituting  the following parametrization
\[A_{0}=s_{0}^{2}-s_{1}^{2}, A_{1}=s_{0}s_{1}, A_{2}=s_{1}^{2}\]
we get a one-parameter distinguished subfamily with the Picard-Fuchs operator equal
\begin{center}
	\(\displaystyle {\Theta}^{4}\)
	\mbox{\(\displaystyle\; + \; 2 ^{2}t(20\,{\Theta}^{4}-32\,{\Theta}^{3}-27\,{\Theta}^{2}-11\,\Theta-2)\)}
	\mbox{\(\displaystyle\; + \; 2 ^{6} 3 t^{2}(28\,{\Theta}^{4}-560\,{\Theta}^{3}-413\,{\Theta}^{2}-273\,\Theta-96)\)}
	\mbox{\(\displaystyle\; - \; 2 ^{11} 3 t^{3}(128\,{\Theta}^{4}+540\,{\Theta}^{3}+293\,{\Theta}^{2}+3\,\Theta-102)\)}
	\mbox{\(\displaystyle\; - \; 2 ^{15} 3 t^{4}(514\,{\Theta}^{4}-292\,{\Theta}^{3}-4799\,{\Theta}^{2}-7305\,\Theta-4041)\)}
	\mbox{\(\displaystyle\; + \; 2 ^{18} 3 t^{5}(596\,{\Theta}^{4}+11624\,{\Theta}^{3}+31979\,{\Theta}^{2}+37719\,\Theta+17482)\)}
	\mbox{\(\displaystyle\; + \; 2 ^{22} 3 t^{6}(1820\,{\Theta}^{4}+12456\,{\Theta}^{3}+18139\,{\Theta}^{2}+5283\,\Theta-6554)\)}
	\mbox{\(\displaystyle\; + \; 2 ^{28} 3 t^{7}(216\,{\Theta}^{4}-1044\,{\Theta}^{3}-10541\,{\Theta}^{2}-21863\,\Theta-15404)\)}
	\mbox{\(\displaystyle\; - \; 2 ^{32} 3 t^{8}(177\,{\Theta}^{4}+3804\,{\Theta}^{3}+15415\,{\Theta}^{2}+24365\,\Theta+14305)\)}
	\mbox{\(\displaystyle\; - \; 2 ^{34} 3 t^{9}(1188\,{\Theta}^{4}+10512\,{\Theta}^{3}+20541\,{\Theta}^{2}+6453\,\Theta-10810)\)}
	\mbox{\(\displaystyle\; - \; 2 ^{38} 3 t^{10}(532\,{\Theta}^{4}+512\,{\Theta}^{3}-16543\,{\Theta}^{2}-51539\,\Theta-44212)\)}
	\mbox{\(\displaystyle\; + \; 2 ^{43} 3 t^{11}(80\,{\Theta}^{4}+2492\,{\Theta}^{3}+13305\,{\Theta}^{2}+26375\,\Theta+18134)\)}
	\mbox{\(\displaystyle\; + \; 2 ^{47} 3 t^{12}(146\,{\Theta}^{4}+2004\,{\Theta}^{3}+8639\,{\Theta}^{2}+15393\,\Theta+9881)\)}
	\mbox{\(\displaystyle\; + \; 2 ^{50} 3 t^{13}(\Theta+2)(52\,{\Theta}^{3}+432\,{\Theta}^{2}+1175\,\Theta+1053)\)}\\
	\mbox{\(\displaystyle\; + \; 2 ^{54}t^{14}(\Theta+2)(\Theta+3)(2\,\Theta+5)^{2}\)}
\end{center}
\[\left\{\begin{array}{*{9}c}
	0& 1/16& -1/16& \alpha_{1}& \alpha_{2}& \beta_{1} & \beta_{2} & \beta_{3} & \infty \\
	\hline0 & 0 & 0 & 0 & 0 & 0 & 0 & 0 & 2 \\
	0 & 0 & 0 & 1 & 1 & 1 & 1 & 1 & 5/2 \\
	0 & 0 & 0 & 1 & 1 & 3 & 3 & 3 & 5/2 \\
	0 & 0 & 0 & 2 & 2 & 4 & 4 & 4 & 3 \\
\end{array}\right\}\]

$\alpha_{i}$ are roots of $t^{2}+\frac1{16}6-\frac1{256}$

$\beta_{i}$ are roots of ${t}^{3}+{\frac {5\,{t}^{2}}{48}}+{\frac {3\,t}{256}}+{\frac{1}{4096}}$

This operator has three points of Maximal Unipotetn Monodromy: $0, -1/16$ and $1/16$ with instanton numbers:

$n_{1}=12, n_{2}=-215/3, n_{3}=2444/3, n_{4}=-12858, n_{5}=717344/3, n_{6}=-4814271, \dots$,

$n_{1}=16, n_{2}=10, n_{3}=1232, n_{4}=-8637, n_{5}=408496, n_{6}=-706174, \dots$

$n_{1}=4/5, n_{2}=7/5, n_{3}=12/5, n_{4}=391/20, n_{5}=80, n_{6}=2249/5, \dots$

The map $\Phi$ acts on each element of the distinguished subfamily as an involution preserving the canonical form. Consequently, the quotient resolves to a one parameter family of Calabi--Yau threefolds with $h^{1,2}=1$ and the above differential operstor is the Picard-Fuchs operstor of constructed family. Now, the Mirror Symmetry Conjecture predicts that in the A-side we get three non-birational Calabi--Yau threefolds with isomorphic derived categories (see \cite{vS})..

\end{exmp}
\begin{exmp}\leavevmode
	Similar construction for arrangements No. 311 yields the following differential operator.
\begin{center}
	\(\displaystyle {\Theta}^{4}\)
	\mbox{\(\displaystyle\; + \; 2 ^{4}t(86\,{\Theta}^{4}+40\,{\Theta}^{3}+51\,{\Theta}^{2}+31\,\Theta+7)\)}
	\mbox{\(\displaystyle\; + \; 2 ^{9}t^{2}(1598\,{\Theta}^{4}+1784\,{\Theta}^{3}+2391\,{\Theta}^{2}+1511\,\Theta+402)\)}
	\mbox{\(\displaystyle\; + \; 2 ^{14}t^{3}(16946\,{\Theta}^{4}+32784\,{\Theta}^{3}+47889\,{\Theta}^{2}+34893\,\Theta+10879)\)}
	\mbox{\(\displaystyle\; + \; 2 ^{19}t^{4}(113936\,{\Theta}^{4}+330640\,{\Theta}^{3}+548153\,{\Theta}^{2}+458685\,\Theta+161792)\)}
	\mbox{\(\displaystyle\; + \; 2 ^{26} 3 t^{5}(42306\,{\Theta}^{4}+170832\,{\Theta}^{3}+328476\,{\Theta}^{2}+311857\,\Theta+122041)\)}
	\mbox{\(\displaystyle\; + \; 2 ^{31}t^{6}(375780\,{\Theta}^{4}+2052984\,{\Theta}^{3}+4642349\,{\Theta}^{2}+4967559\,\Theta+2135990)\)}
	\mbox{\(\displaystyle\; + \; 2 ^{39}t^{7}(86908\,{\Theta}^{4}+666692\,{\Theta}^{3}+1818436\,{\Theta}^{2}+2199564\,\Theta+1036825)\)}
	\mbox{\(\displaystyle\; + \; 2 ^{45}t^{8}(36105\,{\Theta}^{4}+522570\,{\Theta}^{3}+1862077\,{\Theta}^{2}+2606519\,\Theta+1359197)\)}
	\mbox{\(\displaystyle\; - \; 2 ^{51}t^{9}(19328\,{\Theta}^{4}-158496\,{\Theta}^{3}-1128728\,{\Theta}^{2}-1992033\,\Theta-1189419)\)}
	\mbox{\(\displaystyle\; - \; 2 ^{57}t^{10}(38052\,{\Theta}^{4}+125520\,{\Theta}^{3}-217375\,{\Theta}^{2}-837115\,\Theta-647874)\)}
	\mbox{\(\displaystyle\; - \; 2 ^{64} 3 t^{11}(3756\,{\Theta}^{4}+27828\,{\Theta}^{3}+37294\,{\Theta}^{2}-4167\,\Theta-27103)\)}
	\mbox{\(\displaystyle\; - \; 2 ^{71}t^{12}(779\,{\Theta}^{4}+18474\,{\Theta}^{3}+46343\,{\Theta}^{2}+39693\,\Theta+6854)\)}
	\mbox{\(\displaystyle\; + \; 2 ^{78}t^{13}(464\,{\Theta}^{4}-434\,{\Theta}^{3}-4838\,{\Theta}^{2}-7225\,\Theta-3051)\)}
	\mbox{\(\displaystyle\; + \; 2 ^{85}t^{14}(154\,{\Theta}^{4}+634\,{\Theta}^{3}+980\,{\Theta}^{2}+635\,\Theta+139)\)}
	\mbox{\(\displaystyle\; + \; 2 ^{92}t^{15}(20\,{\Theta}^{4}+126\,{\Theta}^{3}+308\,{\Theta}^{2}+345\,\Theta+149)\)}\\
	\mbox{\(\displaystyle\; + \; 2 ^{99}t^{16}(\Theta+2)^{4}\)}
\end{center}

\[\left\{\begin{array}{*{11}c}
	0& -1/32& -{\frac{1}{64}}& -{\frac{1}{128}}& -{\frac{3}{128}}+{\frac {\sqrt {5}}{128}}& -{\frac{3}{128}}-{\frac {\sqrt {5}}{128}}& \beta_{1} & \beta_{2} & \beta_{3} & \beta_{4} & \infty \\
	\hline0 & 0 & 0 & 0 & 0 & 0 & 0 & 0 & 0 & 0 & 2 \\
	0 & 1 & 0 & 0 & 1 & 1 & 1 & 1 & 1 & 1 & 2 \\
	0 & 1 & 0 & 1 & 1 & 1 & 3 & 3 & 3 & 3 & 2 \\
	0 & 2 & 0 & 1 & 2 & 2 & 4 & 4 & 4 & 4 & 2 \\
\end{array}\right\}\]

$\alpha_{i}$ are roots of $t^{2}+\frac3{64}t+\frac1{4096}$

$\beta_{i}$ are roots of ${t}^{4}+{\frac {{t}^{3}}{128}}-{\frac {{t}^{2}}{2048}}-{\frac {11\,t}{1048576}}-{\frac{1}{33554432}}$

This operator has three points of Maximal Unipotetn Monodromy: $0, -1/64$ and $\infty$ with instanton numbers:

$n_{1}=48, n_{2}=-2170, n_{3}=104048, n_{4}=-7767917, n_{5}=619344656,  \dots,$

$n_{1}=4/3, n_{2}=173/3, n_{3}=-1492/3, n_{4}=-4318/3, n_{5}=78944, n_{6}=-515995, \dots,$

$n_{1}=32, n_{2}=-288, n_{3}=6560, n_{4}=-188232, n_{5}=6133600, n_{6}=-214394144, \dots,$

\end{exmp}

\begin{exmp}
The differential operator corresponding to geometrically distinguished subfamily of family no. 148 and the invariant permutation $(2,3)(4,6)(5,7)$ has local exponents at $t=0$ equal $(0,1/2,1/2,2)$ and consequently it is not a Calabi--Yau operator (cf. \cite{AZ, vS}).

\end{exmp}

\section{Fibrations of double octic Calabi--Yau threefolds}
\label{sec:fibr}
If $X\lra B$ is a fibration of a Calabi--Yau threefold then every smooth fiber $X_{b}$ has trivial canonical bundle, hence it is one of the following three types: an elliptic curve, a K3 or abelian surface.
A double octic Calabi--Yau threefold is constructed as double covering of a blow-up $\tilde \PP$ of $\PP^{3}$, we shall describe fibrations of a double octic induced by fibrations of $\tilde\PP$, or linear systems on $\PP^{3}$ that lifts to fibrations on $\tilde\PP$. We shall consider linear systems on $\PP^{3}$ of the following types:
\begin{enumerate}
	\item lines through an arrangement point of multiplicity 4 or 5,
	\item lines intersecting two disjoint double or triple arrangement lines,
	\item planes through two arrangements points at multiplicity at least 4 such that every arrangement plane contains at least one of them,
	\item planes through a double or triple arrangement line,
	\item a pencil of quadrics containing four double or triple arrangement  lines.
\end{enumerate}
Blow-ups performed in the resolution of singularities separate elements of the linear systems described above, and consequently produce fibration of the singular double covering.. We can describe fibers in the fibration more explicitly. For an octic arrangement $D=D_{1}+\dots +D_{8}$ and a smooth subvariety  $W\subset\PP^{3}$, let  $D|W = n_{1}E_{1}+\dots+n_{k}E_{k}$, be the restriction of the Cartier divisor $D$ to $W$. Then the normalization of the double covering of $W$ branched along $D|W$ equals the double cover of $W$ branched along $\widetilde {D|W} = \tilde n_{1}E_{1}+\dots+\tilde n_{k}E_{k}$, where
$\tilde n_{i}=n_{i}-2\lfloor \frac{n_{i}}2\rfloor$ is the reminder of $n_{i}$ modulo 2 ($i=1,\dots,k$).

In particular, if $L\subset \PP^{3}$ is a smooth rational curve such that the restriction $D|L$ contains four points with odd multiplicity, then $L$ lifts to an elliptic curve in the double octic. Similarly, if the restriction of $D|L$
contains 2 (resp. 0) points with odd multiplicity, then $L$ lifts to a rational curve (resp. union of two rational curves).

If the restriction $D|\Pi$ of the divisor $D$ to a plane $\Pi\subset \PP^{3}$ contains exactly 6 lines with odd multiplicity and no four of this six lines intersect, then  $S$ lifts to a surface birational to a double sextic K3 surface in the double octic.

Finally, let $Q\subset \PP^{3}$ be a smooth quadric, assume that the restriction  $D|Q$ consists of four conics (pairs of intersecting lines)
with odd multiplicity. Then $\widetilde {D|Q}$ is a (4,4)-curve. Consequently, $Q$ lifts to a surface birational to a K3 surface.

In subsequent subsections we shall give formulas for the fibrations in the above five cases.

\subsection{Elliptic fibrations determined by fourfold and fivefold points}
Assume that the point $A$ belongs to four arrangement planes $P_{1},\dots,P_{4}$ from the octic arrangement $D$.The restriction $D|L$ of $D$ to a generic line $L$ containing $A$ equals $4A+Q_{1}+Q_{2}+Q_{3}+Q_{4}$ and  $Q_{i}$ are distinct points. The line $L$ lifts to an elliptic curve $E$ in the double octic.

We consider separately the cases when $A$ has multiplicity 4 and 5. Assume that $A=(0:0:0:1)$, after change of coordinates we can write the equation of the plane $P_{i}$ in the form
\[F_{i}(x,y,z)+G_{i}t\]
where $F_{i}(x,y,z)$ is a polynomial linear with respect to variables $x,y,z$ and $G_{i}$ is a constant. Reordering, if necessary, planes $P_{i}$, we can assume that $G_{1}=\dots=G_{4}=0$ and $G_{5}=0$ when $A$ is a fivefold point and $G_{5}=1$ when $A$ is fourfold.

\begin{theorem}[\mbox{\cite[Thm. 2.1]{CK-C2}}]\label{thm:ellfive}
	If $A=(0:0:0:1)$ is a fivefold point then
	the double octic Calabi--Yau threefold defined by
	\begin{equation}\label{docsing5}
		u^{2}=F_{1}F_{2}F_{3}F_{4}F_{5}
		(F_{6}+t)(F_{7}+t)(F_{8}+t)
	\end{equation}
	is birational to the elliptic fibration
	\[u^{2} = \left(F_{1}F_{2}F_{3}F_{4}F_{5} F_{6}   +t \right) \left(F_{1}F_{2}F_{3}F_{4}F_{5} F_{7}  +t \right) \left(F_{1}F_{2}F_{3}F_{4}F_{5}F_{8}  +t \right)
	\]
	with discriminant
	\[
	\Delta = 16 (F_{1}F_{2}F_{3}F_{4}F_{5})^{6} \left({F_{7}}  -{F_{8}}  \right)^{2}
	\left({F_{6}}  -{F_{8}}  \right)^{2} \left({F_{6}} -{F_{7}} \right)^{2}\]
	and the $J$-invariant
	\[J= \frac{4 \left(
		{F_{6}}^{2} + {F_{7}}^{2}  + {F_{8}}^{2}
		-{F_{6}} F_{7}
		-{F_{6}} F_{8}
		-{F_{7}} F_{8}
		\right)^{3}}
	{27 {\left({F_{7}}  -{F_{8}} \right)^{2} \left({F_{6}} -{F_{8}} \right)^{2} \left({F_{6}}  -{F_{7}} \right)^{2}}}
	\]
\end{theorem}

\begin{theorem}[\mbox{\cite[Thm. 2.2]{CK-C2}}]\label{thm:ellfour}
	If $A=(0:0:0:1)$ is a fourfold point, then
	the double octic Calabi--Yau threefold defined by
	\begin{equation}\label{docsing4}
		u^{2}=F_{1}F_{2}F_{3}F_{4}(F_{5}+t)
		(F_{6}+t)(F_{7}+t)(F_{8}+t)
	\end{equation}
	is birational to the elliptic fibration

	\begin{multline*}
		u^{2}  = \Bigl(F_{1}F_{2}F_{3}F_{4} (F_{5} - F_{7}) (F_{5} - F_{6}) + t\Bigr) \times\\
		\Bigl(F_{1}F_{2}F_{3}F_{4} (F_{5} - F_{8}) (F_{5} - F_{6}) + t\Bigr)\Bigl(F_{1}F_{2}F_{3}F_{4} (F_{5} - F_{8}) (F_{5} - F_{7}) + t\Bigr)
	\end{multline*}
	with the discriminant
	\[\Delta = 16 (F_{1}F_{2}F_{3}F_{4})^{6}
	\prod\limits_{5\le i<j\le 8}\left({F_{i}}  -{F_{j}}  \right)^{2}
	\]
	and the $J$-invariant

	\[
	J = \frac
	{4\left(6F_{5}F_{6}F_{7}F_{8}  -\sum\limits_{5\le i<j<k\le 8}F_{i}F_{j}F_{k}(F_{i}+F_{j}+F_{k}) + \sum\limits_{5\le i<j\le 8}F_{i}^{2}F_{j}^{2}
		\right)^{3}}
	{\left(\prod\limits_{5\le i<j\le 8}(F_{i}-F_{j})^{2}\right)}
	\]

\end{theorem}

Elliptic fibrations resulting from the above theorem can be written as
\[u^{2}=t^{3}+f_{6}(x,y,z)t^{2}+f_{12}(x,y,z)t+f_{18}(x,y,z),\]
where $f_{6}, f_{12}, f_{18}$ are homogeneous polynomials in variables $x,y,z$ of degree 6, 12 and 18 respectively. The discriminant of the elliptic fibration is a non-reduced curve of degree 36, its reduction is a union of at most 10 lines (in the case of a fivefold point it is a union of at most 8 lines). Singular fibers at generic points of a line in the discriminant have Kodaira type that can be read from multiplicities of $\Delta$ and coefficients of the Weierstrass equation along that line.

The elliptic fibration \eqref{docsing4}  has the same $J$-function as the fibration
\[			u^{2}  = \Bigl( (F_{5} - F_{7}) (F_{5} - F_{6}) + t\Bigr)
\Bigl( (F_{5} - F_{8}) (F_{5} - F_{6}) + t\Bigr)\Bigl((F_{5} - F_{8}) (F_{5} - F_{7}) + t\Bigr)
\]
but they  are not birational in general because they have singular fibers of different types at generic points of $F_{1}F_{2}F_{3}F_{4} = 0$.
Smooth fibers of both fibrations are isomorphic, they differ by the quadratic twist by
$\sqrt{F_{1}F_{2}F_{3}F_{4}}$.
Similarly, for the elliptic fibrations ~\eqref{docsing5} and the elliptic fibration
\[
u^{2}=(F_{6}+t)(F_{7}+t)(F_{8}+t)
\]

\subsection{Elliptic fibrations determined by two skew double lines}

Assume that the arrangement planes $P_{1}, P_{2}, P_{3}$ and $P_{4}$ do not intersect at a point. Then lines $L_{1,2}=P_{1}\cap P_{2}$ and $L_{3,4}=P_{3}\cap P_{4}$ are disjoint. Let $L$ be a generic line that intersects $L_{12}$ and $L_{34}$, restriction $D|L = 2A+2B+Q_{1}+Q_{2}+Q_{3}+Q_{4}$,
where $Q_{1}, \dots, Q_{4}$ are pair-wise disjoint intersection points of $L$ with planes $P_{5},\dots,P_{8}$. The line $L$ lifts to an elliptic curve in the double octic.

After change of coordinates we can  assume that $P_{1}=x, P_{2}=y, P_{3}=z, P_{4}=t$, denote
\[F_{i}:=F_{i}(\gamma,\beta,0,0), \quad G_{i}:=F_{i}(0,0,\alpha,\gamma).\]
\begin{theorem}[\mbox{\cite{CK-C2}[Thm. 3.1]}]
	A double octic Calabi--Yau threefold defined by the octic arrangement
	\[D:=\{(x,y,z,t)\in\PP^{3}: xyzt\cdot F_{5}(x,y,z,t)\cdots F_{8}(x,y,z,t)=0\}\]
	is birational to the elliptic fibration
	\begin{equation}\label{eq:gen2}
		u^{2}=t\biggl(t-(F_{6} G_{8} - F_{8} G_{6}) (F_{5} G_{7} - F_{7} G_{5})\biggr) \biggl(t-(F_{6} G_{7} - F_{7} G_{6})
		(F_{5} G_{8} - F_{8} G_{5})\biggr),
	\end{equation}
	with the discriminant
	\[
	\Delta=16\prod_{5\le i<j\le8}(F_{i}G_{j}-F_{j}G_{i})^{2}
	\]
	and the $J$-invariant
	\begin{multline*}
		J=	\Biggl(27 \prod\limits_{5\le i<j\le8}\bigl(F_{i}G_{j}-F_{j}G_{i}\bigr)^{2}\Biggr)^{-1}
		\biggl(4\sum\limits_{\{i,j,k,l\}=\{5,6,7,8\},\atop i<j,\; k<l}F_{i}^{2}F_{j}^{2}G_{k}^{2}G_{l}^{2}\\
		-\sum\limits_{\{i,j,k,l\}=\{5,6,7,8\},\atop j<k} F_{i}^{2}F_{j}F_{k}G_{j}G_{k}G_{l}^{2}
		+6 {F_5} {F_6} {F_7} {F_8} {G_5} {G_6} {G_7} {G_8}
		\biggr)^{3}
	\end{multline*}
\end{theorem}

\subsection{Kummer fibrations}\label{sec:kumfibr}

Most double octics are birational to a Kummer fibration associated to a fiber product of two rational elliptic surfaces with section. If an octic arrangement $P_{1}\cup\dots\cup P_{8}$ satisfies property (3) from the beginning of this section, then after reordering planes $P_{i}$, we can assume that $P_{1}\cap \dots\cap P_{4}=\{A\}$, $P_{5}\cap\dots\cap P_{8}=\{B\}$  for some points $A$ and $B$. In this situation we shall show that the associated double octic  is birational to a Kummer fibration. More precisely, let $S$ be a generic plane containing points $A$ and $B$, projections from the point $A$ of lines $P_1\cap S,\dots P_4\cap S$ give four points in a projective line and consequently an elliptic curve $E_{1}$. In a similar manner the planes $P_{5},\dots,P_{8}$ define another elliptic curve $E_{2}$.

Blowing-up points $A$ and $B$ and then contracting the strict transform of the line joining these points we get a surface isomorphic to $\PP^{1}\times\PP^{1}$, the restriction $D|S$ lifts to 8 lines in  $\PP^{1}\times\PP^{1}$ - four in each ruling.
Two quadruples of lines determine the same elliptic curves $E_{1}$
 and $E_{2}$. Consequently, the restriction of the double octic to plane $S$ is birational to the Kummer surface $\operatorname{Km}(E_{1}\times E_{2})$ of the product $E_{1}\times E_{2}$. Varying the plane $S$ defines two elliptic surfaces.

Singular fibers of the first fibration are determined by projections from $A$ of the six lines
\[P_{1}\cap P_{2}, P_{3}\cap P_{4}, P_{1}\cap P_{3}, P_{2}\cap P_{4}, P_{1}\cap P_{4}, P_{2}\cap P_{3}.\]
Projections of these six lines need not be distinct, intersection projecting to a point determines type of the singular fiber at this point, one, two, three or four intersections  correspond to a singular fiber of type $I_{2}$, $I_{4}$, $I_{0}^{*}$ and $I_{2}^{*}$, respectively (equivalently, the Euler characteristic of the singular fiber is twice number of intersections projecting to a point).

On the level of geometry of a double octic, a singular fiber of type $I_{2}^{*}$ happens when
the point $A$ lies on a triple line and $B$ is a fivefold point.

The singular fiber of type $I_{0}^{*}$ is when the point $A$ lies on a triple line and $B$ is a fourfold point or $B$ is a fivefold point but $A$ does not lie on a triple line.
The two cases corresponds to situation when intersections of three pairs of lines  $P_{1}\cap P_{2}$, $P_{1}\cap P_{3}$ and  $P_{2}\cap P_{3}$ or intersections of the lines  $P_{1}\cap P_{2}$, $P_{1}\cap P_{3}$ and $P_{1}\cap P_{4}$ projects to the same point.
The two cases are  different, they where use to construct a birational map between f.i. double octics no. 32 and 69 (see \cite{CK-C}).

The singular fiber of type $I_{4}$ happens, when a plane spanned by two double lines through the point A contains the point B. This geometric property was used in  \cite{CM} to prove modularity of some double octic including two special elements of the family No. 4.

In a similar way the lines \[P_{5}\cap P_{6}, P_{7}\cap P_{8}, P_{5}\cap P_{7}, P_{6}\cap P_{8}, P_{5}\cap P_{8}, P_{6}\cap P_{7}\]
determine singular fibers of the second elliptic surface.

From the above description it follows that we get six types of rational elliptic surfaces (distinguished by the types of singular fibers)
\[
\def\arraycolsep{2mm}
\begin{array}[t]{|c|cccccc|c|}
	\hline
	\def\multicolumn#1#2{}
	\quad S_{1} \quad&I_{0}^{*}&I_{0}^{*}&&&&&{xz(x+t)(x+\lambda t)}\\
	\cline{2-7}
	&\infty&0&&&&&\\
	\hline
	S_{2}&I_{4}&I_{4}&I_{2}&I_{2}&&&{x(x+t)(x+z)(x+z+t)}\\
	\cline{2-7}
	&\infty&0&1&-1&&&\\
	\hline
	S_{3}&I_{0}^{*}&I_{2}&I_{2}&I_{2}&&&{x(x+t)(x+\lambda t)(x+z)}\\
	\cline{2-7}
	&\infty&0&1&\lambda&&&{t(x+\lambda z)(x+z)(x+\lambda t)}\\
	\hline
	S_{4}&I_{2}^{*}&I_{2}&I_{2}&&&&{xt(x+z)(x+t)}\\
	\cline{2-7}
	&\infty&0&1&&&&\\
	\hline
	S_{5}&I_{4}&I_{2}&I_{2}&I_{2}&I_{2}&&{x(x+t)(x+z-\lambda
		t)(x+z)}\\
	\cline{2-7}
	&\infty&0&1&\lambda&\lambda+1&&\\
	\hline
	S_{6}&I_{2}&I_{2}&I_{2}&I_{2}&I_{2}&I_{2}&{x(x+t)(x+z)(x+\frac{1}{\mu-\lambda}
		(z-\lambda t))}\rule[-3mm]{0pt}{8mm}\\
	\cline{2-7}
	&\infty&0&1&\lambda&\mu&\frac{\lambda}{\lambda-\mu+1}&
	x(x+t)(\lambda x+z)((\mu-1)x+z-t)\rule[-3mm]{0pt}{8mm}\\
	\hline
\end{array}
\]\vspace{2mm}

An elliptic surface with singular fibers of types $I_{2}, I_{2}, I_{2}, I_{2}, I_{2}, I_{2}$ or $I_{2}, I_{2}, I_{2}, I_{2}, I_{4}$ is not uniquely
determined by types and position of singular fibers.
In these cases the singular fibers of types $I_{2}$ comes in pairs (coming intersections of disjoint pairs of planes), in particular only a pair of `conjugated'' $I_{2}$ fibers can collapse to an $I_{4}$ fiber.

In the case of surfaces with six $I_{2}$ fibers, the position and coupling of singular fibers
do not determine the surface uniquely. For the two models of the surface $S_{6}$ we get the following position of singular fibers

\[
\begin{array}{c|c|c|c|c|c|c}
u^{2}=\cdots &P_{1}\cap P_{2}&  P_{1}\cap P_{3}&  P_{1}\cap P_{4}&  P_{2}\cap P_{3}&  P_{2}\cap P_{4}&  P_{3}\cap P_{4}\\\hline
 {x(x+t)(x+z)(x+\frac{1}{\mu-\lambda}
 	(z-\lambda t))}&\infty&0&\lambda&1&\mu&\frac\lambda{\lambda-\mu+1} \\\hline
x(x+t)(\lambda x+z)((\mu-1)x+z-t) &\infty&0&1&\lambda&\mu&\frac\lambda{\lambda-\mu+1}
\end{array}\]
Again, we see that these two surfaces differs when considered as double covers of $\PP^{2}$ branched along fours lines, but they are birational as elliptic surfaces. In \cite{CK-C} we used this phenomenon to prove that double octics No. 261 and 264 are birational.

In \cite{CK-C} we gave a more detailed description of the geometry of possible elliptic fibrations. The Mordell-Weil rank of these surfaces equals
\begin{gather*}
	\operatorname{rank}(MW(S_{1})) = 0, \
	\operatorname{rank}(MW(S_{2})) = 0, \
	\operatorname{rank}(MW(S_{3})) = 1, \\
	\operatorname{rank}(MW(S_{4})) = 0, \
	\operatorname{rank}(MW(S_{5})) = 1, \
	\operatorname{rank}(MW(S_{6})) = 2
\end{gather*}
The surface $S_{1}$ is exceptional, all smooth fibers of $S_{1}$ are isomorphic. This surface depends not only on the position of singular fibers but also on the $j$-invariant of generic fiber (equal $256\frac{(1-\lambda+\lambda^{2})^{3}}{(\lambda^{2}(1-\lambda^{2})}$).
Quadratic pull-back of this surface by the map $z\mapsto z^{2}$ is trivial, it is isomorphic to the product $\PP^{1}\times E$, where $E$ is the elliptic curve
\[E:=\{y^{2}=x(x+1)(x+\lambda)\}.\]

Remaining surfaces have the torsion subgroup of the Mordell-Weil group isomorphic to $(\ZZ/2\ZZ)^{2}$. In particular there are exists an isogeny between the surface $S_{2}$ and its pull-back by the map $z\mapsto \frac{z+1}{z-1}$ which swaps fibers of type $I_{2}$ with fibers of type $I_{4}$.

On the other hand the surface $S_{4}$ and its pull-back by the map $z\mapsto 1z$ have isomorphic smooth fiber (and so equal $j$-functions) but they are not isogeneous, on the other hand quadratic pullbacks of both surfaces by the map $z\mapsto z^{2}$ are both isomorphic to $S_{2}$.

Quadratic pull-back of the surface $S_{4}$ ramified along the $I_{2}^{*}$-fiber and a smooth fiber is the surface $S_{5}$. The same holds true for the quadratic pull-back of the surface $S_{3}$ ramified along the fiber $I_{0}^{*}$ and one of the $I_{2}$ fibers, whereas a quadratic pullback of $S_{3}$ ramified along $I_{0}^{*}$ fiber and a smooth fiber is the surface $S_{6}$.

Based on this geometric observations a Kummer fibration gives a lot of extra geometric information about a double octic.
Rational maps between double octic Calabi--Yau threefolds induced by isogenies and pull-backs of elliptic surfaces were used in \cite{CvS} do give a geometric explanation of equality or equivalence of Picard-Fuchs operators for certain families of orphan double octic.
Description of a Kummer fibration and a birational self map that is not induced by projective transformation were key ingredients in the proof of Hilbert modularity of a special element of family no. 250 \cite{CScvS}.

Incidences tables of 238 double octics contain a pair of ``disjoint'' quadruples, there are in total 389 such pairs

\begin{theorem}
	Double octic Calabi--Yau threefolds for Arrangements Nos.: \\[2mm]
	$1$, $2$, $3$, $4$, $5$, $6$, $7$, $8$, $9$, $10$, $11$, $12$, $13$, $14$, $15$, $16$, $17$, $18$, $19$, $20$, $21$, $22$, $23$, $24$, $25$, $26$, $27$, $28$, $29$, $30$, $31$, $32$, $33$, $36$, $38$, $39$, $41$, $42$, $43$, $44$, $49$, $50$, $52$, $53$, $54$, $55$, $56$, $57$, $58$, $59$, $60$, $61$, $62$, $63$, $64$, $65$, $66$, $67$, $68$, $69$, $70$, $73$, $74$, $77$, $78$, $79$, $80$, $84$, $86$, $87$, $89$, $90$, $91$, $92$, $93$, $95$, $96$, $97$, $98$, $100$, $101$, $103$, $104$, $105$, $106$, $108$, $112$, $113$, $114$, $116$, $118$, $119$, $121$, $122$, $123$, $124$, $125$, $126$, $129$, $131$, $132$, $133$, $138$, $140$, $141$, $142$, $144$, $145$, $146$, $147$, $148$, $149$, $150$, $151$, $153$, $159$, $160$, $161$, $167$, $172$, $173$, $174$, $180$, $183$, $184$, $189$, $191$, $193$, $194$, $195$, $196$, $197$, $201$, $208$, $209$, $213$, $216$, $220$, $221$, $222$, $225$, $230$, $234$, $236$, $238$, $239$, $240$, $241$, $242$, $243$, $244$, $245$, $247$, $248$, $249$, $250$, $252$, $253$, $254$, $255$, $256$, $257$, $258$, $259$, $260$, $261$, $264$, $265$, $267$, $268$, $269$, $270$, $271$, $272$, $274$, $275$, $277$, $279$, $280$, $281$, $282$, $284$, $286$, $287$, $289$, $291$, $292$, $296$, $297$, $298$, $299$, $300$, $306$, $307$, $311$, $317$, $320$, $322$, $324$, $327$, $329$, $333$, $334$, $336$, $337$, $341$, $345$, $346$, $347$, $348$, $352$, $356$, $357$, $366$, $372$, $373$, $378$, $382$, $383$, $384$, $385$, $386$, $392$, $393$, $394$, $395$, $401$, $408$, $414$, $420$, $425$, $429$, $437$, $438$, $440$, $448$, $452$, $453$
	are birational to Kummer fibrations.

	The following double octics are birational
$(10, 16)$, $(11, 17)$, $(12, 18)$, $(14, 52)$, $(15, 57, 68)$, $(21, 53)$, $(23, 54)$, $(24, 25)$, $(27, 56)$, $(30, 55)$, $(31, 59)$, $(32, 69)$, $(33, 70)$, $(36.73)$, $(38, 60, 74, 89)$, $(39, 78)$, $(41, 77)$, $(42, 79)$, $(44, 61, 84, 90)$, $(49, 62, 87, 91)$, $(50, 86)$, $(63, 92)$, $(64, 114, 193)$, $(65, 129, 194)$, $(66, 138, 195)$, $(67, 140, 196)$, $(96, 100)$, $(97, 98)$, $(101, 104)$, $(103, 105, 112)$, $(106, 108)$, $(118, 119, 121, 126)$, $(122, 123)$, $(124, 125)$, $(131, 133)$, $(141, 142)$, $(153, 197)$, $(159, 201)$, $(160, 208)$, $(161, 209)$, $(167, 174, 180, 220, 221)$, $(172, 216)$, $(173, 213)$, $(183, 225)$, $(184, 230)$, $(189, 234)$, $(191, 236)$, $(250, 258)$, $(255, 268)$, $(259, 265)$, $(260, 269)$, $(261, 264)$, $(267, 275)$, $(277, 291)$, $(282, 284)$, $(289, 297)$, $(292, 311, 320)$, $(296, 298)$, $(299, 322)$, $(306, 324, 327)$, $(329, 356)$, $(334, 378)$, $(337, 348)$, $(347, 357, 373)$, $(386, 392)$, $(393, 414, 420)$
\end{theorem}
\begin{remark}
	The type of singular fibers for a fiber product need not be the same for all smooth Calabi--Yau threefolds in a family.
\end{remark}

\begin{exmp}
	Consider the arrangement no. 4 given bye the configuration of eight planes
	\[tx(x+t)y(y+t)z(z+t)\left(A_{0}x+A_{1}y-(A_{0}+A_{1})z \right) =0.\]
	Two ``opposite'' fourfold points are given by incidences quadruples
	\[(1,6,7,8), (2,3,4,5).\]
	By the substitution $y=y+\frac{A_{0}}{A_{1}}x$ and reordering we get two elliptic surfaces given by a resolution of singularities of double coverings of the projective plane $\PP^{2}$ branched along unions of four lines
	\[u^{2} = x(x+t)(A_{0}x-A_{1}y)(A_{0}x-A_{1}y-A_{1}t), \quad
	v^{2} = zt(z+t)(A_{1}y-(A_{0}+A_{1})z).\]
	Consequently, this double octic is birational to the Kummer fibration of a fiber product of the following type

	\[\displaystyle
	\begin{array}[t]{ccccc}
		\infty&0&-1&\frac {A_{0}}{A_{1}}&-\frac{A_{0}+A_{1}}{A_{1}}\\[1mm]
		\hline\wl4 &\wl1& \wl1 &-&- \\ \wl2 &\wl1&\wl1&\wl1&\wl1\\[-3mm]
		&\multicolumn{2}{c}{\rule{.5pt}{2mm}\rule{1.3cm}{.5pt}\rule{.5pt}{2mm}}
		&\multicolumn{2}{c}{\rule{.5pt}{2mm}\rule{1.3cm}{.5pt}\rule{.5pt}{2mm}}
	\end{array}
	\]
	This presentation is valid only for $A_{0},A_{1}$ such that $\frac {A_{0}}{A_{1}}\not=\frac{A_{0}+A_{1}}{A_{1}}$, i. e. $A_{1}\not= 0,-2A_{0}$.
	For $A_{1}=0$ we get a singular member of the family. On the other hand for $A_{1}=-2A_{0}$ we get a smooth element of the family, with different type of the Kummer fibration
	\[\displaystyle
	\begin{array}[t]{cccc}
		\infty&0&-1&-\frac 12\\[1mm]
		\hline\wl4 &\wl1& \wl1 &- \\ \wl2 &\wl1&\wl1&\wl2\\[-3mm]
	\end{array}
	\]
	In particular the types of elliptic fibrations are not equal for all elements of a family, and consequently, they cannot be deduced from the incidence table alone.
\end{exmp}

\subsection{K3 fibrations by double sextics}

For any pair $P_{i}, P_{j}$, $i\not=j$ of arrangement planes the planes through the double or triple line $L=P_{i}\cap P_{j}$ determine a K3 fibration of the double octic. Namely, for a plane $P_{(s_{0}:s_{1})}$ defined by the polynomial $s_{0}F_{i}+s_{1}F_{j}$, where $F_{i}$ and $F_{j}$ are equations of planes $P_{i}$ and $P_{j}$ respectively, the fiber is given by the double cover of $P_{s_{0}:s_{1}}$ branched along a union of six lines
\[l_{k}:=P_{(s_{0}:s_{1})}\cap P_{k}, \quad k\not\in\{i,j\}.\]
Double cover of the projective plane branched along a smooth curve of degree 6 is a K3 surface. If the branch curve is singular (possibly reducible) but has at worst Kleinian singularities, then the double cover has also Kleinian singularities (of the same type) and its minimal resolution of singularities is a K3 surface.
This assumption is satisfied in particular, when the branch curve is a union of six lines, with at most three intersecting at a point (for details see  \cite{Persson}).
Since for a generic choice of $(s_{0}:s_{1})\in\PP^{1}$ the lines $l_{k}$ are all distinct and no four of them intersect, the double cover has a resolution of singularities which is a K3 surface.

To get an explicit equation of the K3 fibration assume that $F_{7}=z$, $F_{8}=t$ and the line $L$ is given by $z=t=0$. For any $s\in\CC$ the plane $P_{(-s:1)}$ is given by the equation $t=sz$.

\begin{theorem}
	A double octic Calabi--Yau threefold defined by the octic arrangement
	\[D:=\{(x,y,z,t)\in\PP^{3}:P_{1}(x,y,z,t)\cdots P_{6}(x,y,z,t)zt=0\}\]
	is birational to the K3 fibration
	\[u^{2}=F_{1}(x,y,z,sz) \cdots F_{6}(x,y,z,sz)\]
\end{theorem}
The fiber at infinity is the double sextic
\[u^{2}=F_{1}(x,y,0,t)\cdots F_{6}(x,y,0,t).\]

Arrangement double and triple lines disjoint from $L$ yield double and triple points of the
branch curve in a generic fiber, similarly - fourfold and fivefold points lying on $L$. The special fibers depend on double and triple lines intersecting $L$ and fourfold and fivefold points not lying on $L$.
Explicit description of general and singular fibers was given by P. Borówka, who studied in detail the case of rigid double octics  (\cite{Bor}).

\subsection{K3 fibrations by double quadrics}

A pair of skew arrangement double or triple lines determines a K3 fibration of a double octic. The lines $L_{ij}=P_{i}\cap P_{j}$ and $L_{kl}=P_{k}\cap P_{l}$, $\#\{i,j,k,l\}=4$, are skew exactly when $P_{i}\cap P_{j}\cap P_{k}\cap P_{l}=\emptyset$. The base locus of the quadric pencil $Q_{(\lambda_{0}:\lambda_{1})}$, $(\lambda_{0}:\lambda_{1})\in\PP^{1}$ defined by quadratic forms $\lambda_{0}F_{i}F_{k}-\lambda_{1}F_{j}F_{l}$
is the union of four lines $L_{ij}$, $L_{il}$, $L_{kj}$ and $L_{kl}$. In the resolution process we blow-up these lines, consequently we get a fibration on the double octic. The generic fiber $S_{(\lambda_{0}:\lambda_{1})}$ is a resolution of the double cover of the quadric $G_{\lambda_{0}:\lambda_{1}}$ branched along the intersection of $Q_{(\lambda_{0}:\lambda_{1})}$ with the octic arrangement.
The intersection equals
\[S_{(\lambda_{0}:\lambda_{1})}\cdot D = 2L_{ij}+2L_{il}+2L_{kj}+2L_{kl}+\sum_{m\not\in\{i,j,k,l\}} Q_{(\lambda_{0}:\lambda_{1})}\cap P_{m}.\]
Consequently, the branch locus
\[
S_{(\lambda_{0}:\lambda_{1})}\cdot D - (2L_{ij}+2L_{il}+2L_{kj}+2L_{kl}) = \sum_{m\not\in\{i,j,k,l\}} Q_{(\lambda_{0}:\lambda_{1})}\cap P_{m}.
\]
is a curve of bi-degree $(4,4)$, while the canonical divisor of a quadric has bi-degree $(-2,-2)$. Consequently, the resolution of the double cover of $Q_{(\lambda_{0}:\lambda_{1})}$ is a K3 surface.

The quadric $Q_{(\lambda_{0}:\lambda_{1})}$ is isomorphic to $\PP^{1}\times \PP^{1}$, an explicit isomorphism can be given by letting

\[\lambda_{0} p_{0}q_{0}=F_{i}, \ \ \lambda_{1} p_{1}q_{1}=F_{k}, \ \ \lambda_{0}p_{0}q_{1}=F_{j}, \ \ \lambda_{0}p_{1}q_{0}=F_{l}.\]
The assumption that lines $L_{ij}$ and $L_{kl}$ are disjoint implies that the linear map $(F_{i},F_{j},F_{k},F_{l})$ is an isomorphism. Consequently, the above system of equations has a unique solution in $x,y,z,t$. Substituting the solution into double octic, yields a family of double quadrics.
To simplify computations we assume that the double lines equal  $x=y=0$, $z=t=0$ and put $\lambda_{0}=1, \lambda_{1}=\lambda$.

\begin{theorem}[\mbox{\cite{CK-C2}[Thm.~4.3]}]
	The double octic
	\[u^{2} = xyztF_{5}(x,y,z,t)\dots F_{8}(x,y,z,t)\]
	is birational to a double quadrics fibration
	\[u^{2} = \lambda F_{5}(p_{0}q_{0}, p_{0}q_{1}, p_{1}q_{0}, \lambda p_{1}q_{1}) \cdots F_{8}(p_{0}q_{0}, p_{0}q_{1}, p_{1}q_{0}, \lambda p_{1}q_{1}).\]
\end{theorem}
In fact four non-intersecting arrangement planes define three pairs of disjoint arrangement double or triple lines, they are base loci of the following pencils of quadrics
\[\lambda_{0}xt-\lambda_{1}yz=0, \quad \lambda_{0}xz-\lambda_{1}yt=0, \quad\text{and} \quad \lambda_{0}xy-\lambda_{1}zt=0.\]
Consequently, the double octic
	\[u^{2} = xyztF_{5}(x,y,z,t)\dots F_{8}(x,y,z,t)\]
is birational also to the following fibrations
\[u^{2} = \lambda F_{5}(p_{0}q_{0}, p_{0}q_{1}, \lambda p_{1}q_{1}, p_{1}q_{0}) \cdots F_{8}(p_{0}q_{0}, p_{0}q_{1}, \lambda p_{1}q_{1}p_{1}q_{0})\]
and
\[u^{2} = \lambda F_{5}(p_{0}q_{0}, \lambda p_{1}q_{1}, p_{0}q_{1}, p_{1}q_{0}) \cdots F_{8}(p_{0}q_{0}, \lambda p_{1}q_{1}, p_{0}q_{1}, p_{1}q_{0}).\]
If moreover one of the two skew-lines is an arrangement triple line, we get 9 fibrations, if both lines have multiplicity 3, we get 27 fibrations.

Fibrations by double quadrics are  more difficult to classify then fibrations by Kummer surfaces or double sextics. A basic invariant of a family of K3 surfaces is given by the Picard-Fuchs operator, which is a differential operator annihilating the period integral (\cite{CvS}). Unfortunately, the Picard-Fuchs operators of families of K3 surfaces do not have as good properties as in the case of Calabi--Yau threefolds. In particular, we do not have a nice classification of the possible Riemann symbols. Nevertheless, a birational isomorphism of two fibrations results in equal (up to a pullback) Picard-Fuchs operators. The opposite implication also is expected to hold true.

A singular point of the branch divisor determines an elliptic fibration on the double quadric. Checking a fibration preserving isomorphism of two elliptically fibered K3 surfaces is an elementary linear algebra.
In \cite{CK-C2} we use the above simple observation to find a birational transformation between the following families of double octic.

        \begin{eqnarray*}
	X^{35}_{A_{0}:A_{1}}&:&\
	u^{2}=(x - y)xy(y + t)(x + t)z\left(A_{0}x + A_{1}y + (-A_{0} + A_{1})z + A_{1}t\right)(z + t)\\
	X^{71}_{A_{0}:A_{1}}&:&\
	u^{2}=xy(x + y)\left(A_{0}x + A_{1}y + A_{1}z\right)zt
	\left((-A_{0} + A_{1})x + A_{1}y + A_{1}t\right)(x + y + z + t)
\end{eqnarray*}

\begin{theorem}
	The following map is birational
	\begin{multline*}
		\mathcal{X}^{35}\ni(x,y,z,t,u,A_{0},A_{1}) \longmapsto
		\Bigl(
		-(y-z)(x-y)x(A_{0}-A_{1}),\\ -(x-y)((A_{0}-A_{1})xz-(A_{0}-A_{1})yz+A_{1}ty+A_{1}y^2),\\ (y-z)(x-y)(A_{0}x-(A_{0}-A_{1})z+A_{1}t+A_{1}y),\\ (x+t)(A_{0}xz-A_{0}yz+A_{1}ty-A_{1}xy+A_{1}
		y^2+A_{1}yz),\\ (A_{0}-A_{1})\left(\frac{A_{0}-A_{1}}{A_{0}}\right)^{1/2}
		(A_{0}x-A_{0}z+A_{1}t-A_{1}x+A_{1}y+A_{1}z)\\
		(x-y)^2(y-z)(A_{0}xz-A_{0}yz+A_{1}ty-A_{1}xz+A_{1}y^2+A_{1}yz)\\
		(A_{0}xz-A_{0}yz+A_{1}ty-A_{1}xy+A_{1}y^2+A_{1}yz), A_{0}, A_{0}-A_{1}
		\Bigr)\in\mathcal{X}^{71}
	\end{multline*}
	The inverse birational map is given by
	\begin{multline*}
		(x,y,z,t,u,A_{0},A_{1}) \longmapsto
		\Bigl(
		x\bigl(A_1(A_1-A_0)(x^2+xz+tx+ty)+\\(A_{0}^2-3A_0A_1+3A_1^2)xy+A_1(2A_1-A_{0})y^2+A_{1}^2yz\bigr),\\
		-(x+y)A_{1}((A_0-A_1)(x^2+xz+tx)-A_{1}(xy+yz)),\\ -(A_{0}x+A_{1}y+A_{1}z)((A_{0}-A_1)(x^2+xz+tx)-A_{1}(xy+yz)),\\
		A_1(A_0-A_1)(x^3+x^2z+tx^2+2txy)+(3A_{0}A_{1}-A_{0}^2-3A_{1}^2)x^2y+\\
		A_1(A_{0}-2A_1)xy^2+
		A_{1}^2(tyz-xyz), \ A_{0}, \ A_{0}-A_{1},\\
		A_{1}\sqrt{A_{0}A_{1}}y\bigl((A_{0}-A_{1})x+A_{1}z\bigr)^2
		((A_{0}-A_1)x(x+z+t)-A_{1}(x+z)y)\times\\
		\bigl(A_1(A_1-A_{0})(x^2+xz+tx+ty)+(A_{0}^2-3A_0A_1+3A_1^2)xy+A_1(2A_{1}-A_{0})y^2+\\A_{1}^2yz\bigr)(A_{0}x-A_{1}x-A_{1}y)
		\Bigr)
	\end{multline*}
\end{theorem}
\begin{proof}
	This theorem can be proved by direct computations in Maple, in \cite{CK-C2} we instead described in full detail how the above formulas were derived.
\end{proof}

not previously observed in the Calabi--Yau setting.not previously observed in the Calabi--Yau setting.not previously observed in the Calabi--Yau setting.not previously observed in the Calabi--Yau setting.not previously observed in the Calabi--Yau setting.not previously observed in the Calabi--Yau setting.not previously observed in the Calabi--Yau setting.not previously observed in the Calabi--Yau setting.not previously observed in the Calabi--Yau setting.not previously observed in the Calabi--Yau setting.not previously observed in the Calabi--Yau setting.not previously observed in the Calabi--Yau setting.not previously observed in the Calabi--Yau setting.not previously observed in the Calabi--Yau setting.and $G_{i}$ is a constant. reordering if necessary the order of planes $P_{i}$, we can assume that $G_{1}=\dots=G_{4}=0$ and $G_{5}=0$ when $A$ is a fivefold point and $G_{5}\not=0$ when $A$ is fourfold.
\section{Equations and Data}\label{sec:EqsData}
In this section, we collect equations of octic arrangements resulting from the classification. To keep the length of the paper reasonable we do not present here other details, instead we include the magma code that enables easy computation of all data.
\subsection{Equations}
\leavevmode\small
\medskip

\medskip

\noindent\parbox[right]{9mm}{\textbf{1:}}\parbox{3mm}{\hfill}%
\parbox[t]{15cm}{\((x - y)xy(x + t)(y + t)(x - z)z(z + t)\)}

\noindent\parbox[right]{9mm}{\textbf{2:}}\parbox{3mm}{\hfill}%
\parbox[t]{15cm}{\(xy(A_{0}x + A_{1}y)(x + z)z(y + t)t(x + y + z + t)\)}

\noindent\parbox[right]{9mm}{\textbf{3:}}\parbox{3mm}{\hfill}%
\parbox[t]{15cm}{\(tx(x + t)y(y + t)z(z + t)(x + y - z)\)}

\noindent\parbox[right]{9mm}{\textbf{4:}}\parbox{3mm}{\hfill}%
\parbox[t]{15cm}{\(tx(x + t)y(y + t)z(z + t)(A_{0}x + A_{1}y + (-A_{0} - A_{1})z)\)}

\noindent\parbox[right]{9mm}{\textbf{5:}}\parbox{3mm}{\hfill}%
\parbox[t]{15cm}{\(tx(x + t)y(y + t)z(z + t)(A_{0}x + A_{1}y - A_{1}z)\)}

\noindent\parbox[right]{9mm}{\textbf{6:}}\parbox{3mm}{\hfill}%
\parbox[t]{15cm}{\(tx(x + t)y(y + t)z(z + t)(A_{0}x + A_{1}y + A_{2}z)\)}

\noindent\parbox[right]{9mm}{\textbf{7:}}\parbox{3mm}{\hfill}%
\parbox[t]{15cm}{\(xy(A_{0}x + A_{1}y)(A_{0}x + A_{2}z)zt(A_{0}x + A_{3}t)(x + y + z + t)\)}

\noindent\parbox[right]{9mm}{\textbf{8:}}\parbox{3mm}{\hfill}%
\parbox[t]{15cm}{\(x(x - y)yz(A_{0}x + A_{1}z)(x + t)(y + t)(z + t)\)}

\noindent\parbox[right]{9mm}{\textbf{9:}}\parbox{3mm}{\hfill}%
\parbox[t]{15cm}{\(xy(A_{0}x + A_{1}y)z(A_{0}x + A_{2}z)(y + t)t(x + y + z + t)\)}

\noindent\parbox[right]{9mm}{\textbf{10:}}\parbox{3mm}{\hfill}%
\parbox[t]{15cm}{\(xy(x + y)z(A_{0}x + A_{1}z)t(A_{0}y + A_{1}t)(x + y + z + t)\)}

\noindent\parbox[right]{9mm}{\textbf{11:}}\parbox{3mm}{\hfill}%
\parbox[t]{15cm}{\(xy(x + y)z(A_{0}x + A_{1}z)t(A_{0}y + A_{2}t)(x + y + z + t)\)}

\noindent\parbox[right]{9mm}{\textbf{12:}}\parbox{3mm}{\hfill}%
\parbox[t]{15cm}{\(xy(A_{0}x + A_{1}y)(A_{0}x + A_{2}z)zt(A_{1}y + A_{3}t)(x + y + z + t)\)}

\noindent\parbox[right]{9mm}{\textbf{13:}}\parbox{3mm}{\hfill}%
\parbox[t]{15cm}{\(x(A_{0}x + A_{1}t)(x + t)y(y + t)z(z + t)(y - z)\)}

\noindent\parbox[right]{9mm}{\textbf{14:}}\parbox{3mm}{\hfill}%
\parbox[t]{15cm}{\(xy(A_{0}x + A_{1}y)(x + z)zt(x + y + z + t)(A_{0}x + A_{0}y + A_{0}z + A_{2}t)\)}

\noindent\parbox[right]{9mm}{\textbf{15:}}\parbox{3mm}{\hfill}%
\parbox[t]{15cm}{\(xy(A_{0}x + A_{1}y)z(A_{0}x + A_{2}z)(A_{0}x + A_{0}y + A_{0}z + A_{3}t)t(x + y + z + t)\)}

\noindent\parbox[right]{9mm}{\textbf{16:}}\parbox{3mm}{\hfill}%
\parbox[t]{15cm}{\((x - y)xy(x + t)(y + t)(A_{0}x + A_{1}y + (-A_{0} - A_{1})z)z(z + t)\)}

\noindent\parbox[right]{9mm}{\textbf{17:}}\parbox{3mm}{\hfill}%
\parbox[t]{15cm}{\(xy(A_{0}x + A_{1}y)z(A_{0}x + A_{2}z)(x + t)(y + t)(z + t)\)}

\noindent\parbox[right]{9mm}{\textbf{18:}}\parbox{3mm}{\hfill}%
\parbox[t]{15cm}{\(x(A_{0}x + A_{1}y)yz(A_{0}x + A_{2}z)(A_{0}x + A_{3}y + A_{3}t)t(x + y + z + t)\)}

\noindent\parbox[right]{9mm}{\textbf{19:}}\parbox{3mm}{\hfill}%
\parbox[t]{15cm}{\((x - y)xy(x + t)(y + t)z(x + y - z)(z + t)\)}

\noindent\parbox[right]{9mm}{\textbf{20:}}\parbox{3mm}{\hfill}%
\parbox[t]{15cm}{\((x - y)xy(y + t)(x + t)(A_{0}x + A_{1}y - A_{1}z)z(z + t)\)}

\noindent\parbox[right]{9mm}{\textbf{21:}}\parbox{3mm}{\hfill}%
\parbox[t]{15cm}{\(xy(x + y)z(A_{0}x + A_{1}z)((-A_{0} + A_{1})x + A_{1}y + (-A_{0} + A_{1})t)t(x + y + z + t)\)}

\noindent\parbox[right]{9mm}{\textbf{22:}}\parbox{3mm}{\hfill}%
\parbox[t]{15cm}{\((A_{0}x + A_{1}y)xyz(A_{0}x + A_{1}y + A_{2}z)(x + t)(y + t)(z + t)\)}

\noindent\parbox[right]{9mm}{\textbf{23:}}\parbox{3mm}{\hfill}%
\parbox[t]{15cm}{\(xy(A_{0}x + A_{1}y)(x + z)zt(A_{0}x + A_{0}y + A_{2}t)(x + y + z + t)\)}

\noindent\parbox[right]{9mm}{\textbf{24:}}\parbox{3mm}{\hfill}%
\parbox[t]{15cm}{\(xy(A_{0}x + A_{1}y)z(x + z)(A_{0}x + A_{2}y + A_{0}t)t(x + y + z + t)\)}

\noindent\parbox[right]{9mm}{\textbf{25:}}\parbox{3mm}{\hfill}%
\parbox[t]{15cm}{\((x - y)xy(x + t)(y + t)z(A_{0}x + A_{1}y + A_{2}z)(z + t)\)}

\noindent\parbox[right]{9mm}{\textbf{26:}}\parbox{3mm}{\hfill}%
\parbox[t]{15cm}{\(xy(A_{0}x + A_{1}y)z((-A_{0} + A_{1})x + A_{1}z)(A_{0}x + A_{2}y + A_{0}t)t(x + y + z + t)\)}

\noindent\parbox[right]{9mm}{\textbf{27:}}\parbox{3mm}{\hfill}%
\parbox[t]{15cm}{\(xy(A_{0}x + A_{1}y)((-A_{0} + A_{2})x + A_{2}z)zt(A_{0}x + A_{2}y + A_{3}t)(x + y + z + t)\)}

\noindent\parbox[right]{9mm}{\textbf{28:}}\parbox{3mm}{\hfill}%
\parbox[t]{15cm}{\(xy(A_{0}x + A_{1}y)((-A_{2} + A_{3})x + A_{3}z)z(A_{2}x + A_{1}y + A_{3}t)t(x + y + z + t)\)}

\noindent\parbox[right]{9mm}{\textbf{29:}}\parbox{3mm}{\hfill}%
\parbox[t]{15cm}{\(xy(A_{0}x + A_{1}y)(A_{0}x + A_{2}z)zt(A_{0}x + A_{3}y + A_{4}t)(x + y + z + t)\)}

\noindent\parbox[right]{9mm}{\textbf{30:}}\parbox{3mm}{\hfill}%
\parbox[t]{15cm}{\((x - y)xy(x + t)(y + t)z(z + t)(A_{0}x + A_{1}z + A_{2}t)\)}

\noindent\parbox[right]{9mm}{\textbf{31:}}\parbox{3mm}{\hfill}%
\parbox[t]{15cm}{\(xy(A_{0}x + A_{1}y)(x + z)z(A_{0}x + A_{2}y + A_{0}z + A_{3}t)t(x + y + z + t)\)}

\noindent\parbox[right]{9mm}{\textbf{32:}}\parbox{3mm}{\hfill}%
\parbox[t]{15cm}{\((x - y)xy(y + t)(x + t)z(x - y - 2z - t)(z + t)\)}

\noindent\parbox[right]{9mm}{\textbf{33:}}\parbox{3mm}{\hfill}%
\parbox[t]{15cm}{\((x - y)xy(y + t)(x + t)z(A_{0}x - A_{0}y + A_{1}z - A_{0}t)(z + t)\)}

\noindent\parbox[right]{9mm}{\textbf{34:}}\parbox{3mm}{\hfill}%
\parbox[t]{15cm}{\((x - y)xy(y + t)(x + t)z(A_{0}x + A_{0}y + A_{1}z + A_{0}t)(z + t)\)}

\noindent\parbox[right]{9mm}{\textbf{35:}}\parbox{3mm}{\hfill}%
\parbox[t]{15cm}{\((x - y)xy(y + t)(x + t)z(A_{0}x + A_{1}y + (-A_{0} + A_{1})z + A_{1}t)(z + t)\)}

\noindent\parbox[right]{9mm}{\textbf{36:}}\parbox{3mm}{\hfill}%
\parbox[t]{15cm}{\(xy(A_{0}x + A_{1}y)z(x + z)t((A_{0} - A_{1})y + (A_{0} - A_{1})z + A_{0}t)(x + y + z + t)\)}

\noindent\parbox[right]{9mm}{\textbf{37:}}\parbox{3mm}{\hfill}%
\parbox[t]{15cm}{\(xy(A_{0}x + (A_{0} + A_{1})y)z(x + z)t(A_{1}y + A_{2}z - A_{0}t)(x + y + z + t)\)}

\noindent\parbox[right]{9mm}{\textbf{38:}}\parbox{3mm}{\hfill}%
\parbox[t]{15cm}{\(xy(A_{0}x + A_{1}y)z(x + z)t((A_{0} - A_{1})y + A_{0}z + A_{2}t)(x + y + z + t)\)}

\noindent\parbox[right]{9mm}{\textbf{39:}}\parbox{3mm}{\hfill}%
\parbox[t]{15cm}{\(xy(A_{0}x + A_{1}y)z(x + z)t(A_{0}y + A_{0}z + A_{2}t)(x + y + z + t)\)}

\noindent\parbox[right]{9mm}{\textbf{40:}}\parbox{3mm}{\hfill}%
\parbox[t]{15cm}{\((x - y)xy(y + t)(x + t)z(A_{0}x + A_{1}y + A_{2}z + A_{1}t)(z + t)\)}

\noindent\parbox[right]{9mm}{\textbf{41:}}\parbox{3mm}{\hfill}%
\parbox[t]{15cm}{\(xy(A_{0}x + A_{1}y)(x + z)zt(x + y + z + t)(A_{0}x + A_{1}y + A_{1}z + A_{2}t)\)}

\noindent\parbox[right]{9mm}{\textbf{42:}}\parbox{3mm}{\hfill}%
\parbox[t]{15cm}{\((x - y)xy(x + t)(y + t)z(z + t)(A_{0}x - A_{0}y + A_{1}z + A_{2}t)\)}

\noindent\parbox[right]{9mm}{\textbf{43:}}\parbox{3mm}{\hfill}%
\parbox[t]{15cm}{\(xy(A_{0}x + A_{1}y)(x + z)zt(x + y + z + t)(A_{0}x + A_{2}y + A_{2}z + A_{3}t)\)}

\noindent\parbox[right]{9mm}{\textbf{44:}}\parbox{3mm}{\hfill}%
\parbox[t]{15cm}{\(xy((A_{0} - A_{2})x + (A_{1} - A_{2})y)(x + z)zt(x + y + z + t)(A_{0}x + A_{1}y + A_{2}z + A_{3}t)\)}

\noindent\parbox[right]{9mm}{\textbf{45:}}\parbox{3mm}{\hfill}%
\parbox[t]{15cm}{\(xy(x + y)(x + z)zt(x + y + z + t)(A_{0}x + A_{1}y + A_{2}z + A_{3}t)\)}

\noindent\parbox[right]{9mm}{\textbf{46:}}\parbox{3mm}{\hfill}%
\parbox[t]{15cm}{\(xy(A_{0}x + A_{1}y)(x + z)zt(x + y + z + t)(A_{1}y + A_{2}z + A_{3}t)\)}

\noindent\parbox[right]{9mm}{\textbf{47:}}\parbox{3mm}{\hfill}%
\parbox[t]{15cm}{\(xy(A_{0}x + A_{1}y)(x + z)zt(x + y + z + t)(A_{0}x + A_{1}y + A_{2}z + A_{3}t)\)}

\noindent\parbox[right]{9mm}{\textbf{48:}}\parbox{3mm}{\hfill}%
\parbox[t]{15cm}{\(xy(A_{0}x + A_{1}y)(x + z)zt(x + y + z + t)(A_{0}x + A_{2}y + A_{3}z + A_{4}t)\)}

\noindent\parbox[right]{9mm}{\textbf{49:}}\parbox{3mm}{\hfill}%
\parbox[t]{15cm}{\(xy(A_{0}x + A_{1}y)z(A_{0}x + A_{2}z)(A_{0}x + A_{3}y + A_{0}z + A_{4}t)t(x + y + z + t)\)}

\noindent\parbox[right]{9mm}{\textbf{50:}}\parbox{3mm}{\hfill}%
\parbox[t]{15cm}{\(xy(A_{0}x + A_{1}y)z(A_{0}x + A_{2}z)(A_{0}x + A_{3}y + A_{3}z + A_{4}t)t(x + y + z + t)\)}

\noindent\parbox[right]{9mm}{\textbf{51:}}\parbox{3mm}{\hfill}%
\parbox[t]{15cm}{\(xy(A_{0}x + A_{1}y)z(A_{0}x + A_{2}z)(A_{0}x + A_{3}y + A_{4}z + A_{5}t)t(x + y + z + t)\)}

\noindent\parbox[right]{9mm}{\textbf{52:}}\parbox{3mm}{\hfill}%
\parbox[t]{15cm}{\((A_{0}x + A_{1}t)x(x + t)y(y + t)z(z + t)(A_{0}y + A_{2}z)\)}

\noindent\parbox[right]{9mm}{\textbf{53:}}\parbox{3mm}{\hfill}%
\parbox[t]{15cm}{\(xy(x + y)(A_{0}x + A_{1}y + A_{1}z)z(A_{0}x + A_{1}y + A_{1}z + A_{0}t)t(x + y + z + t)\)}

\noindent\parbox[right]{9mm}{\textbf{54:}}\parbox{3mm}{\hfill}%
\parbox[t]{15cm}{\(xy(A_{0}x + A_{1}y)(x + t)(y + t)z(A_{0}x + A_{2}z + A_{0}t)(z + t)\)}

\noindent\parbox[right]{9mm}{\textbf{55:}}\parbox{3mm}{\hfill}%
\parbox[t]{15cm}{\(xy(A_{0}x + A_{1}y)z(A_{0}x + A_{2}y + A_{2}z)(x + y + t)t(x + y + z + t)\)}

\noindent\parbox[right]{9mm}{\textbf{56:}}\parbox{3mm}{\hfill}%
\parbox[t]{15cm}{\(xy(A_{0}x + A_{1}y)(A_{0}x + A_{2}y + A_{2}z)z(A_{0}x + A_{2}y + A_{2}z + A_{3}t)t(x + y + z + t)\)}

\noindent\parbox[right]{9mm}{\textbf{57:}}\parbox{3mm}{\hfill}%
\parbox[t]{15cm}{\(xy(A_{0}x + A_{1}y)(A_{0}x + A_{2}y + A_{2}z)z(A_{0}x + A_{0}y + A_{0}z + A_{3}t)t(x + y + z + t)\)}

\noindent\parbox[right]{9mm}{\textbf{58:}}\parbox{3mm}{\hfill}%
\parbox[t]{15cm}{\(xy(A_{0}x + A_{1}y)z(A_{0}x + A_{0}y + A_{2}z)(A_{0}x + A_{0}y + A_{0}z + A_{3}t)t(x + y + z + t)\)}

\noindent\parbox[right]{9mm}{\textbf{59:}}\parbox{3mm}{\hfill}%
\parbox[t]{15cm}{\(xy(A_{0}x + A_{1}y)(A_{0}x + A_{3}y + A_{2}z)z((-A_{0} + A_{2})x + (A_{2} - A_{3})y + A_{2}t)t(x + y + z + t)\)}

\noindent\parbox[right]{9mm}{\textbf{60:}}\parbox{3mm}{\hfill}%
\parbox[t]{15cm}{\(xy(A_{0}x + A_{1}y)z(A_{0}x + A_{2}y + A_{2}z)t((A_{0} - A_{2})z + A_{0}t)(x + y + z + t)\)}

\noindent\parbox[right]{9mm}{\textbf{61:}}\parbox{3mm}{\hfill}%
\parbox[t]{15cm}{\(xy(A_{0}x + A_{1}y)(x + y + z)zt(x + y + z + t)(A_{0}x + A_{2}z + A_{3}t)\)}

\noindent\parbox[right]{9mm}{\textbf{62:}}\parbox{3mm}{\hfill}%
\parbox[t]{15cm}{\(xy(A_{0}x + A_{1}y)(x + y + z)zt(x + y + z + t)(A_{0}x + A_{2}y + A_{3}z + A_{4}t)\)}

\noindent\parbox[right]{9mm}{\textbf{63:}}\parbox{3mm}{\hfill}%
\parbox[t]{15cm}{\(xy(A_{0}x + A_{1}y)z(A_{0}x + A_{2}y + A_{3}z)(A_{0}x + A_{0}y + A_{0}z + A_{4}t)t(x + y + z + t)\)}

\noindent\parbox[right]{9mm}{\textbf{64:}}\parbox{3mm}{\hfill}%
\parbox[t]{15cm}{\(xy(A_{0}x + A_{1}y)zt(A_{0}z + A_{1}t)(x + y + z + t)(A_{0}x + A_{2}y + A_{0}z + A_{2}t)\)}

\noindent\parbox[right]{9mm}{\textbf{65:}}\parbox{3mm}{\hfill}%
\parbox[t]{15cm}{\(xy(A_{0}x + A_{1}y)zt(A_{0}z + A_{2}t)(x + y + z + t)(A_{0}x + A_{3}y + A_{0}z + A_{3}t)\)}

\noindent\parbox[right]{9mm}{\textbf{66:}}\parbox{3mm}{\hfill}%
\parbox[t]{15cm}{\(xy(A_{0}x + A_{1}y)z(A_{2}z + A_{3}t)t(x + y + z + t)(A_{0}x + A_{4}y + A_{2}z + A_{4}t)\)}

\noindent\parbox[right]{9mm}{\textbf{67:}}\parbox{3mm}{\hfill}%
\parbox[t]{15cm}{\(xy(A_{0}x + A_{1}y)z(A_{2}z + A_{3}t)t(x + y + z + t)(A_{0}x + A_{4}y + A_{2}z + A_{5}t)\)}

\noindent\parbox[right]{9mm}{\textbf{68:}}\parbox{3mm}{\hfill}%
\parbox[t]{15cm}{\(xy(A_{0}x + A_{1}y)(A_{0}x + A_{0}y + A_{2}z)zt(A_{0}x + A_{0}y + A_{3}t)(x + y + z + t)\)}

\noindent\parbox[right]{9mm}{\textbf{69:}}\parbox{3mm}{\hfill}%
\parbox[t]{15cm}{\(xy(-2x + y)(y + t)(x + t)(x - y + z)z(z + t)\)}

\noindent\parbox[right]{9mm}{\textbf{70:}}\parbox{3mm}{\hfill}%
\parbox[t]{15cm}{\(x(x + y - z)(A_{0}x + A_{1}y - A_{1}z)yz(y + t)(z + t)(x + t)\)}

\noindent\parbox[right]{9mm}{\textbf{71:}}\parbox{3mm}{\hfill}%
\parbox[t]{15cm}{\(xy(x + y)(A_{0}x + A_{1}y + A_{1}z)zt((-A_{0} + A_{1})x + A_{1}y + A_{1}t)(x + y + z + t)\)}

\noindent\parbox[right]{9mm}{\textbf{72:}}\parbox{3mm}{\hfill}%
\parbox[t]{15cm}{\(xy(A_{0}x + A_{1}y)(y + t)(x + t)(x + y - z)z(z + t)\)}

\noindent\parbox[right]{9mm}{\textbf{73:}}\parbox{3mm}{\hfill}%
\parbox[t]{15cm}{\(x(A_{0}x + A_{1}t)(x + t)y(y + t)z(z + t)(-A_{0}x + A_{0}y + A_{1}z)\)}

\noindent\parbox[right]{9mm}{\textbf{74:}}\parbox{3mm}{\hfill}%
\parbox[t]{15cm}{\(xy(A_{0}x + A_{1}y)z(A_{0}x + A_{2}y + (-A_{0} - A_{2})z)(x + t)(y + t)(z + t)\)}

\noindent\parbox[right]{9mm}{\textbf{75:}}\parbox{3mm}{\hfill}%
\parbox[t]{15cm}{\(xy(x + y)(A_{0}x + A_{1}y + A_{1}z)zt((-A_{0} + A_{1})x + A_{2}y + A_{1}t)(x + y + z + t)\)}

\noindent\parbox[right]{9mm}{\textbf{76:}}\parbox{3mm}{\hfill}%
\parbox[t]{15cm}{\(xy(A_{0}x + A_{1}y)(y + t)(x + t)(A_{0}x + A_{2}y - A_{2}z)z(z + t)\)}

\noindent\parbox[right]{9mm}{\textbf{77:}}\parbox{3mm}{\hfill}%
\parbox[t]{15cm}{\(x(x + t)(A_{0}x + A_{1}t)y(y + t)z(z + t)(A_{0}x + A_{2}y + (-A_{0} - A_{2})z)\)}

\noindent\parbox[right]{9mm}{\textbf{78:}}\parbox{3mm}{\hfill}%
\parbox[t]{15cm}{\(xy(A_{0}x + A_{1}y)z(A_{0}x - A_{0}y + A_{2}z)(x + t)(y + t)(z + t)\)}

\noindent\parbox[right]{9mm}{\textbf{79:}}\parbox{3mm}{\hfill}%
\parbox[t]{15cm}{\(x(A_{0}x + A_{1}y - A_{1}z)(A_{0}x + A_{2}y - A_{2}z)yz(y + t)(z + t)(x + t)\)}

\noindent\parbox[right]{9mm}{\textbf{80:}}\parbox{3mm}{\hfill}%
\parbox[t]{15cm}{\(xy(A_{0}x + A_{1}y)(A_{2}x + A_{3}y + 2A_{3}z)z(A_{2}x + A_{3}y + 2A_{3}t)t(x + y + z + t)\)}

\noindent\parbox[right]{9mm}{\textbf{81:}}\parbox{3mm}{\hfill}%
\parbox[t]{15cm}{\(x(x + y)yz(A_{0}x + A_{1}y + A_{2}z)(A_{0}x + A_{3}y + A_{3}t)t(x + y + z + t)\)}

\noindent\parbox[right]{9mm}{\textbf{82:}}\parbox{3mm}{\hfill}%
\parbox[t]{15cm}{\(x((A_{0} - A_{2})x + (A_{1} - A_{2})y)yz(A_{0}x + A_{1}y + A_{2}z)(A_{0}x + A_{3}y + A_{3}t)t(x + y + z + t)\)}

\noindent\parbox[right]{9mm}{\textbf{83:}}\parbox{3mm}{\hfill}%
\parbox[t]{15cm}{\(x(A_{0}x + A_{1}y)yz(A_{0}x + A_{2}y + A_{2}z)(A_{0}x + A_{3}y + A_{3}t)t(x + y + z + t)\)}

\noindent\parbox[right]{9mm}{\textbf{84:}}\parbox{3mm}{\hfill}%
\parbox[t]{15cm}{\(xy(A_{0}x + A_{1}y)((-A_{2} + A_{3})x + (-A_{2} + A_{3})y + A_{3}z)z(A_{0}x + A_{2}y + A_{3}t)t(x + y + z + t)\)}

\noindent\parbox[right]{9mm}{\textbf{85:}}\parbox{3mm}{\hfill}%
\parbox[t]{15cm}{\(xy(A_{0}x + A_{1}y)(A_{0}x + A_{2}y + A_{3}z)z(A_{4}x + (-A_{2} + A_{3})y + A_{3}t)t(x + y + z + t)\)}

\noindent\parbox[right]{9mm}{\textbf{86:}}\parbox{3mm}{\hfill}%
\parbox[t]{15cm}{\(xy(A_{1}x + A_{0}y)(A_{1}x + A_{2}y + A_{3}z)zt(A_{1}x + A_{2}y + A_{4}t)(x + y + z + t)\)}

\noindent\parbox[right]{9mm}{\textbf{87:}}\parbox{3mm}{\hfill}%
\parbox[t]{15cm}{\(xy(A_{0}x + A_{1}y)(A_{0}x + A_{0}y + A_{2}z)zt(A_{0}x + A_{3}y + A_{4}t)(x + y + z + t)\)}

\noindent\parbox[right]{9mm}{\textbf{88:}}\parbox{3mm}{\hfill}%
\parbox[t]{15cm}{\(xy(A_{0}x + A_{1}y)(A_{0}x + A_{2}y + A_{3}z)zt(A_{0}x + A_{4}y + A_{5}t)(x + y + z + t)\)}

\noindent\parbox[right]{9mm}{\textbf{89:}}\parbox{3mm}{\hfill}%
\parbox[t]{15cm}{\(x(A_{0}x + A_{1}y)y(x + t)(y + t)z(A_{0}x + A_{2}z - A_{1}t)(z + t)\)}

\noindent\parbox[right]{9mm}{\textbf{90:}}\parbox{3mm}{\hfill}%
\parbox[t]{15cm}{\(x(x + y)y(A_{0}x + A_{1}y + A_{1}z)zt(x + y + z + t)(A_{0}x + A_{2}y + A_{2}z + A_{3}t)\)}

\noindent\parbox[right]{9mm}{\textbf{91:}}\parbox{3mm}{\hfill}%
\parbox[t]{15cm}{\(xy(A_{0}x + A_{1}y)(A_{0}x + A_{2}y + A_{2}z)z(A_{0}x + A_{3}y + A_{3}z + A_{4}t)t(x + y + z + t)\)}

\noindent\parbox[right]{9mm}{\textbf{92:}}\parbox{3mm}{\hfill}%
\parbox[t]{15cm}{\(xy(A_{0}x + A_{1}y)z(A_{0}x + A_{0}y + A_{2}z)(A_{0}x + A_{0}y + A_{3}z + A_{4}t)t(x + y + z + t)\)}

\noindent\parbox[right]{9mm}{\textbf{93:}}\parbox{3mm}{\hfill}%
\parbox[t]{15cm}{\(xy(x + y)(y + t)(x + t)z(x + y - z + t)(z + t)\)}

\noindent\parbox[right]{9mm}{\textbf{94:}}\parbox{3mm}{\hfill}%
\parbox[t]{15cm}{\(xy(x + y)z(A_{0}x + A_{1}y + A_{0}z)(A_{1}x + (-A_{0} + A_{1})z + A_{1}t)t(x + y + z + t)\)}

\noindent\parbox[right]{9mm}{\textbf{95:}}\parbox{3mm}{\hfill}%
\parbox[t]{15cm}{\(xy(A_{0}x + A_{1}y)(y + t)(x + t)z(A_{0}x + A_{1}y - A_{0}z + A_{1}t)(z + t)\)}

\noindent\parbox[right]{9mm}{\textbf{96:}}\parbox{3mm}{\hfill}%
\parbox[t]{15cm}{\(xy(x + y)(A_{0}x + A_{1}z + A_{0}t)(A_{1}x + A_{0}y + A_{1}z + A_{0}t)zt(x + y + z + t)\)}

\noindent\parbox[right]{9mm}{\textbf{97:}}\parbox{3mm}{\hfill}%
\parbox[t]{15cm}{\(xy(x + y)(A_{0}x + A_{1}z + A_{0}t)(A_{0}y + A_{1}z + A_{0}t)zt(x + y + z + t)\)}

\noindent\parbox[right]{9mm}{\textbf{98:}}\parbox{3mm}{\hfill}%
\parbox[t]{15cm}{\(xy(A_{0}x + A_{1}y)(y + t)(x + t)z(x + y - z + t)(z + t)\)}

\noindent\parbox[right]{9mm}{\textbf{99:}}\parbox{3mm}{\hfill}%
\parbox[t]{15cm}{\(xy(A_{0}x + A_{1}y)(y + t)(x + t)z(A_{0}x + A_{0}y + A_{1}z + A_{0}t)(z + t)\)}

\noindent\parbox[right]{9mm}{\textbf{100:}}\parbox{3mm}{\hfill}%
\parbox[t]{15cm}{\(x(A_{0}x + A_{1}y - A_{0}z)(-A_{1}x + A_{1}y - A_{0}z)yz(x + t)(y + t)(z + t)\)}

\noindent\parbox[right]{9mm}{\textbf{101:}}\parbox{3mm}{\hfill}%
\parbox[t]{15cm}{\(xy(A_{0}x + A_{1}y)(y + t)(x + t)z(A_{0}x + A_{2}y - A_{0}z + A_{2}t)(z + t)\)}

\noindent\parbox[right]{9mm}{\textbf{102:}}\parbox{3mm}{\hfill}%
\parbox[t]{15cm}{\(xy(A_{0}x + A_{1}y)(y + t)(x + t)z(A_{0}x + A_{2}y + A_{1}z + A_{2}t)(z + t)\)}

\noindent\parbox[right]{9mm}{\textbf{103:}}\parbox{3mm}{\hfill}%
\parbox[t]{15cm}{\(xy(A_{0}x + A_{1}y)z(A_{0}x + A_{0}y + A_{2}z)((-A_{0} + A_{1})x + A_{1}z + (-A_{0} + A_{1})t)t(x + y + z + t)\)}

\noindent\parbox[right]{9mm}{\textbf{104:}}\parbox{3mm}{\hfill}%
\parbox[t]{15cm}{\(xy(A_{0}x + A_{1}y)(y + t)(x + t)z(A_{0}x + A_{1}y + A_{2}z + A_{1}t)(z + t)\)}

\noindent\parbox[right]{9mm}{\textbf{105:}}\parbox{3mm}{\hfill}%
\parbox[t]{15cm}{\(xy(A_{0}x + A_{1}y)z(A_{0}x + A_{2}y + A_{2}z)t((-A_{0} + A_{2})x + (-A_{0} + A_{2})z + A_{2}t)(x + y + z + t)\)}

\noindent\parbox[right]{9mm}{\textbf{106:}}\parbox{3mm}{\hfill}%
\parbox[t]{15cm}{\(xy(A_{0}x + A_{1}y)z(A_{0}x + A_{2}y + A_{2}z)t((-A_{0} + A_{1})x + A_{1}z + (-A_{0} + A_{1})t)(x + y + z + t)\)}

\noindent\parbox[right]{9mm}{\textbf{107:}}\parbox{3mm}{\hfill}%
\parbox[t]{15cm}{\(xy(A_{0}x + A_{1}y)(y + t)(x + t)z(A_{0}x + (A_{0} + A_{2})y + A_{2}z + (A_{0} + A_{2})t)(z + t)\)}

\noindent\parbox[right]{9mm}{\textbf{108:}}\parbox{3mm}{\hfill}%
\parbox[t]{15cm}{\(x((-A_{0} + A_{2})x + A_{2}y)yz(A_{0}x + A_{0}y + A_{1}z)(A_{0}x + A_{1}z + A_{2}t)t(x + y + z + t)\)}

\noindent\parbox[right]{9mm}{\textbf{109:}}\parbox{3mm}{\hfill}%
\parbox[t]{15cm}{\(x(x + y)yz((A_{0} + A_{1} - A_{2})x + A_{0}y + A_{0}z)t(A_{1}x + A_{2}z + (-A_{0} + A_{2})t)(x + y + z + t)\)}

\noindent\parbox[right]{9mm}{\textbf{110:}}\parbox{3mm}{\hfill}%
\parbox[t]{15cm}{\(x(A_{0}x + A_{1}y + A_{0}z)(A_{1}x + A_{1}y + A_{0}z)yz(A_{0}x + A_{2}y + A_{2}t)t(x + y + z + t)\)}

\noindent\parbox[right]{9mm}{\textbf{111:}}\parbox{3mm}{\hfill}%
\parbox[t]{15cm}{\(xy((A_{0} - A_{1})x + A_{1}y)(y + t)(x + t)z(A_{0}x + A_{1}y + A_{2}z + A_{1}t)(z + t)\)}

\noindent\parbox[right]{9mm}{\textbf{112:}}\parbox{3mm}{\hfill}%
\parbox[t]{15cm}{\(x(x + y)y(A_{0}x + A_{1}y + A_{1}z)zt(x + y + z + t)(A_{0}x + A_{1}y + A_{0}z + A_{2}t)\)}

\noindent\parbox[right]{9mm}{\textbf{113:}}\parbox{3mm}{\hfill}%
\parbox[t]{15cm}{\(xy((A_{1} - A_{2})x - A_{2}y)z(A_{0}x + (A_{0} + A_{1})y + A_{0}z)t(A_{1}x + A_{2}z + (A_{0} + A_{1})t)(x + y + z + t)\)}

\noindent\parbox[right]{9mm}{\textbf{114:}}\parbox{3mm}{\hfill}%
\parbox[t]{15cm}{\(xy(x + y)z(A_{0}x + A_{1}y + A_{2}z)((A_{0} - A_{1})x + (-A_{1} + A_{2})z + (A_{0} - A_{1})t)t(x + y + z + t)\)}

\noindent\parbox[right]{9mm}{\textbf{115:}}\parbox{3mm}{\hfill}%
\parbox[t]{15cm}{\(xy(A_{0}x + A_{1}y)(y + t)(x + t)z(A_{0}x + A_{0}y + A_{2}z + A_{0}t)(z + t)\)}

\noindent\parbox[right]{9mm}{\textbf{116:}}\parbox{3mm}{\hfill}%
\parbox[t]{15cm}{\(xy(A_{0}x + A_{1}y)z(A_{0}x + A_{0}y + A_{2}z)t(A_{0}x + A_{0}z + A_{3}t)(x + y + z + t)\)}

\noindent\parbox[right]{9mm}{\textbf{117:}}\parbox{3mm}{\hfill}%
\parbox[t]{15cm}{\(xy(A_{0}x + A_{1}y)(A_{2}x + A_{1}y + A_{2}z)zt(A_{0}x + A_{1}y + A_{2}z + A_{3}t)(x + y + z + t)\)}

\noindent\parbox[right]{9mm}{\textbf{118:}}\parbox{3mm}{\hfill}%
\parbox[t]{15cm}{\(xy(A_{0}x + A_{1}y)(x + t)(y + t)z(z + t)(A_{0}x + A_{2}y + A_{3}z + (A_{0} + A_{2})t)\)}

\noindent\parbox[right]{9mm}{\textbf{119:}}\parbox{3mm}{\hfill}%
\parbox[t]{15cm}{\(xy(A_{0}x + A_{1}y)(A_{2}x + (-A_{1} + A_{2})y + A_{2}z)zt((-A_{0} + A_{2})x + (-A_{1} + A_{2})y + A_{2}z + A_{3}t)(x + y + z + t)\)}

\noindent\parbox[right]{9mm}{\textbf{120:}}\parbox{3mm}{\hfill}%
\parbox[t]{15cm}{\(xy(A_{0}x + A_{1}y)((-A_{2} + 2A_{3})x + 2A_{3}y + A_{3}z)z(A_{2}x + 2A_{3}y + A_{3}z + 2A_{3}t)t(x + y + z + t)\)}

\noindent\parbox[right]{9mm}{\textbf{121:}}\parbox{3mm}{\hfill}%
\parbox[t]{15cm}{\(xy(A_{2}x + A_{0}y)z(A_{2}x + A_{2}y + A_{1}z)(A_{2}x + A_{3}z + A_{2}t)t(x + y + z + t)\)}

\noindent\parbox[right]{9mm}{\textbf{122:}}\parbox{3mm}{\hfill}%
\parbox[t]{15cm}{\(xy(A_{0}x + A_{1}y)(A_{0}x + A_{2}y + A_{2}z)zt(x + y + z + t)((A_{0} - A_{2})y + A_{3}z + A_{0}t)\)}

\noindent\parbox[right]{9mm}{\textbf{123:}}\parbox{3mm}{\hfill}%
\parbox[t]{15cm}{\(xy(A_{0}x + A_{1}y)(A_{0}x + A_{2}y + A_{2}z)zt(x + y + z + t)((A_{0} - A_{1})y + A_{0}z + A_{3}t)\)}

\noindent\parbox[right]{9mm}{\textbf{124:}}\parbox{3mm}{\hfill}%
\parbox[t]{15cm}{\(xy(A_{0}x + A_{1}y)(x + t)(y + t)z(z + t)(A_{0}x + A_{1}y + A_{2}z + A_{3}t)\)}

\noindent\parbox[right]{9mm}{\textbf{125:}}\parbox{3mm}{\hfill}%
\parbox[t]{15cm}{\(xy(A_{0}x + A_{1}y)((-A_{0} + A_{3})x + A_{3}y + A_{3}z)zt(x + y + z + t)(A_{0}x + A_{2}z + A_{3}t)\)}

\noindent\parbox[right]{9mm}{\textbf{126:}}\parbox{3mm}{\hfill}%
\parbox[t]{15cm}{\(xy(A_{0}x + A_{1}y)(A_{0}x + A_{2}y + A_{2}z)zt(x + y + z + t)(A_{0}x + A_{0}z + A_{3}t)\)}

\noindent\parbox[right]{9mm}{\textbf{127:}}\parbox{3mm}{\hfill}%
\parbox[t]{15cm}{\(xy(A_{0}x + A_{1}y)(A_{0}A_{3}x + (A_{0}A_{1} + A_{0}A_{3} - A_{1}A_{2})y + A_{0}A_{3}z)zt(A_{0}^2A_{3}x + (A_{0}A_{1}A_{2} + A_{0}A_{2}A_{3} - A_{1}A_{2}^2)y + A_{0}A_{2}A_{3}z + A_{0}A_{3}^2t)(x + y + z + t)\)}

\noindent\parbox[right]{9mm}{\textbf{128:}}\parbox{3mm}{\hfill}%
\parbox[t]{15cm}{\(xy(A_{0}x + A_{1}y)(A_{3}x + A_{2}y + A_{3}z)z(A_{0}x + A_{2}y + A_{3}z + A_{2}t)t(x + y + z + t)\)}

\noindent\parbox[right]{9mm}{\textbf{129:}}\parbox{3mm}{\hfill}%
\parbox[t]{15cm}{\(xy(A_{0}x + A_{1}y)(A_{0}x + A_{2}y + A_{2}z)zt(x + y + z + t)(A_{0}x + A_{0}y + A_{2}z + A_{3}t)\)}

\noindent\parbox[right]{9mm}{\textbf{130:}}\parbox{3mm}{\hfill}%
\parbox[t]{15cm}{\(xy(A_{0}x + A_{1}y)(A_{2}x + (A_{1} + A_{3})y + A_{2}z)zt((A_{0} + A_{3})x + (A_{1} + A_{3})y + A_{2}z + A_{3}t)(x + y + z + t)\)}

\noindent\parbox[right]{9mm}{\textbf{131:}}\parbox{3mm}{\hfill}%
\parbox[t]{15cm}{\(xy(A_{0}x + A_{1}y)(A_{0}x + A_{2}y + A_{2}z)zt(x + y + z + t)(A_{0}x + A_{2}y + A_{3}z + A_{4}t)\)}

\noindent\parbox[right]{9mm}{\textbf{132:}}\parbox{3mm}{\hfill}%
\parbox[t]{15cm}{\(xy(A_{0}x + A_{1}y)z(A_{0}x + A_{2}y + A_{3}z)(A_{0}x + A_{2}y + A_{2}z + A_{4}t)t(x + y + z + t)\)}

\noindent\parbox[right]{9mm}{\textbf{133:}}\parbox{3mm}{\hfill}%
\parbox[t]{15cm}{\(xy(A_{0}x + A_{1}y)(A_{0}x + A_{2}y + A_{2}z)zt(x + y + z + t)(A_{0}x + A_{3}y + A_{0}z + A_{4}t)\)}

\noindent\parbox[right]{9mm}{\textbf{134:}}\parbox{3mm}{\hfill}%
\parbox[t]{15cm}{\(xy(A_{0}x + A_{1}y)(A_{0}x + A_{2}y + A_{3}z)zt(A_{2}y + A_{3}z + A_{4}t)(x + y + z + t)\)}

\noindent\parbox[right]{9mm}{\textbf{135:}}\parbox{3mm}{\hfill}%
\parbox[t]{15cm}{\(xy(A_{1}x + A_{0}y)(A_{1}x + A_{2}y + A_{1}z)zt(A_{3}x + A_{2}y + A_{1}z + A_{4}t)(x + y + z + t)\)}

\noindent\parbox[right]{9mm}{\textbf{136:}}\parbox{3mm}{\hfill}%
\parbox[t]{15cm}{\(xy(A_{0}x + A_{1}y)(A_{0}x + A_{2}y + A_{2}z)zt(x + y + z + t)(A_{0}x + A_{3}z + A_{4}t)\)}

\noindent\parbox[right]{9mm}{\textbf{137:}}\parbox{3mm}{\hfill}%
\parbox[t]{15cm}{\(xy(A_{0}x + A_{1}y)(A_{2}x + A_{1}y + A_{1}z)zt(x + y + z + t)(A_{1}y + A_{3}z + A_{4}t)\)}

\noindent\parbox[right]{9mm}{\textbf{138:}}\parbox{3mm}{\hfill}%
\parbox[t]{15cm}{\(xy(A_{0}x + A_{1}y)(A_{0}x + A_{2}y + A_{2}z)zt(x + y + z + t)(A_{0}x + A_{0}y + A_{3}z + A_{4}t)\)}

\noindent\parbox[right]{9mm}{\textbf{139:}}\parbox{3mm}{\hfill}%
\parbox[t]{15cm}{\(xy(A_{0}x + A_{1}y)(A_{0}x + A_{2}y + A_{2}z)zt(x + y + z + t)(A_{0}x + A_{3}y + A_{4}z + A_{5}t)\)}

\noindent\parbox[right]{9mm}{\textbf{140:}}\parbox{3mm}{\hfill}%
\parbox[t]{15cm}{\(xy(A_{0}x + A_{1}y)z(A_{0}x + A_{2}y + A_{3}z)(A_{0}x + A_{0}y + A_{4}z + A_{5}t)t(x + y + z + t)\)}

\noindent\parbox[right]{9mm}{\textbf{141:}}\parbox{3mm}{\hfill}%
\parbox[t]{15cm}{\(xy(A_{0}x + A_{1}y)z(A_{0}x + A_{2}y + A_{3}z)(A_{0}x + A_{2}y + A_{4}z + A_{5}t)t(x + y + z + t)\)}

\noindent\parbox[right]{9mm}{\textbf{142:}}\parbox{3mm}{\hfill}%
\parbox[t]{15cm}{\(xy(A_{0}x + A_{1}y)z(A_{0}x + A_{2}y + A_{3}z)(A_{0}x + A_{4}y + A_{4}z + A_{5}t)t(x + y + z + t)\)}

\noindent\parbox[right]{9mm}{\textbf{143:}}\parbox{3mm}{\hfill}%
\parbox[t]{15cm}{\(xy(A_{0}x + A_{1}y)z(A_{0}x + A_{2}y + A_{3}z)(A_{0}x + A_{4}y + A_{5}z + A_{6}t)t(x + y + z + t)\)}

\noindent\parbox[right]{9mm}{\textbf{144:}}\parbox{3mm}{\hfill}%
\parbox[t]{15cm}{\(xy(x + y)z(A_{0}x + A_{1}z + A_{0}t)t(A_{0}x + A_{0}z + (2A_{0} - A_{1})t)(x + y + z + t)\)}

\noindent\parbox[right]{9mm}{\textbf{145:}}\parbox{3mm}{\hfill}%
\parbox[t]{15cm}{\((A_{0}y + A_{1}z + A_{2}t)x(A_{2}x + A_{0}y + A_{1}z + A_{2}t)y(y + t)z(z + t)(x + t)\)}

\noindent\parbox[right]{9mm}{\textbf{146:}}\parbox{3mm}{\hfill}%
\parbox[t]{15cm}{\(xy(x + y)z(A_{0}x + A_{1}z + A_{0}t)t(A_{0}x + A_{2}z + A_{3}t)(x + y + z + t)\)}

\noindent\parbox[right]{9mm}{\textbf{147:}}\parbox{3mm}{\hfill}%
\parbox[t]{15cm}{\((A_{0}y + A_{1}z + A_{2}t)x(A_{3}x + A_{0}y + A_{1}z + A_{2}t)y(y + t)z(z + t)(x + t)\)}

\noindent\parbox[right]{9mm}{\textbf{148:}}\parbox{3mm}{\hfill}%
\parbox[t]{15cm}{\(xy(A_{0}x + A_{1}y)z(A_{0}x + A_{2}z + A_{0}t)t((-A_{0} + A_{1})x + A_{1}z + A_{3}t)(x + y + z + t)\)}

\noindent\parbox[right]{9mm}{\textbf{149:}}\parbox{3mm}{\hfill}%
\parbox[t]{15cm}{\(xy(A_{0}x + A_{1}y)(A_{0}x + A_{2}z + A_{0}t)zt(A_{0}x + A_{3}z + A_{4}t)(x + y + z + t)\)}

\noindent\parbox[right]{9mm}{\textbf{150:}}\parbox{3mm}{\hfill}%
\parbox[t]{15cm}{\(xy(A_{0}x + A_{1}y)z(A_{0}x + A_{2}z + A_{3}t)(A_{0}x + A_{4}z + A_{5}t)t(x + y + z + t)\)}

\noindent\parbox[right]{9mm}{\textbf{151:}}\parbox{3mm}{\hfill}%
\parbox[t]{15cm}{\(xy(A_{0}x + A_{1}y)zt(A_{0}x + A_{0}y + A_{2}z + A_{3}t)(A_{0}x + A_{0}y + A_{4}z + A_{5}t)(x + y + z + t)\)}

\noindent\parbox[right]{9mm}{\textbf{152:}}\parbox{3mm}{\hfill}%
\parbox[t]{15cm}{\(x(x + 2y)yz(x + z + 2t)t(x + y + z + t)(A_{0}x + 2A_{0}y + A_{1}z + 2A_{0}t)\)}

\noindent\parbox[right]{9mm}{\textbf{153:}}\parbox{3mm}{\hfill}%
\parbox[t]{15cm}{\(x(x + y)yzt(A_{1}x + A_{0}z - A_{1}t)(x + y + z + t)(A_{1}y + (-A_{0} + A_{1})z + A_{1}t)\)}

\noindent\parbox[right]{9mm}{\textbf{154:}}\parbox{3mm}{\hfill}%
\parbox[t]{15cm}{\(xy(x + y)zt(A_{0}x + A_{0}z + A_{1}t)(A_{0}y + A_{1}z + A_{0}t)(x + y + z + t)\)}

\noindent\parbox[right]{9mm}{\textbf{155:}}\parbox{3mm}{\hfill}%
\parbox[t]{15cm}{\(xy(x + y)z(A_{0}x + A_{0}z + A_{1}t)t(A_{0}^2x + A_{0}A_{1}y + (A_{0}^2 - A_{0}A_{1} + A_{1}^2)z + A_{0}A_{1}t)(x + y + z + t)\)}

\noindent\parbox[right]{9mm}{\textbf{156:}}\parbox{3mm}{\hfill}%
\parbox[t]{15cm}{\(xy(x + y)zt(A_{0}x + A_{1}z + A_{0}t)(A_{2}y + A_{1}z + A_{2}t)(x + y + z + t)\)}

\noindent\parbox[right]{9mm}{\textbf{157:}}\parbox{3mm}{\hfill}%
\parbox[t]{15cm}{\(xy(A_{0}x + A_{1}y)zt(-A_{0}x - A_{0}z + A_{1}t)(A_{1}y + A_{2}z + A_{1}t)(x + y + z + t)\)}

\noindent\parbox[right]{9mm}{\textbf{158:}}\parbox{3mm}{\hfill}%
\parbox[t]{15cm}{\(xy(x + y)z(A_{0}x + A_{0}z + A_{1}t)t(A_{0}x + A_{1}y + A_{2}z + A_{1}t)(x + y + z + t)\)}

\noindent\parbox[right]{9mm}{\textbf{159:}}\parbox{3mm}{\hfill}%
\parbox[t]{15cm}{\(x(x + y)yzt(A_{0}x + A_{1}z + A_{2}t)(x + y + z + t)(A_{0}y + (A_{0} - A_{1})z + A_{0}t)\)}

\noindent\parbox[right]{9mm}{\textbf{160:}}\parbox{3mm}{\hfill}%
\parbox[t]{15cm}{\(x(A_{0}x + (A_{0} + A_{1})y)yz(A_{1}x + A_{2}z + (A_{0} + A_{1})t)t(x + y + z + t)(A_{1}y + (A_{1} - A_{2})z + A_{1}t)\)}

\noindent\parbox[right]{9mm}{\textbf{161:}}\parbox{3mm}{\hfill}%
\parbox[t]{15cm}{\(xy(A_{0}x + A_{1}y)zt(A_{0}x + A_{2}z + A_{0}t)(A_{0}y + (A_{0} - A_{2})z + A_{0}t)(x + y + z + t)\)}

\noindent\parbox[right]{9mm}{\textbf{162:}}\parbox{3mm}{\hfill}%
\parbox[t]{15cm}{\(xy(x + y)(A_{0}A_{2}x + A_{0}A_{1}y + A_{0}A_{1}z + A_{1}A_{2}t)(A_{2}x + A_{0}y + A_{1}z + A_{2}t)zt(x + y + z + t)\)}

\noindent\parbox[right]{9mm}{\textbf{163:}}\parbox{3mm}{\hfill}%
\parbox[t]{15cm}{\(xy((A_{0} - A_{1})x + A_{0}y)z(A_{1}x + A_{0}z + A_{1}t)t(x + y + z + t)(A_{0}y + A_{2}z + A_{0}t)\)}

\noindent\parbox[right]{9mm}{\textbf{164:}}\parbox{3mm}{\hfill}%
\parbox[t]{15cm}{\(xy(x + y)zt(A_{0}x + A_{0}z + A_{1}t)(A_{0}y + A_{2}z + A_{0}t)(x + y + z + t)\)}

\noindent\parbox[right]{9mm}{\textbf{165:}}\parbox{3mm}{\hfill}%
\parbox[t]{15cm}{\(xy(A_{0}x + A_{1}y)z((-A_{0} + A_{1})x + A_{1}z + (-A_{0} + A_{1})t)t(x + y + z + t)(A_{0}x + A_{2}y + A_{3}z + A_{2}t)\)}

\noindent\parbox[right]{9mm}{\textbf{166:}}\parbox{3mm}{\hfill}%
\parbox[t]{15cm}{\(xy(x + y)zt(A_{0}x + A_{1}z + A_{2}t)(A_{3}y + A_{0}z + A_{3}t)(x + y + z + t)\)}

\noindent\parbox[right]{9mm}{\textbf{167:}}\parbox{3mm}{\hfill}%
\parbox[t]{15cm}{\(xy((-A_{0} + A_{1})x + A_{1}y)(A_{0}x + A_{1}z + A_{0}t)zt(x + y + z + t)(A_{0}x + A_{0}y + A_{2}z + A_{3}t)\)}

\noindent\parbox[right]{9mm}{\textbf{168:}}\parbox{3mm}{\hfill}%
\parbox[t]{15cm}{\(xy((A_{1} - A_{3})x + (A_{2} - A_{3})y)zt(A_{1}x + A_{1}z + A_{0}t)(x + y + z + t)(A_{1}x + A_{2}y + A_{3}z + A_{2}t)\)}

\noindent\parbox[right]{9mm}{\textbf{169:}}\parbox{3mm}{\hfill}%
\parbox[t]{15cm}{\(xy(A_{0}x + A_{1}y)z(A_{0}x + A_{2}z + A_{0}t)t(x + y + z + t)(A_{0}x + A_{1}y + A_{3}z + A_{1}t)\)}

\noindent\parbox[right]{9mm}{\textbf{170:}}\parbox{3mm}{\hfill}%
\parbox[t]{15cm}{\(xy(A_{0}x + A_{1}y)zt(A_{0}x + A_{0}z + A_{2}t)(A_{3}y + A_{0}z + A_{3}t)(x + y + z + t)\)}

\noindent\parbox[right]{9mm}{\textbf{171:}}\parbox{3mm}{\hfill}%
\parbox[t]{15cm}{\(xy(A_{0}x + A_{1}y)zt(A_{0}x + A_{2}z + A_{0}t)(A_{3}y + A_{0}z + A_{3}t)(x + y + z + t)\)}

\noindent\parbox[right]{9mm}{\textbf{172:}}\parbox{3mm}{\hfill}%
\parbox[t]{15cm}{\(xy((A_{0} - A_{1})x + A_{0}y)zt(A_{1}x + (A_{1} - A_{2})z + A_{0}t)(A_{1}y + A_{2}z + A_{3}t)(x + y + z + t)\)}

\noindent\parbox[right]{9mm}{\textbf{173:}}\parbox{3mm}{\hfill}%
\parbox[t]{15cm}{\(x(A_{0}x + A_{1}y + A_{2}z + A_{0}t)(A_{3}x + A_{1}y + A_{2}z + A_{0}t)(x + t)y(y + t)z(z + t)\)}

\noindent\parbox[right]{9mm}{\textbf{174:}}\parbox{3mm}{\hfill}%
\parbox[t]{15cm}{\(xy(x + y)zt(A_{1}x + A_{0}z + (A_{1} - A_{3})t)(A_{1}y + A_{2}z + A_{3}t)(x + y + z + t)\)}

\noindent\parbox[right]{9mm}{\textbf{175:}}\parbox{3mm}{\hfill}%
\parbox[t]{15cm}{\(xy(A_{0}x + A_{1}y)zt((-A_{0} + A_{1})x + A_{1}z + A_{2}t)(A_{3}y + A_{0}z + A_{3}t)(x + y + z + t)\)}

\noindent\parbox[right]{9mm}{\textbf{176:}}\parbox{3mm}{\hfill}%
\parbox[t]{15cm}{\(xy((A_{1} - A_{3})x + A_{2}y)zt(A_{1}x + A_{1}z + A_{0}t)(x + y + z + t)(A_{1}x + A_{2}y + A_{3}z + A_{2}t)\)}

\noindent\parbox[right]{9mm}{\textbf{177:}}\parbox{3mm}{\hfill}%
\parbox[t]{15cm}{\(xy(A_{1}x + A_{2}y)z(A_{1}x + A_{0}z + A_{4}t)t(A_{1}x + A_{2}y + A_{3}z + A_{4}t)(x + y + z + t)\)}

\noindent\parbox[right]{9mm}{\textbf{178:}}\parbox{3mm}{\hfill}%
\parbox[t]{15cm}{\(xy(A_{0}x + A_{1}y)z(A_{0}x + A_{2}z + A_{3}t)t(A_{0}x + A_{3}y + A_{4}z + A_{3}t)(x + y + z + t)\)}

\noindent\parbox[right]{9mm}{\textbf{179:}}\parbox{3mm}{\hfill}%
\parbox[t]{15cm}{\(xy(A_{0}x + A_{1}y)z((-A_{0} + A_{1})x + A_{2}z + A_{1}t)t((-A_{0} + A_{1})x + A_{3}y + A_{4}z + A_{1}t)(x + y + z + t)\)}

\noindent\parbox[right]{9mm}{\textbf{180:}}\parbox{3mm}{\hfill}%
\parbox[t]{15cm}{\(xy(A_{0}x + A_{1}y)z(A_{1}x + A_{2}z + A_{4}t)t(A_{1}x + A_{1}y + A_{3}z + A_{4}t)(x + y + z + t)\)}

\noindent\parbox[right]{9mm}{\textbf{181:}}\parbox{3mm}{\hfill}%
\parbox[t]{15cm}{\(xy(A_{0}x + A_{1}y)z(A_{0}x + A_{0}z + A_{2}t)t(x + y + z + t)(A_{0}x + A_{3}y + A_{4}z + A_{3}t)\)}

\noindent\parbox[right]{9mm}{\textbf{182:}}\parbox{3mm}{\hfill}%
\parbox[t]{15cm}{\(xy(x + y)z(A_{0}x + A_{1}z + A_{2}t)t(A_{0}x + A_{3}y + A_{4}z + A_{2}t)(x + y + z + t)\)}

\noindent\parbox[right]{9mm}{\textbf{183:}}\parbox{3mm}{\hfill}%
\parbox[t]{15cm}{\(xy(A_{0}x + A_{1}y)z((A_{0} - A_{4})x + (A_{3} - A_{4})z + A_{2}t)t(x + y + z + t)(A_{0}x + A_{4}y + A_{3}z + A_{4}t)\)}

\noindent\parbox[right]{9mm}{\textbf{184:}}\parbox{3mm}{\hfill}%
\parbox[t]{15cm}{\(xy(A_{0}x + A_{1}y)z((-A_{0} + A_{3})x + (A_{3} - A_{4})z + A_{2}t)t(x + y + z + t)(A_{0}x + A_{3}y + A_{4}z + A_{0}t)\)}

\noindent\parbox[right]{9mm}{\textbf{185:}}\parbox{3mm}{\hfill}%
\parbox[t]{15cm}{\(xy((A_{1} - A_{3})x + (A_{2} - A_{3})y)z(A_{1}x + A_{0}z + A_{4}t)t(A_{1}x + A_{2}y + A_{3}z + A_{4}t)(x + y + z + t)\)}

\noindent\parbox[right]{9mm}{\textbf{186:}}\parbox{3mm}{\hfill}%
\parbox[t]{15cm}{\(xy(A_{0}x + A_{1}y)z(A_{0}x + A_{2}z + A_{3}t)t(A_{0}x + A_{4}y + A_{5}z + A_{3}t)(x + y + z + t)\)}

\noindent\parbox[right]{9mm}{\textbf{187:}}\parbox{3mm}{\hfill}%
\parbox[t]{15cm}{\(xy(A_{0}x + A_{1}y)z(A_{0}x + A_{2}z + A_{3}t)t(x + y + z + t)(A_{0}x + A_{4}y + A_{5}z + A_{0}t)\)}

\noindent\parbox[right]{9mm}{\textbf{188:}}\parbox{3mm}{\hfill}%
\parbox[t]{15cm}{\(xy(A_{0}x + A_{1}y)z(A_{0}x + A_{2}z + A_{3}t)t(x + y + z + t)(A_{0}x + A_{4}y + A_{5}z + A_{4}t)\)}

\noindent\parbox[right]{9mm}{\textbf{189:}}\parbox{3mm}{\hfill}%
\parbox[t]{15cm}{\(xy(A_{0}x + A_{1}y)z(A_{0}x + A_{2}z + A_{3}t)t(x + y + z + t)((-A_{0} + A_{4})x + A_{4}y + A_{5}z + (-A_{3} + A_{4})t)\)}

\noindent\parbox[right]{9mm}{\textbf{190:}}\parbox{3mm}{\hfill}%
\parbox[t]{15cm}{\(xy(A_{0}x + A_{1}y)z(A_{0}x + A_{2}z + A_{3}t)t(x + y + z + t)(A_{0}x + A_{4}y + A_{5}z + A_{6}t)\)}

\noindent\parbox[right]{9mm}{\textbf{191:}}\parbox{3mm}{\hfill}%
\parbox[t]{15cm}{\(xy(A_{0}x + A_{1}y)zt(A_{0}x + A_{2}y + A_{3}z + A_{4}t)(A_{0}x + A_{2}y + A_{5}z + A_{6}t)(x + y + z + t)\)}

\noindent\parbox[right]{9mm}{\textbf{192:}}\parbox{3mm}{\hfill}%
\parbox[t]{15cm}{\(xy(A_{0}x + A_{1}y)zt(A_{0}x + A_{2}y + A_{3}z + A_{4}t)(A_{0}x + A_{5}y + A_{6}z + A_{7}t)(x + y + z + t)\)}

\noindent\parbox[right]{9mm}{\textbf{193:}}\parbox{3mm}{\hfill}%
\parbox[t]{15cm}{\((A_{0}x + A_{1}y + (-A_{0} - A_{1})z)(A_{0}x + A_{2}y + (-A_{0} - A_{2})z)xyz(x + t)(y + t)(z + t)\)}

\noindent\parbox[right]{9mm}{\textbf{194:}}\parbox{3mm}{\hfill}%
\parbox[t]{15cm}{\(x(x + t)y(y + t)(A_{0}x + A_{1}y + A_{2}t)z(z + t)(A_{0}x + A_{3}z + A_{2}t)\)}

\noindent\parbox[right]{9mm}{\textbf{195:}}\parbox{3mm}{\hfill}%
\parbox[t]{15cm}{\(x(A_{0}x + A_{1}y + A_{1}z)y(A_{0}x + A_{2}y + A_{0}z)zt(x + y + z + t)(A_{0}x + A_{3}y + A_{3}z + A_{4}t)\)}

\noindent\parbox[right]{9mm}{\textbf{196:}}\parbox{3mm}{\hfill}%
\parbox[t]{15cm}{\(x(A_{0}x + A_{1}y + A_{1}z)yz(A_{0}x + A_{2}y + A_{3}z)(A_{0}x + A_{4}y + A_{4}z + A_{5}t)t(x + y + z + t)\)}

\noindent\parbox[right]{9mm}{\textbf{197:}}\parbox{3mm}{\hfill}%
\parbox[t]{15cm}{\(xy(x + t)(y + t)(A_{0}x + A_{0}y + A_{1}t)z(x + y - z)(z + t)\)}

\noindent\parbox[right]{9mm}{\textbf{198:}}\parbox{3mm}{\hfill}%
\parbox[t]{15cm}{\(x(x + 2y + 2z)yz(A_{0}x + A_{1}y + A_{0}z)t(x + y + z + t)(x + 2y + 2t)\)}

\noindent\parbox[right]{9mm}{\textbf{199:}}\parbox{3mm}{\hfill}%
\parbox[t]{15cm}{\(xyz(A_{0}x + A_{1}y - A_{1}z)(A_{1}x + A_{0}y - A_{1}z)(x + t)(y + t)(z + t)\)}

\noindent\parbox[right]{9mm}{\textbf{200:}}\parbox{3mm}{\hfill}%
\parbox[t]{15cm}{\(x(A_{0}x + A_{1}z + A_{0}t)(A_{1}x + A_{0}y + A_{1}z + A_{0}t)(A_{1}y + A_{1}z + A_{0}t)yzt(x + y + z + t)\)}

\noindent\parbox[right]{9mm}{\textbf{201:}}\parbox{3mm}{\hfill}%
\parbox[t]{15cm}{\(xyz(A_{0}x + A_{1}y - A_{1}z)(A_{0}x + A_{2}y + (-A_{0} - A_{2})z)(x + t)(y + t)(z + t)\)}

\noindent\parbox[right]{9mm}{\textbf{202:}}\parbox{3mm}{\hfill}%
\parbox[t]{15cm}{\(xy(A_{0}x + A_{1}y + A_{0}z)(A_{2}x + (A_{0} - A_{1})y + (A_{0} - A_{1})z)z(A_{2}x + (A_{0} - A_{1})y + A_{0}t)t(x + y + z + t)\)}

\noindent\parbox[right]{9mm}{\textbf{203:}}\parbox{3mm}{\hfill}%
\parbox[t]{15cm}{\(xy(x + t)(y + t)(A_{0}x + A_{1}y + A_{2}t)z(x + y - z)(z + t)\)}

\noindent\parbox[right]{9mm}{\textbf{204:}}\parbox{3mm}{\hfill}%
\parbox[t]{15cm}{\(xy(A_{0}x + A_{0}y + A_{2}z)(A_{0}x + A_{1}y + A_{1}z)z(A_{2}x + (-A_{0} + A_{2})y + A_{2}t)t(x + y + z + t)\)}

\noindent\parbox[right]{9mm}{\textbf{205:}}\parbox{3mm}{\hfill}%
\parbox[t]{15cm}{\(xyz(A_{0}x + A_{1}y - A_{1}z)(A_{0}x + A_{2}y - A_{0}z)(x + t)(y + t)(z + t)\)}

\noindent\parbox[right]{9mm}{\textbf{206:}}\parbox{3mm}{\hfill}%
\parbox[t]{15cm}{\(xy(A_{2}x + A_{0}y + A_{0}z)(A_{1}x + (A_{0} + A_{1} - A_{2})y + A_{1}z)zt(A_{2}x + A_{0}y + (-A_{1} + A_{2})t)(x + y + z + t)\)}

\noindent\parbox[right]{9mm}{\textbf{207:}}\parbox{3mm}{\hfill}%
\parbox[t]{15cm}{\(x(x + t)y(A_{0}x + A_{1}y + A_{2}t)(y + t)z(z + t)(A_{1}x - A_{1}y + A_{2}z)\)}

\noindent\parbox[right]{9mm}{\textbf{208:}}\parbox{3mm}{\hfill}%
\parbox[t]{15cm}{\(x(A_{0}x + A_{1}y + A_{2}z)y(A_{2}x + (-A_{0} + A_{1} + A_{2})y + A_{2}z)z(A_{1}x + A_{1}y + A_{2}z + A_{2}t)t(x + y + z + t)\)}

\noindent\parbox[right]{9mm}{\textbf{209:}}\parbox{3mm}{\hfill}%
\parbox[t]{15cm}{\(x(A_{0}x + A_{1}y + A_{1}z)y(A_{1}x + (-A_{0} + 2A_{1})y + A_{1}z)zt(x + y + z + t)(A_{0}x + A_{0}y + A_{2}t)\)}

\noindent\parbox[right]{9mm}{\textbf{210:}}\parbox{3mm}{\hfill}%
\parbox[t]{15cm}{\(xy((A_{1} - A_{3})x + (A_{2} - A_{3})y + (A_{2} - A_{3})z)((-A_{1} + A_{3})x + A_{0}y + A_{3}z)zt(A_{1}x + A_{2}y + A_{3}t)(x + y + z + t)\)}

\noindent\parbox[right]{9mm}{\textbf{211:}}\parbox{3mm}{\hfill}%
\parbox[t]{15cm}{\(xy(x + t)(A_{0}x + A_{1}y + A_{2}t)(y + t)z(A_{3}x + A_{1}y - A_{2}z)(z + t)\)}

\noindent\parbox[right]{9mm}{\textbf{212:}}\parbox{3mm}{\hfill}%
\parbox[t]{15cm}{\(xy(y + t)(x + t)(A_{0}x + A_{1}y + A_{2}t)(A_{0}x + A_{3}y - A_{3}z)z(z + t)\)}

\noindent\parbox[right]{9mm}{\textbf{213:}}\parbox{3mm}{\hfill}%
\parbox[t]{15cm}{\(xy(x + t)(y + t)(A_{0}x + A_{1}y + A_{2}t)z(A_{0}x + A_{1}y + A_{3}z)(z + t)\)}

\noindent\parbox[right]{9mm}{\textbf{214:}}\parbox{3mm}{\hfill}%
\parbox[t]{15cm}{\(x(x + t)y(A_{0}x + A_{1}y + A_{2}t)(y + t)z(z + t)((A_{0} - A_{2})x + A_{1}y + A_{3}z)\)}

\noindent\parbox[right]{9mm}{\textbf{215:}}\parbox{3mm}{\hfill}%
\parbox[t]{15cm}{\(xy(x + t)(y + t)(A_{0}x + A_{1}y + A_{2}t)z(A_{0}x + A_{3}y + (-A_{0} - A_{3})z)(z + t)\)}

\noindent\parbox[right]{9mm}{\textbf{216:}}\parbox{3mm}{\hfill}%
\parbox[t]{15cm}{\(xy(A_{2}x + A_{0}y + A_{0}z)((-A_{2} + A_{3})x + A_{1}y + A_{3}z)zt(A_{2}x + A_{2}y + A_{3}t)(x + y + z + t)\)}

\noindent\parbox[right]{9mm}{\textbf{217:}}\parbox{3mm}{\hfill}%
\parbox[t]{15cm}{\(x(x + t)y(y + t)(A_{0}x + A_{1}y + A_{2}t)z(z + t)(A_{0}x - A_{0}y + A_{3}z)\)}

\noindent\parbox[right]{9mm}{\textbf{218:}}\parbox{3mm}{\hfill}%
\parbox[t]{15cm}{\(xy(x + t)(y + t)(A_{0}x + A_{1}y + A_{2}t)z(A_{3}x + A_{1}y + (A_{0} - A_{2} - A_{3})z)(z + t)\)}

\noindent\parbox[right]{9mm}{\textbf{219:}}\parbox{3mm}{\hfill}%
\parbox[t]{15cm}{\(xyz(A_{0}x + A_{1}y + A_{2}z)(A_{3}x + (A_{0} + A_{1} - A_{3})y + A_{2}z)(x + t)(y + t)(z + t)\)}

\noindent\parbox[right]{9mm}{\textbf{220:}}\parbox{3mm}{\hfill}%
\parbox[t]{15cm}{\(xy(A_{0}x + A_{1}y + A_{1}z)(A_{0}x + A_{2}y + A_{3}z)zt((-A_{2} + A_{3})x + (-A_{2} + A_{3})y + A_{3}t)(x + y + z + t)\)}

\noindent\parbox[right]{9mm}{\textbf{221:}}\parbox{3mm}{\hfill}%
\parbox[t]{15cm}{\(xyz(A_{0}x + A_{1}y + (-A_{0} - A_{1})z)(A_{0}x + A_{2}y + A_{3}z)(x + t)(y + t)(z + t)\)}

\noindent\parbox[right]{9mm}{\textbf{222:}}\parbox{3mm}{\hfill}%
\parbox[t]{15cm}{\(xy(A_{0}x + A_{1}y + A_{2}z)z(A_{0}x + A_{3}y + A_{4}z)((-A_{1} + A_{2})x + (-A_{1} + A_{2})y + A_{2}t)t(x + y + z + t)\)}

\noindent\parbox[right]{9mm}{\textbf{223:}}\parbox{3mm}{\hfill}%
\parbox[t]{15cm}{\(xy(A_{0}x + A_{1}y + A_{1}z)(A_{0}x + A_{2}y + A_{0}z)zt(A_{0}x + A_{3}y + A_{4}t)(x + y + z + t)\)}

\noindent\parbox[right]{9mm}{\textbf{224:}}\parbox{3mm}{\hfill}%
\parbox[t]{15cm}{\(xy(A_{0}x + A_{1}y + A_{2}z)z(A_{3}x + (-A_{1} + A_{2})y + A_{4}z)(A_{3}x + (-A_{1} + A_{2})y + A_{2}t)t(x + y + z + t)\)}

\noindent\parbox[right]{9mm}{\textbf{225:}}\parbox{3mm}{\hfill}%
\parbox[t]{15cm}{\(xy((A_{1} - A_{3})x + (A_{2} - A_{3})y + A_{0}z)z(A_{1}x + A_{2}y + A_{3}z)t((A_{1} - A_{3})x + (A_{2} - A_{3})y + A_{4}t)(x + y + z + t)\)}

\noindent\parbox[right]{9mm}{\textbf{226:}}\parbox{3mm}{\hfill}%
\parbox[t]{15cm}{\(xy(A_{0}x + A_{1}y + A_{2}z)z(A_{0}x + A_{0}y + A_{3}z)(A_{4}x + (-A_{1} + A_{2})y + A_{2}t)t(x + y + z + t)\)}

\noindent\parbox[right]{9mm}{\textbf{227:}}\parbox{3mm}{\hfill}%
\parbox[t]{15cm}{\(xy(A_{0}x + A_{1}y + A_{1}z)(A_{0}x + A_{2}y + A_{3}z)zt(A_{4}x + (-A_{2} + A_{3})y + A_{3}t)(x + y + z + t)\)}

\noindent\parbox[right]{9mm}{\textbf{228:}}\parbox{3mm}{\hfill}%
\parbox[t]{15cm}{\(xy(A_{1}x + A_{0}y + A_{0}z)(A_{1}x + A_{2}y + A_{3}z)zt((-A_{1} + A_{3})x + A_{4}y + A_{3}t)(x + y + z + t)\)}

\noindent\parbox[right]{9mm}{\textbf{229:}}\parbox{3mm}{\hfill}%
\parbox[t]{15cm}{\(xy(A_{0}x + A_{1}y + A_{1}z)(A_{0}x + A_{2}y + A_{3}z)zt((-A_{0} + A_{1})x + A_{4}y + A_{1}t)(x + y + z + t)\)}

\noindent\parbox[right]{9mm}{\textbf{230:}}\parbox{3mm}{\hfill}%
\parbox[t]{15cm}{\(xy(A_{0}x + A_{0}y + A_{1}z)z(A_{0}x + A_{2}y + A_{3}z)t((A_{0} - A_{3})x + (A_{2} - A_{3})y + A_{4}t)(x + y + z + t)\)}

\noindent\parbox[right]{9mm}{\textbf{231:}}\parbox{3mm}{\hfill}%
\parbox[t]{15cm}{\(xy(A_{0}x + A_{1}y + A_{2}z)z(A_{0}x + A_{3}y + A_{4}z)t(A_{0}x + A_{1}y + A_{5}t)(x + y + z + t)\)}

\noindent\parbox[right]{9mm}{\textbf{232:}}\parbox{3mm}{\hfill}%
\parbox[t]{15cm}{\(xy(A_{0}x + A_{0}y + A_{1}z)z(A_{0}x + A_{2}y + A_{3}z)t(A_{0}x + A_{4}y + A_{5}t)(x + y + z + t)\)}

\noindent\parbox[right]{9mm}{\textbf{233:}}\parbox{3mm}{\hfill}%
\parbox[t]{15cm}{\(xy(A_{0}x + A_{1}y + A_{2}z)z(A_{0}x + A_{3}y + A_{4}z)(A_{5}x + (-A_{1} + A_{2})y + A_{2}t)t(x + y + z + t)\)}

\noindent\parbox[right]{9mm}{\textbf{234:}}\parbox{3mm}{\hfill}%
\parbox[t]{15cm}{\(xy(A_{0}x + A_{1}y + A_{2}z)z(A_{0}x + A_{3}y + A_{4}z)t((A_{0} - A_{2})x + (A_{1} - A_{2})y + A_{5}t)(x + y + z + t)\)}

\noindent\parbox[right]{9mm}{\textbf{235:}}\parbox{3mm}{\hfill}%
\parbox[t]{15cm}{\(xy(A_{0}x + A_{1}y + A_{2}z)z(A_{0}x + A_{3}y + A_{4}z)t(A_{0}x + A_{5}y + A_{6}t)(x + y + z + t)\)}

\noindent\parbox[right]{9mm}{\textbf{236:}}\parbox{3mm}{\hfill}%
\parbox[t]{15cm}{\(xyz(A_{0}x + A_{1}y + A_{2}z)(A_{0}x + A_{3}y + A_{4}z)(A_{0}x + A_{5}y + A_{5}z + A_{6}t)t(x + y + z + t)\)}

\noindent\parbox[right]{9mm}{\textbf{237:}}\parbox{3mm}{\hfill}%
\parbox[t]{15cm}{\(xyz(A_{0}x + A_{1}y + A_{2}z)(A_{0}x + A_{3}y + A_{4}z)(A_{0}x + A_{5}y + A_{6}z + A_{7}t)t(x + y + z + t)\)}

\noindent\parbox[right]{9mm}{\textbf{238:}}\parbox{3mm}{\hfill}%
\parbox[t]{15cm}{\(xyz(x - y - z)(x + t)(y + t)(z + t)(x - y - z - t)\)}

\noindent\parbox[right]{9mm}{\textbf{239:}}\parbox{3mm}{\hfill}%
\parbox[t]{15cm}{\(xyz(x + y - z)(z + t)(x + y + z + t)(x + t)(y + t)\)}

\noindent\parbox[right]{9mm}{\textbf{240:}}\parbox{3mm}{\hfill}%
\parbox[t]{15cm}{\(xyz(x - y - z)(x + t)(y + t)(z + t)(x - 2y - 2z - 2t)\)}

\noindent\parbox[right]{9mm}{\textbf{241:}}\parbox{3mm}{\hfill}%
\parbox[t]{15cm}{\(xyz(x - y - z)(x - y + t)t(x + y - z - t)(x + y + z + t)\)}

\noindent\parbox[right]{9mm}{\textbf{242:}}\parbox{3mm}{\hfill}%
\parbox[t]{15cm}{\(xy(x + t)(y + t)(A_{0}x - A_{0}y + A_{1}z)z(A_{0}x - A_{0}y + A_{1}z + A_{1}t)(z + t)\)}

\noindent\parbox[right]{9mm}{\textbf{243:}}\parbox{3mm}{\hfill}%
\parbox[t]{15cm}{\(xyz(x - y - z)(x + t)(y + t)(z + t)(A_{0}x + A_{1}y + A_{1}z + A_{1}t)\)}

\noindent\parbox[right]{9mm}{\textbf{244:}}\parbox{3mm}{\hfill}%
\parbox[t]{15cm}{\(xy(x + t)(y + t)(A_{0}x + A_{1}y + (-A_{0} - A_{1})z)z(A_{0}x + A_{1}y + (-A_{0} - A_{1})z + (-A_{0} - A_{1})t)(z + t)\)}

\noindent\parbox[right]{9mm}{\textbf{245:}}\parbox{3mm}{\hfill}%
\parbox[t]{15cm}{\(xyz(-2x - y + z)(z + t)(x + y + z + t)(x + t)(y + t)\)}

\noindent\parbox[right]{9mm}{\textbf{246:}}\parbox{3mm}{\hfill}%
\parbox[t]{15cm}{\(xy(A_{0}x + A_{1}y + (-A_{0} - A_{1})z)z(-A_{1}x - A_{0}y + (-A_{0} - A_{1})z + (-A_{0} - A_{1})t)(z + t)(y + t)(x + t)\)}

\noindent\parbox[right]{9mm}{\textbf{247:}}\parbox{3mm}{\hfill}%
\parbox[t]{15cm}{\(xyz(A_{0}x + A_{1}y + A_{1}z)(x + t)(y + t)(z + t)(-1x + y + z + t)\)}

\noindent\parbox[right]{9mm}{\textbf{248:}}\parbox{3mm}{\hfill}%
\parbox[t]{15cm}{\(xyz(A_{0}x + A_{1}y - A_{1}z)(z + t)(x + y + z + t)(x + t)(y + t)\)}

\noindent\parbox[right]{9mm}{\textbf{249:}}\parbox{3mm}{\hfill}%
\parbox[t]{15cm}{\(x(x + 2y + 2z)yz(A_{0}x + (2A_{0} - A_{1})y + 2A_{0}t)(A_{0}x + 2A_{0}y + A_{1}z + 2A_{0}t)t(x + y + z + t)\)}

\noindent\parbox[right]{9mm}{\textbf{250:}}\parbox{3mm}{\hfill}%
\parbox[t]{15cm}{\(xyz(-2A_{1}x + A_{0}y + A_{1}z)(z + t)(x - y - z - t)(x + t)(y + t)\)}

\noindent\parbox[right]{9mm}{\textbf{251:}}\parbox{3mm}{\hfill}%
\parbox[t]{15cm}{\(xyz(A_{0}x + A_{1}y + (-A_{0} - A_{1})z)(y + t)(x + t)(-A_{1}x + (-A_{0} - A_{1})y + (-A_{0} - A_{1})z + (-A_{0} - A_{1})t)(z + t)\)}

\noindent\parbox[right]{9mm}{\textbf{252:}}\parbox{3mm}{\hfill}%
\parbox[t]{15cm}{\(x(A_{0}x + A_{1}y + 2A_{1}z)yz(y + 2z + 2t)t(x + y + z + t)(x + z - t)\)}

\noindent\parbox[right]{9mm}{\textbf{253:}}\parbox{3mm}{\hfill}%
\parbox[t]{15cm}{\(xy(x + t)(y + t)z(A_{0}x + A_{1}y + (-A_{0} - A_{1})z)(z + t)((-A_{0} - A_{1})x + A_{1}y + (-A_{0} - A_{1})z + (-A_{0} - A_{1})t)\)}

\noindent\parbox[right]{9mm}{\textbf{254:}}\parbox{3mm}{\hfill}%
\parbox[t]{15cm}{\(xyz(A_{1}x + A_{0}y + A_{0}z)(A_{1}x + A_{0}y + A_{1}t)t(A_{0}x + A_{0}y + A_{1}z + A_{1}t)(x + y + z + t)\)}

\noindent\parbox[right]{9mm}{\textbf{255:}}\parbox{3mm}{\hfill}%
\parbox[t]{15cm}{\(x(A_{0}x + A_{1}y + A_{1}z)yz(A_{0}x + A_{1}y + A_{0}t)((A_{0} + A_{1})x + (A_{0} + A_{1})y + A_{0}z + A_{0}t)t(x + y + z + t)\)}

\noindent\parbox[right]{9mm}{\textbf{256:}}\parbox{3mm}{\hfill}%
\parbox[t]{15cm}{\(xyz(A_{0}x + A_{1}y - A_{0}z)(z + t)(A_{0}x + A_{1}y + A_{1}z + A_{1}t)(x + t)(y + t)\)}

\noindent\parbox[right]{9mm}{\textbf{257:}}\parbox{3mm}{\hfill}%
\parbox[t]{15cm}{\(xy(A_{0}x + A_{1}y + A_{1}z)z(A_{0}x + A_{1}y + (A_{0} - A_{1})t)t(A_{0}x + A_{0}y + (A_{0} - A_{1})z + (A_{0} - A_{1})t)(x + y + z + t)\)}

\noindent\parbox[right]{9mm}{\textbf{258:}}\parbox{3mm}{\hfill}%
\parbox[t]{15cm}{\(xyz(-2x + y + z)(x + t)(y + t)(z + t)(A_{0}x + A_{1}y + A_{1}z + A_{1}t)\)}

\noindent\parbox[right]{9mm}{\textbf{259:}}\parbox{3mm}{\hfill}%
\parbox[t]{15cm}{\(xy(x + t)(y + t)z(A_{0}x + A_{1}y + (-A_{0} - A_{1})z)(z + t)((-A_{0} - 2A_{1})x + A_{1}y + (-A_{0} - A_{1})z + (-A_{0} - A_{1})t)\)}

\noindent\parbox[right]{9mm}{\textbf{260:}}\parbox{3mm}{\hfill}%
\parbox[t]{15cm}{\((A_{0}A_{1}x + A_{0}A_{2}y + A_{1}A_{2}z + A_{2}^2t)xy(A_{1}x + A_{2}y + A_{1}z + A_{2}t)(A_{0}x + A_{0}y + A_{2}z + A_{2}t)zt(x + y + z + t)\)}

\noindent\parbox[right]{9mm}{\textbf{261:}}\parbox{3mm}{\hfill}%
\parbox[t]{15cm}{\(xyz((-A_{0} + 2A_{1})x + A_{0}y + A_{1}z)(2x + z + 2t)(A_{0}x + A_{0}y + A_{1}z + 2A_{1}t)t(x + y + z + t)\)}

\noindent\parbox[right]{9mm}{\textbf{262:}}\parbox{3mm}{\hfill}%
\parbox[t]{15cm}{\(x((A_{0} - A_{1})x + A_{0}y + A_{0}z)yz(A_{1}x + A_{1}y + A_{0}t)(A_{0}y + (A_{0} - A_{1})z + A_{0}t)t(x + y + z + t)\)}

\noindent\parbox[right]{9mm}{\textbf{263:}}\parbox{3mm}{\hfill}%
\parbox[t]{15cm}{\(xyz(A_{0}x + A_{2}y + A_{0}z)(A_{2}x + A_{1}y + A_{1}t)t(x + y + z + t)(A_{2}x + A_{2}y + A_{0}z + A_{1}t)\)}

\noindent\parbox[right]{9mm}{\textbf{264:}}\parbox{3mm}{\hfill}%
\parbox[t]{15cm}{\(x(A_{0}x + A_{1}y + 2A_{1}z)yz(y + 2z + 2t)t(x + y + z + t)(A_{0}x + A_{1}y + A_{0}z + (-A_{0} + 2A_{1})t)\)}

\noindent\parbox[right]{9mm}{\textbf{265:}}\parbox{3mm}{\hfill}%
\parbox[t]{15cm}{\(x(x + t)y(y + t)(A_{0}x + A_{1}y - A_{0}z)(-A_{0}x + A_{1}y - A_{0}z + (-A_{0} + A_{1})t)z(z + t)\)}

\noindent\parbox[right]{9mm}{\textbf{266:}}\parbox{3mm}{\hfill}%
\parbox[t]{15cm}{\(xyz(x - y + z)(A_{0}x + A_{1}y + A_{1}t)t(x + y + z + t)(A_{0}x + A_{1}z + A_{1}t)\)}

\noindent\parbox[right]{9mm}{\textbf{267:}}\parbox{3mm}{\hfill}%
\parbox[t]{15cm}{\(xyz(A_{0}^2x + (A_{0}A_{1} - A_{1}^2)y + A_{0}A_{1}z)((A_{0}^2 - A_{0}A_{1})x - A_{1}^2y - A_{1}^2t)t(x + y + z + t)(A_{0}x + A_{0}y + A_{1}z + A_{1}t)\)}

\noindent\parbox[right]{9mm}{\textbf{268:}}\parbox{3mm}{\hfill}%
\parbox[t]{15cm}{\(x(A_{0}x + A_{1}y + A_{1}z)yz((-A_{0} + A_{1})x + (-A_{0} + 2A_{1})y + A_{1}t)(A_{1}y + (A_{0} - A_{1})z + A_{1}t)t(x + y + z + t)\)}

\noindent\parbox[right]{9mm}{\textbf{269:}}\parbox{3mm}{\hfill}%
\parbox[t]{15cm}{\(x(A_{2}x + A_{1}y + A_{1}z + A_{2}t)(A_{0}A_{1}x + A_{1}A_{2}y + A_{0}A_{1}z + A_{0}A_{2}t)y(A_{1}A_{2}x + A_{1}A_{2}y + A_{0}A_{1}z + A_{2}^2t)zt(x + y + z + t)\)}

\noindent\parbox[right]{9mm}{\textbf{270:}}\parbox{3mm}{\hfill}%
\parbox[t]{15cm}{\(xy(A_{1}x + A_{1}y + 2A_{0}z)z(A_{0}x + (2A_{0} - A_{1})y + 2A_{0}t)t(x + y + z + t)(A_{0}x + A_{1}y + 2A_{0}z + 2A_{0}t)\)}

\noindent\parbox[right]{9mm}{\textbf{271:}}\parbox{3mm}{\hfill}%
\parbox[t]{15cm}{\(x(x + t)(A_{0}x + A_{1}y - A_{1}z)(A_{0}x + A_{2}y - A_{2}z + A_{0}t)y(y + t)z(z + t)\)}

\noindent\parbox[right]{9mm}{\textbf{272:}}\parbox{3mm}{\hfill}%
\parbox[t]{15cm}{\((A_{0}x + A_{1}y + (-A_{0} - A_{1})z)(A_{0}x + A_{1}y + (-A_{0} - A_{1})z + A_{2}t)x(x + t)y(y + t)z(z + t)\)}

\noindent\parbox[right]{9mm}{\textbf{273:}}\parbox{3mm}{\hfill}%
\parbox[t]{15cm}{\(xyz(A_{0}x + A_{1}y + (-A_{0} + A_{1})z)(z + t)(x + y + z + t)(x + t)(y + t)\)}

\noindent\parbox[right]{9mm}{\textbf{274:}}\parbox{3mm}{\hfill}%
\parbox[t]{15cm}{\(xyz(A_{0}x + A_{1}y - A_{0}z)(z + t)((A_{0} + A_{1})x - A_{0}y - A_{0}z - A_{0}t)(x + t)(y + t)\)}

\noindent\parbox[right]{9mm}{\textbf{275:}}\parbox{3mm}{\hfill}%
\parbox[t]{15cm}{\(xy(A_{1}x + (-A_{0} + A_{1})y + A_{1}z)zt((-A_{0}^2 + A_{0}A_{1})x + (-A_{0}^2 + A_{0}A_{1})y + A_{1}^2t)(A_{1}^2x + (-A_{0}^2 + A_{0}A_{1})y + A_{0}A_{1}z + A_{1}^2t)(x + y + z + t)\)}

\noindent\parbox[right]{9mm}{\textbf{276:}}\parbox{3mm}{\hfill}%
\parbox[t]{15cm}{\(xy(A_{1}x + (A_{0} + A_{1})y + (A_{0} + A_{1})z)zt(A_{0}x + A_{0}y + A_{1}t)(x + y + z + t)(A_{1}x + A_{0}y + (A_{0} + A_{1})z + A_{1}t)\)}

\noindent\parbox[right]{9mm}{\textbf{277:}}\parbox{3mm}{\hfill}%
\parbox[t]{15cm}{\(xy(x + t)(y + t)z(A_{0}x + A_{1}y + A_{2}z)(z + t)((-A_{1} + A_{2})x + A_{1}y + A_{2}z + A_{2}t)\)}

\noindent\parbox[right]{9mm}{\textbf{278:}}\parbox{3mm}{\hfill}%
\parbox[t]{15cm}{\(xyz(A_{0}x + A_{1}y + (-A_{0} - A_{1})z)(z + t)(A_{0}x + A_{2}y + (A_{0} + A_{2})z + (A_{0} + A_{2})t)(x + t)(y + t)\)}

\noindent\parbox[right]{9mm}{\textbf{279:}}\parbox{3mm}{\hfill}%
\parbox[t]{15cm}{\(xyz(A_{0}x + A_{1}y + A_{1}z)(A_{0}x + A_{1}y + A_{0}t)t(A_{2}x + A_{1}y + A_{0}z + A_{0}t)(x + y + z + t)\)}

\noindent\parbox[right]{9mm}{\textbf{280:}}\parbox{3mm}{\hfill}%
\parbox[t]{15cm}{\(xyz(A_{0}x + A_{1}y - A_{0}z)(z + t)(A_{0}x + A_{2}y + A_{2}z + A_{2}t)(x + t)(y + t)\)}

\noindent\parbox[right]{9mm}{\textbf{281:}}\parbox{3mm}{\hfill}%
\parbox[t]{15cm}{\(xyz(A_{0}x + A_{1}y + A_{0}z)t(A_{0}x + A_{2}y + A_{2}t)((-A_{0} + A_{2})x + A_{2}z + A_{2}t)(x + y + z + t)\)}

\noindent\parbox[right]{9mm}{\textbf{282:}}\parbox{3mm}{\hfill}%
\parbox[t]{15cm}{\(x(x + t)(x + y - z)(A_{0}x + A_{1}y - A_{1}z + A_{2}t)y(y + t)z(z + t)\)}

\noindent\parbox[right]{9mm}{\textbf{283:}}\parbox{3mm}{\hfill}%
\parbox[t]{15cm}{\(x((A_{0} + A_{2})x + A_{0}y + A_{0}z)yzt(x + y + z + t)(A_{1}x + A_{2}y + A_{1}t)((A_{0} + A_{2})x + (A_{0} + A_{2})y + A_{0}z + A_{1}t)\)}

\noindent\parbox[right]{9mm}{\textbf{284:}}\parbox{3mm}{\hfill}%
\parbox[t]{15cm}{\(xyz(A_{0}x + A_{1}y + A_{1}z)(x + t)(y + t)(z + t)(A_{0}x + A_{2}y + A_{2}z + A_{2}t)\)}

\noindent\parbox[right]{9mm}{\textbf{285:}}\parbox{3mm}{\hfill}%
\parbox[t]{15cm}{\(xyz(A_{0}x + A_{1}y + A_{1}z)(A_{2}x + A_{1}y + A_{2}t)t(A_{0}x + A_{1}y + A_{2}z + A_{2}t)(x + y + z + t)\)}

\noindent\parbox[right]{9mm}{\textbf{286:}}\parbox{3mm}{\hfill}%
\parbox[t]{15cm}{\(x(A_{0}x + A_{1}z + 2A_{1}t)(2y + z + 2t)yzt(x + y + z + t)(A_{0}x + (-A_{0} + 2A_{2})y + A_{2}z + A_{0}t)\)}

\noindent\parbox[right]{9mm}{\textbf{287:}}\parbox{3mm}{\hfill}%
\parbox[t]{15cm}{\(x(A_{0}x + A_{1}y + A_{1}z + A_{2}t)(A_{1}x + A_{1}z + A_{2}t)yz(A_{1}x + A_{1}y + A_{3}z + A_{2}t)t(x + y + z + t)\)}

\noindent\parbox[right]{9mm}{\textbf{288:}}\parbox{3mm}{\hfill}%
\parbox[t]{15cm}{\(xyz(A_{2}x + (-A_{1} + A_{2})y + (-A_{1} + A_{2})z)(A_{0}x + A_{0}y + A_{2}t)t(A_{0}y + A_{1}z + A_{2}t)(x + y + z + t)\)}

\noindent\parbox[right]{9mm}{\textbf{289:}}\parbox{3mm}{\hfill}%
\parbox[t]{15cm}{\(xyz(A_{0}x + A_{0}y + A_{2}z)(A_{1}x + A_{2}y + A_{2}t)t(x + y + z + t)(A_{1}^2x + (2A_{1}A_{2} - A_{2}^2)y + A_{1}A_{2}z + A_{1}A_{2}t)\)}

\noindent\parbox[right]{9mm}{\textbf{290:}}\parbox{3mm}{\hfill}%
\parbox[t]{15cm}{\((A_{0}x + A_{1}y + A_{0}z)xyzt(A_{2}x + A_{1}y + A_{0}z + A_{2}t)(x + y + z + t)((A_{0} - A_{1})x + (A_{0} - A_{1})y + A_{0}t)\)}

\noindent\parbox[right]{9mm}{\textbf{291:}}\parbox{3mm}{\hfill}%
\parbox[t]{15cm}{\(xyz((A_{0} - A_{2})x + (A_{1} - A_{2})y + (A_{1} - A_{2})z)(A_{2}x + A_{1}y + A_{2}t)t(A_{0}x + A_{1}y + A_{2}z + A_{2}t)(x + y + z + t)\)}

\noindent\parbox[right]{9mm}{\textbf{292:}}\parbox{3mm}{\hfill}%
\parbox[t]{15cm}{\(xyz(A_{0}x + A_{2}y + A_{0}z)(A_{1}x + A_{0}z + A_{1}t)(A_{2}x + A_{2}y + A_{0}z + A_{1}t)t(x + y + z + t)\)}

\noindent\parbox[right]{9mm}{\textbf{293:}}\parbox{3mm}{\hfill}%
\parbox[t]{15cm}{\(x(A_{0}x + A_{1}y + A_{1}z)yzt(x + y + z + t)(A_{0}x + A_{2}y + A_{0}t)(A_{0}y + (A_{0} - A_{2})z + A_{0}t)\)}

\noindent\parbox[right]{9mm}{\textbf{294:}}\parbox{3mm}{\hfill}%
\parbox[t]{15cm}{\(x(A_{0}x + A_{1}y + A_{1}z)yz(A_{0}x + A_{2}y + A_{0}t)((A_{0} - A_{2})x + A_{0}y + (A_{0} - A_{2})z + A_{0}t)t(x + y + z + t)\)}

\noindent\parbox[right]{9mm}{\textbf{295:}}\parbox{3mm}{\hfill}%
\parbox[t]{15cm}{\(xyz(A_{0}x + A_{1}y + A_{1}z)(A_{0}x + A_{2}y + A_{0}t)t(A_{2}x + A_{2}y + A_{0}z + A_{0}t)(x + y + z + t)\)}

\noindent\parbox[right]{9mm}{\textbf{296:}}\parbox{3mm}{\hfill}%
\parbox[t]{15cm}{\(xyz(A_{0}x + A_{1}y + A_{1}z)(A_{0}x + A_{2}y + A_{2}t)t(x + y + z + t)(A_{0}x + A_{0}y + A_{1}z + A_{2}t)\)}

\noindent\parbox[right]{9mm}{\textbf{297:}}\parbox{3mm}{\hfill}%
\parbox[t]{15cm}{\(xy(A_{0}x + A_{1}y + A_{1}z)z(A_{0}x + A_{0}y + A_{2}t)t(A_{0}x + (A_{0} + A_{2})y + A_{2}z + A_{2}t)(x + y + z + t)\)}

\noindent\parbox[right]{9mm}{\textbf{298:}}\parbox{3mm}{\hfill}%
\parbox[t]{15cm}{\(xyz(A_{0}x + A_{1}y + A_{0}z)(A_{0}x + A_{1}y + A_{2}t)t(A_{1}x + A_{1}y + A_{2}z + A_{2}t)(x + y + z + t)\)}

\noindent\parbox[right]{9mm}{\textbf{299:}}\parbox{3mm}{\hfill}%
\parbox[t]{15cm}{\(xyz(A_{0}x + A_{0}y + A_{1}z)t(A_{0}x + A_{2}y + A_{0}t)((-A_{0} + A_{2})x + A_{2}z + A_{2}t)(x + y + z + t)\)}

\noindent\parbox[right]{9mm}{\textbf{300:}}\parbox{3mm}{\hfill}%
\parbox[t]{15cm}{\(x(A_{0}x + A_{1}y + A_{1}z)yz((A_{0} + A_{2})x + (A_{0} + A_{2})y + A_{2}t)((A_{0} + A_{2})x + (A_{0} + A_{1} + 2A_{2})y + (A_{1} + A_{2})z + A_{2}t)t(x + y + z + t)\)}

\noindent\parbox[right]{9mm}{\textbf{301:}}\parbox{3mm}{\hfill}%
\parbox[t]{15cm}{\(x(A_{2}x + (-A_{1} + A_{2})y + (-A_{1} + A_{2})z)yzt(x + y + z + t)((A_{0} - A_{1})x + (A_{0} - A_{1})y + A_{2}t)(A_{0}y + A_{1}z + A_{2}t)\)}

\noindent\parbox[right]{9mm}{\textbf{302:}}\parbox{3mm}{\hfill}%
\parbox[t]{15cm}{\(xyz((A_{0} - A_{2})x + (A_{1} - A_{2})y + (A_{1} - A_{2})z)(A_{2}x + A_{0}y + A_{2}t)t(A_{0}x + A_{0}y + A_{1}z + A_{2}t)(x + y + z + t)\)}

\noindent\parbox[right]{9mm}{\textbf{303:}}\parbox{3mm}{\hfill}%
\parbox[t]{15cm}{\(xy(A_{0}x + A_{1}y + A_{1}z)zt(A_{0}x + A_{0}y + A_{2}t)(x + y + z + t)((A_{0} + A_{2})y + A_{2}z + A_{2}t)\)}

\noindent\parbox[right]{9mm}{\textbf{304:}}\parbox{3mm}{\hfill}%
\parbox[t]{15cm}{\(xy(A_{1}x + (-A_{0} + A_{1})y + A_{1}z)z(A_{0}x + A_{0}y + A_{1}t)t(x + y + z + t)(A_{2}x + A_{0}z + A_{2}t)\)}

\noindent\parbox[right]{9mm}{\textbf{305:}}\parbox{3mm}{\hfill}%
\parbox[t]{15cm}{\(xyz(x - y + z)(A_{0}x + A_{1}y + A_{1}t)t(x + y + z + t)(A_{0}x + A_{2}z + A_{2}t)\)}

\noindent\parbox[right]{9mm}{\textbf{306:}}\parbox{3mm}{\hfill}%
\parbox[t]{15cm}{\(x(A_{0}x + A_{0}y + (A_{0} + A_{1})z)yz(A_{1}x + A_{2}y + (A_{0} + A_{1})t)(-A_{0}x + A_{2}z - A_{0}t)t(x + y + z + t)\)}

\noindent\parbox[right]{9mm}{\textbf{307:}}\parbox{3mm}{\hfill}%
\parbox[t]{15cm}{\(x(x + t)y(y + t)z(z + t)(A_{0}x - A_{0}y + A_{1}z)(A_{2}x - A_{0}y + A_{1}z + A_{1}t)\)}

\noindent\parbox[right]{9mm}{\textbf{308:}}\parbox{3mm}{\hfill}%
\parbox[t]{15cm}{\(xy(A_{1}x + A_{0}y + A_{1}z)z(-A_{0}x - A_{0}y + A_{2}t)t(A_{2}x + A_{1}z + A_{2}t)(x + y + z + t)\)}

\noindent\parbox[right]{9mm}{\textbf{309:}}\parbox{3mm}{\hfill}%
\parbox[t]{15cm}{\(xyz(A_{0}x + A_{1}y + A_{2}z)(z + t)(x + y + z + t)(x + t)(y + t)\)}

\noindent\parbox[right]{9mm}{\textbf{310:}}\parbox{3mm}{\hfill}%
\parbox[t]{15cm}{\(xyz(A_{0}x + A_{1}y + A_{0}z)t(A_{0}x + A_{2}y + A_{2}t)((A_{0} + A_{2})x + A_{2}z + A_{2}t)(x + y + z + t)\)}

\noindent\parbox[right]{9mm}{\textbf{311:}}\parbox{3mm}{\hfill}%
\parbox[t]{15cm}{\(xyz((A_{0} - A_{2})x + (A_{1} - A_{2})y + (A_{0} - A_{2})z)(A_{0}x + A_{1}y + A_{2}t)t(A_{1}x + A_{1}y + A_{2}z + A_{2}t)(x + y + z + t)\)}

\noindent\parbox[right]{9mm}{\textbf{312:}}\parbox{3mm}{\hfill}%
\parbox[t]{15cm}{\(xy(y + t)(x + t)(A_{0}x + A_{1}y - A_{1}z)z(z + t)(A_{2}x + A_{0}y + A_{1}z + (A_{0} + A_{2})t)\)}

\noindent\parbox[right]{9mm}{\textbf{313:}}\parbox{3mm}{\hfill}%
\parbox[t]{15cm}{\(x(A_{0}x + A_{1}y + A_{1}z)((-A_{0} + A_{1})x + A_{1}y + A_{1}t)(A_{1}y + A_{2}z + (A_{1} - A_{2})t)yzt(x + y + z + t)\)}

\noindent\parbox[right]{9mm}{\textbf{314:}}\parbox{3mm}{\hfill}%
\parbox[t]{15cm}{\(xy(y + t)(x + t)(A_{0}x + A_{1}y - A_{1}z)z(z + t)(A_{0}x + A_{2}y + A_{2}z + (A_{0} + A_{2})t)\)}

\noindent\parbox[right]{9mm}{\textbf{315:}}\parbox{3mm}{\hfill}%
\parbox[t]{15cm}{\(xyz(A_{0}x + A_{1}y + (-A_{0} - A_{1})z)(y + t)(x + t)(A_{0}x + A_{2}y + A_{2}z + A_{2}t)(z + t)\)}

\noindent\parbox[right]{9mm}{\textbf{316:}}\parbox{3mm}{\hfill}%
\parbox[t]{15cm}{\(xyz(A_{1}x + A_{1}y + (A_{1} - A_{2})z)t(A_{0}x + A_{1}y + A_{2}t)((-A_{0} + A_{1})x + (A_{1} - A_{2})z + (A_{1} - A_{2})t)(x + y + z + t)\)}

\noindent\parbox[right]{9mm}{\textbf{317:}}\parbox{3mm}{\hfill}%
\parbox[t]{15cm}{\(xy(A_{0}^2x + A_{1}A_{2}y + (A_{1}^2 + A_{1}A_{2})z)z(A_{0}^2x + A_{1}A_{2}y + (A_{1}A_{2} + A_{2}^2)t)t(x + y + z + t)(A_{0}^2x + A_{0}^2y + A_{1}A_{2}z + A_{1}A_{2}t)\)}

\noindent\parbox[right]{9mm}{\textbf{318:}}\parbox{3mm}{\hfill}%
\parbox[t]{15cm}{\(xy(y + t)(x + t)(x + y - z)z(z + t)(A_{0}x + A_{1}y + A_{2}z + (A_{0} + A_{1} + 2A_{2})t)\)}

\noindent\parbox[right]{9mm}{\textbf{319:}}\parbox{3mm}{\hfill}%
\parbox[t]{15cm}{\(xyz(A_{0}x + A_{1}y + A_{0}z)(A_{0}x + A_{2}y + A_{2}t)t(x + y + z + t)(A_{0}x + A_{2}z + A_{2}t)\)}

\noindent\parbox[right]{9mm}{\textbf{320:}}\parbox{3mm}{\hfill}%
\parbox[t]{15cm}{\(xy(A_{2}y + A_{0}z + A_{1}t)(A_{0}x + A_{1}y + A_{0}z + A_{1}t)(A_{2}x + A_{2}y + A_{0}z)zt(x + y + z + t)\)}

\noindent\parbox[right]{9mm}{\textbf{321:}}\parbox{3mm}{\hfill}%
\parbox[t]{15cm}{\(xyz(A_{0}x + A_{1}y - A_{1}z)(z + t)(A_{0}x + A_{2}y + A_{2}z + A_{2}t)(x + t)(y + t)\)}

\noindent\parbox[right]{9mm}{\textbf{322:}}\parbox{3mm}{\hfill}%
\parbox[t]{15cm}{\(xyz(A_{0}x + A_{1}y + A_{2}z)(A_{0}x + A_{1}y + A_{1}t)t(x + y + z + t)(A_{1}x + A_{1}y + A_{2}z + A_{2}t)\)}

\noindent\parbox[right]{9mm}{\textbf{323:}}\parbox{3mm}{\hfill}%
\parbox[t]{15cm}{\(xyz((A_{0} - A_{1})x + A_{0}y + A_{0}z)(A_{1}x + A_{1}y + A_{0}t)t(A_{1}y + A_{2}z + A_{0}t)(x + y + z + t)\)}

\noindent\parbox[right]{9mm}{\textbf{324:}}\parbox{3mm}{\hfill}%
\parbox[t]{15cm}{\(xy(A_{0}A_{2}x + A_{0}A_{1}z + (A_{0}A_{1} + A_{1}A_{2})t)((A_{0}A_{1} + A_{1}A_{2})x + (A_{0}A_{1} + A_{1}A_{2} - A_{2}^2)y + A_{0}A_{1}z + (A_{0}A_{1} + A_{1}A_{2})t)z(A_{2}x + A_{2}y + A_{1}z)t(x + y + z + t)\)}

\noindent\parbox[right]{9mm}{\textbf{325:}}\parbox{3mm}{\hfill}%
\parbox[t]{15cm}{\(xyz(A_{0}x + (A_{1} + A_{2})y + A_{0}z)((A_{1}^2 + A_{1}A_{2})x + A_{0}A_{2}z + A_{1}A_{2}t)((A_{1}A_{2} + A_{2}^2)y + A_{0}A_{2}z + A_{1}A_{2}t)t(x + y + z + t)\)}

\noindent\parbox[right]{9mm}{\textbf{326:}}\parbox{3mm}{\hfill}%
\parbox[t]{15cm}{\(xy(A_{2}x + A_{0}y + A_{0}z)z((A_{0} - A_{1})x + A_{1}y + (A_{0} - A_{1})t)t(A_{2}x + A_{0}y + (A_{0} - A_{1})z + (A_{0} - A_{1})t)(x + y + z + t)\)}

\noindent\parbox[right]{9mm}{\textbf{327:}}\parbox{3mm}{\hfill}%
\parbox[t]{15cm}{\(xyz((-A_{0} + A_{2})x + (-A_{0} + A_{2})y + A_{2}z)t(A_{0}x + A_{1}y + A_{2}t)((A_{0} - A_{1})x + A_{2}z + A_{2}t)(x + y + z + t)\)}

\noindent\parbox[right]{9mm}{\textbf{328:}}\parbox{3mm}{\hfill}%
\parbox[t]{15cm}{\(xy(A_{1}x + A_{2}y + A_{2}z)zt(A_{2}x + A_{0}y + A_{2}t)(x + y + z + t)(A_{1}x + A_{0}y + A_{2}z + A_{2}t)\)}

\noindent\parbox[right]{9mm}{\textbf{329:}}\parbox{3mm}{\hfill}%
\parbox[t]{15cm}{\(xyz(A_{0}x - A_{0}y + A_{1}z)(z + t)(A_{0}x + A_{2}y + A_{3}z + A_{3}t)(x + t)(y + t)\)}

\noindent\parbox[right]{9mm}{\textbf{330:}}\parbox{3mm}{\hfill}%
\parbox[t]{15cm}{\(xyz(A_{0}x + A_{0}y + A_{1}z)t(A_{0}x + A_{2}y + A_{0}t)(A_{0}x + A_{3}z + A_{3}t)(x + y + z + t)\)}

\noindent\parbox[right]{9mm}{\textbf{331:}}\parbox{3mm}{\hfill}%
\parbox[t]{15cm}{\(xyz(A_{2}x + A_{0}y + A_{1}z)(z + t)(A_{2}x + A_{3}y + (A_{2} + A_{3})z + (A_{2} + A_{3})t)(x + t)(y + t)\)}

\noindent\parbox[right]{9mm}{\textbf{332:}}\parbox{3mm}{\hfill}%
\parbox[t]{15cm}{\(xy(A_{3}x + A_{0}y + A_{0}z)zt(A_{1}x + A_{2}y + A_{3}t)(x + y + z + t)(A_{3}x + A_{2}y + A_{0}z + A_{3}t)\)}

\noindent\parbox[right]{9mm}{\textbf{333:}}\parbox{3mm}{\hfill}%
\parbox[t]{15cm}{\(xyz(A_{0}x + A_{1}y + A_{2}z)(A_{0}x + A_{1}y + A_{1}t)t(x + y + z + t)(A_{0}x + A_{0}y + A_{3}z + A_{3}t)\)}

\noindent\parbox[right]{9mm}{\textbf{334:}}\parbox{3mm}{\hfill}%
\parbox[t]{15cm}{\(x(x + t)((-A_{0} - A_{1})x + A_{0}y + A_{1}z)(A_{2}x + A_{0}y + A_{1}z + A_{3}t)y(y + t)z(z + t)\)}

\noindent\parbox[right]{9mm}{\textbf{335:}}\parbox{3mm}{\hfill}%
\parbox[t]{15cm}{\(xyz(A_{0}x + A_{1}y + A_{1}z)(A_{0}x + A_{0}y + A_{2}t)t(A_{0}y + A_{3}z + A_{2}t)(x + y + z + t)\)}

\noindent\parbox[right]{9mm}{\textbf{336:}}\parbox{3mm}{\hfill}%
\parbox[t]{15cm}{\(xyz(A_{0}x + A_{1}y + A_{2}z)(z + t)(A_{0}x + A_{1}y + A_{3}z + A_{3}t)(x + t)(y + t)\)}

\noindent\parbox[right]{9mm}{\textbf{337:}}\parbox{3mm}{\hfill}%
\parbox[t]{15cm}{\(x(x + t)y(y + t)z(z + t)(A_{0}x + A_{1}y + A_{2}z)(A_{3}x + A_{1}y + A_{2}z + (A_{1} + A_{2} + A_{3})t)\)}

\noindent\parbox[right]{9mm}{\textbf{338:}}\parbox{3mm}{\hfill}%
\parbox[t]{15cm}{\(xy(A_{0}x + A_{1}y + A_{0}z)z((A_{0} - A_{1})x + (A_{0} - A_{1})y + A_{0}t)t(x + y + z + t)(A_{0}x + A_{2}z + A_{3}t)\)}

\noindent\parbox[right]{9mm}{\textbf{339:}}\parbox{3mm}{\hfill}%
\parbox[t]{15cm}{\(xyz((-A_{1} + A_{2})x + (-A_{1} + A_{2})y + A_{2}z)t(A_{0}x + A_{1}y + A_{2}t)(A_{0}x + A_{3}z + A_{3}t)(x + y + z + t)\)}

\noindent\parbox[right]{9mm}{\textbf{340:}}\parbox{3mm}{\hfill}%
\parbox[t]{15cm}{\(xyz(A_{0}x + A_{1}y + A_{1}z)(A_{0}x + A_{2}y + A_{0}t)t(A_{2}x + A_{2}y + A_{3}z + A_{0}t)(x + y + z + t)\)}

\noindent\parbox[right]{9mm}{\textbf{341:}}\parbox{3mm}{\hfill}%
\parbox[t]{15cm}{\(xyz(A_{0}x + A_{1}y + A_{2}z)t((-A_{0} + A_{3})x + A_{3}y + A_{3}t)(A_{0}x + A_{3}z + A_{3}t)(x + y + z + t)\)}

\noindent\parbox[right]{9mm}{\textbf{342:}}\parbox{3mm}{\hfill}%
\parbox[t]{15cm}{\(xyz(A_{0}x + A_{1}y + A_{2}z)t((-A_{1} + A_{2})x + (-A_{1} + A_{2})y + A_{2}t)(A_{0}x + A_{3}z + A_{0}t)(x + y + z + t)\)}

\noindent\parbox[right]{9mm}{\textbf{343:}}\parbox{3mm}{\hfill}%
\parbox[t]{15cm}{\(xyz(A_{0}x + A_{1}y + A_{2}z)((A_{0} - A_{2})x + (A_{1} - A_{2})y + (A_{1} - A_{2})t)t(x + y + z + t)(A_{3}x + A_{1}y + A_{2}z + A_{3}t)\)}

\noindent\parbox[right]{9mm}{\textbf{344:}}\parbox{3mm}{\hfill}%
\parbox[t]{15cm}{\(xyz(A_{0}x + A_{1}y + A_{1}z)(A_{0}x + A_{2}y + A_{2}t)t(x + y + z + t)(A_{3}y + (A_{0} - A_{1})z + A_{0}t)\)}

\noindent\parbox[right]{9mm}{\textbf{345:}}\parbox{3mm}{\hfill}%
\parbox[t]{15cm}{\(xyz(A_{0}x + A_{0}y + A_{1}z)t((A_{0} - A_{1} + A_{2})x + (A_{0} - A_{1})y - A_{1}t)(A_{2}x + (-A_{0} + A_{1})z + A_{3}t)(x + y + z + t)\)}

\noindent\parbox[right]{9mm}{\textbf{346:}}\parbox{3mm}{\hfill}%
\parbox[t]{15cm}{\(x(x + t)y(y + t)z(z + t)(A_{0}x + A_{1}y + A_{2}z)(A_{3}x + A_{1}y + A_{2}z + (A_{2} + A_{3})t)\)}

\noindent\parbox[right]{9mm}{\textbf{347:}}\parbox{3mm}{\hfill}%
\parbox[t]{15cm}{\(xyz(A_{0}x + A_{1}y + A_{0}z)(A_{2}x + A_{0}z + A_{2}t)(A_{3}x + A_{1}y + A_{0}z + A_{2}t)t(x + y + z + t)\)}

\noindent\parbox[right]{9mm}{\textbf{348:}}\parbox{3mm}{\hfill}%
\parbox[t]{15cm}{\(xyz(A_{0}x + A_{1}y + A_{1}z)(A_{0}x + A_{2}y + A_{2}t)t(x + y + z + t)(A_{0}x + A_{3}y + A_{1}z + A_{2}t)\)}

\noindent\parbox[right]{9mm}{\textbf{349:}}\parbox{3mm}{\hfill}%
\parbox[t]{15cm}{\(xy(A_{1}x + A_{0}y + A_{0}z)zt(A_{0}x + A_{0}y + A_{3}t)(x + y + z + t)(A_{1}x + A_{0}y + A_{2}z + A_{3}t)\)}

\noindent\parbox[right]{9mm}{\textbf{350:}}\parbox{3mm}{\hfill}%
\parbox[t]{15cm}{\(xyz(A_{0}x + A_{0}y + A_{1}z)t(A_{0}x + A_{2}y + A_{3}t)((A_{0} - A_{2})x + A_{3}z + A_{3}t)(x + y + z + t)\)}

\noindent\parbox[right]{9mm}{\textbf{351:}}\parbox{3mm}{\hfill}%
\parbox[t]{15cm}{\(xyz(A_{0}x + A_{1}y + A_{1}z)(A_{2}x + A_{3}y + A_{0}t)t(A_{3}y + (A_{0} - A_{1})z + A_{0}t)(x + y + z + t)\)}

\noindent\parbox[right]{9mm}{\textbf{352:}}\parbox{3mm}{\hfill}%
\parbox[t]{15cm}{\(xyz(A_{0}x + A_{1}y + A_{2}z)t((-A_{1} + A_{2} - A_{3})x + (-A_{1} + A_{2})y + A_{2}t)(A_{3}x + (-A_{1} + A_{2})z + A_{3}t)(x + y + z + t)\)}

\noindent\parbox[right]{9mm}{\textbf{353:}}\parbox{3mm}{\hfill}%
\parbox[t]{15cm}{\(xyz((A_{1} - A_{3})x + A_{2}y + A_{0}z)(z + t)(A_{1}x + A_{2}y + A_{3}z + A_{3}t)(x + t)(y + t)\)}

\noindent\parbox[right]{9mm}{\textbf{354:}}\parbox{3mm}{\hfill}%
\parbox[t]{15cm}{\(xyz(A_{0}x + A_{3}y + A_{0}z)(A_{2}x + A_{1}y + A_{1}t)t(x + y + z + t)(A_{2}x + A_{3}y + A_{0}z + A_{1}t)\)}

\noindent\parbox[right]{9mm}{\textbf{355:}}\parbox{3mm}{\hfill}%
\parbox[t]{15cm}{\(xyz(A_{0}x + A_{1}y + A_{0}z)t(A_{0}x + A_{2}y + A_{2}t)(A_{0}x + A_{3}z + A_{3}t)(x + y + z + t)\)}

\noindent\parbox[right]{9mm}{\textbf{356:}}\parbox{3mm}{\hfill}%
\parbox[t]{15cm}{\(xy(x + t)(y + t)z(A_{0}x + A_{1}y + A_{2}z)(z + t)(A_{3}x + A_{1}y + A_{2}z + A_{2}t)\)}

\noindent\parbox[right]{9mm}{\textbf{357:}}\parbox{3mm}{\hfill}%
\parbox[t]{15cm}{\(xyz(A_{0}x + A_{0}y + A_{1}z)t((-A_{0} + A_{1})x + A_{2}y + A_{1}t)(A_{0}x + A_{3}z + A_{3}t)(x + y + z + t)\)}

\noindent\parbox[right]{9mm}{\textbf{358:}}\parbox{3mm}{\hfill}%
\parbox[t]{15cm}{\(xyz(A_{0}x + A_{0}y + A_{1}z)t(A_{0}x + A_{2}y + A_{3}t)((A_{0} - A_{2})x + (-A_{2} + A_{3})z + (-A_{2} + A_{3})t)(x + y + z + t)\)}

\noindent\parbox[right]{9mm}{\textbf{359:}}\parbox{3mm}{\hfill}%
\parbox[t]{15cm}{\(xy(A_{0}x + A_{1}y + A_{1}z)zt((A_{2} - A_{3})x + (A_{2} - A_{3})y + A_{3}t)(x + y + z + t)(A_{0}x + A_{2}y + A_{3}z + A_{3}t)\)}

\noindent\parbox[right]{9mm}{\textbf{360:}}\parbox{3mm}{\hfill}%
\parbox[t]{15cm}{\(xyz(A_{0}x + A_{1}y + A_{2}z)(z + t)(A_{0}x + A_{3}y + A_{3}z + A_{3}t)(x + t)(y + t)\)}

\noindent\parbox[right]{9mm}{\textbf{361:}}\parbox{3mm}{\hfill}%
\parbox[t]{15cm}{\(xyz(A_{0}x + A_{1}y + A_{2}z)((A_{0} - A_{2})x + (A_{1} - A_{2})y + (A_{1} - A_{2})t)t(x + y + z + t)(A_{1}y + A_{3}z + A_{3}t)\)}

\noindent\parbox[right]{9mm}{\textbf{362:}}\parbox{3mm}{\hfill}%
\parbox[t]{15cm}{\(xy(y + t)(x + t)(A_{0}x + A_{1}y - A_{1}z)z(z + t)(A_{0}x + A_{2}y + A_{3}z + (A_{0} + A_{3})t)\)}

\noindent\parbox[right]{9mm}{\textbf{363:}}\parbox{3mm}{\hfill}%
\parbox[t]{15cm}{\(xy((-A_{2} + A_{3})x + A_{0}y + A_{0}z)zt(A_{1}x + A_{2}y + A_{3}t)(x + y + z + t)(A_{2}x + A_{2}y + (-A_{0} + A_{3})z + A_{3}t)\)}

\noindent\parbox[right]{9mm}{\textbf{364:}}\parbox{3mm}{\hfill}%
\parbox[t]{15cm}{\(xyz(A_{0}x + (A_{2} - A_{3})y + A_{3}z)(z + t)(A_{1}x + A_{2}y + A_{3}z + A_{3}t)(x + t)(y + t)\)}

\noindent\parbox[right]{9mm}{\textbf{365:}}\parbox{3mm}{\hfill}%
\parbox[t]{15cm}{\(x(A_{0}x + A_{1}y + A_{1}z)yz(A_{3}x + (A_{0} - A_{2})y + A_{3}t)(A_{0}x + A_{0}y + A_{2}z + A_{3}t)t(x + y + z + t)\)}

\noindent\parbox[right]{9mm}{\textbf{366:}}\parbox{3mm}{\hfill}%
\parbox[t]{15cm}{\(x(x + t)y(y + t)z(z + t)(A_{0}x + A_{1}y + (-A_{1} - A_{2})z)(A_{0}A_{3}x + A_{0}A_{1}y + (-A_{0}A_{1} - A_{0}A_{2})z + (-A_{0}A_{2} + A_{2}A_{3})t)\)}

\noindent\parbox[right]{9mm}{\textbf{367:}}\parbox{3mm}{\hfill}%
\parbox[t]{15cm}{\(xy(A_{0}x + A_{1}y + A_{1}z)zt((-A_{0} + A_{1})x + A_{1}y + A_{1}t)(x + y + z + t)(A_{1}y + A_{2}z + A_{3}t)\)}

\noindent\parbox[right]{9mm}{\textbf{368:}}\parbox{3mm}{\hfill}%
\parbox[t]{15cm}{\(xy(A_{0}x + A_{1}y + A_{1}z)zt((-A_{0} + A_{1})x + A_{2}y + A_{1}t)(x + y + z + t)(A_{2}x + A_{2}y + A_{3}z + A_{1}t)\)}

\noindent\parbox[right]{9mm}{\textbf{369:}}\parbox{3mm}{\hfill}%
\parbox[t]{15cm}{\(xy(A_{0}x + A_{1}y + A_{1}z)zt(A_{2}x + A_{3}y + A_{3}t)(x + y + z + t)((A_{0} + A_{2})x + (A_{0} + A_{2})y + A_{1}z + A_{3}t)\)}

\noindent\parbox[right]{9mm}{\textbf{370:}}\parbox{3mm}{\hfill}%
\parbox[t]{15cm}{\(xyz(A_{0}x + A_{1}y + A_{0}z)t(A_{2}x - A_{1}y + A_{2}t)(A_{3}x + A_{0}z + A_{2}t)(x + y + z + t)\)}

\noindent\parbox[right]{9mm}{\textbf{371:}}\parbox{3mm}{\hfill}%
\parbox[t]{15cm}{\(xyz(A_{0}x + A_{1}y + A_{2}z)t(A_{3}x + (-A_{1} + A_{2})y + A_{2}t)(A_{3}x + A_{3}z + A_{2}t)(x + y + z + t)\)}

\noindent\parbox[right]{9mm}{\textbf{372:}}\parbox{3mm}{\hfill}%
\parbox[t]{15cm}{\(x(A_{0}x + A_{2}y + A_{3}z)yzt(A_{1}x + A_{2}y + A_{3}z + A_{2}t)((A_{1} - A_{3})x + (A_{2} - A_{3})y + (A_{1} - A_{3})t)(x + y + z + t)\)}

\noindent\parbox[right]{9mm}{\textbf{373:}}\parbox{3mm}{\hfill}%
\parbox[t]{15cm}{\(xyz(A_{0}x + A_{1}y + A_{1}z)(2A_{2}x + A_{2}y + A_{3}t)t(A_{2}y + 2A_{2}z + A_{3}t)(x + y + z + t)\)}

\noindent\parbox[right]{9mm}{\textbf{374:}}\parbox{3mm}{\hfill}%
\parbox[t]{15cm}{\(x(A_{0}A_{2}x + A_{1}A_{3}y + A_{0}A_{2}z + A_{1}^2t)(A_{0}A_{3}x + A_{0}A_{3}y + A_{0}A_{2}z + A_{1}^2t)y(A_{1}x + A_{3}y + A_{2}z + A_{1}t)zt(x + y + z + t)\)}

\noindent\parbox[right]{9mm}{\textbf{375:}}\parbox{3mm}{\hfill}%
\parbox[t]{15cm}{\(xyz(A_{0}x + A_{1}y + (A_{0} - A_{2} + A_{3})z)(A_{2}x + A_{3}y + A_{3}t)t(x + y + z + t)(A_{2}x + (-A_{0} + A_{1} + A_{2})y + A_{3}z + A_{3}t)\)}

\noindent\parbox[right]{9mm}{\textbf{376:}}\parbox{3mm}{\hfill}%
\parbox[t]{15cm}{\(xy(y + t)(x + t)(x + y - z)z(z + t)(A_{0}x + A_{1}y + A_{2}z + A_{3}t)\)}

\noindent\parbox[right]{9mm}{\textbf{377:}}\parbox{3mm}{\hfill}%
\parbox[t]{15cm}{\(xyz(A_{0}x + A_{1}y + A_{1}z)t(A_{0}x + A_{2}y + A_{2}t)(A_{0}x + A_{3}z + A_{3}t)(x + y + z + t)\)}

\noindent\parbox[right]{9mm}{\textbf{378:}}\parbox{3mm}{\hfill}%
\parbox[t]{15cm}{\(x(x + t)y(y + t)z(z + t)(A_{0}x + A_{2}y + A_{3}z)(A_{1}x + A_{2}y + A_{3}z + (A_{2} + A_{3})t)\)}

\noindent\parbox[right]{9mm}{\textbf{379:}}\parbox{3mm}{\hfill}%
\parbox[t]{15cm}{\(x(A_{0}x + A_{2}y + A_{3}z)yzt(A_{1}x + A_{2}y + A_{3}z + A_{2}t)((A_{1} - A_{3})x + (A_{1} - A_{3})y + (A_{2} - A_{3})t)(x + y + z + t)\)}

\noindent\parbox[right]{9mm}{\textbf{380:}}\parbox{3mm}{\hfill}%
\parbox[t]{15cm}{\(xyz(A_{0}x + A_{0}y + A_{1}z)t(A_{2}x + A_{3}y + A_{1}t)(x + y + z + t)(A_{2}x + A_{3}y + (-A_{0} + A_{1})z + (-A_{0} + A_{1})t)\)}

\noindent\parbox[right]{9mm}{\textbf{381:}}\parbox{3mm}{\hfill}%
\parbox[t]{15cm}{\(xyz(A_{0}x + A_{1}y + A_{0}z)t(A_{2}x + (A_{0} - A_{1})y + A_{0}t)(A_{2}x + A_{3}z + A_{0}t)(x + y + z + t)\)}

\noindent\parbox[right]{9mm}{\textbf{382:}}\parbox{3mm}{\hfill}%
\parbox[t]{15cm}{\((A_{0}A_{1}x + A_{0}A_{2}y + A_{1}A_{2}z + A_{2}A_{3}t)(A_{1}x + A_{2}y + A_{1}z + A_{3}t)xyzt(A_{3}x + A_{0}y + A_{0}z + A_{3}t)(x + y + z + t)\)}

\noindent\parbox[right]{9mm}{\textbf{383:}}\parbox{3mm}{\hfill}%
\parbox[t]{15cm}{\(xyz(A_{0}x + A_{1}y + A_{2}z)t(A_{0}x + A_{1}y - A_{2}t)(A_{0}x + A_{0}y + A_{3}z + A_{3}t)(x + y + z + t)\)}

\noindent\parbox[right]{9mm}{\textbf{384:}}\parbox{3mm}{\hfill}%
\parbox[t]{15cm}{\(x(x + t)(A_{0}x + A_{1}y + A_{2}z + A_{3}t)(A_{0}x + A_{1}y + A_{2}z + A_{4}t)y(y + t)z(z + t)\)}

\noindent\parbox[right]{9mm}{\textbf{385:}}\parbox{3mm}{\hfill}%
\parbox[t]{15cm}{\((A_{1}A_{2}x + A_{0}^2y + A_{1}A_{3}z + A_{0}A_{3}t)x(A_{1}A_{2}x + A_{0}^2y + A_{1}A_{2}z + A_{0}A_{2}t)yz(A_{0}^2x + A_{0}^2y + A_{0}A_{3}z + A_{0}A_{3}t)t(x + y + z + t)\)}

\noindent\parbox[right]{9mm}{\textbf{386:}}\parbox{3mm}{\hfill}%
\parbox[t]{15cm}{\(xyz(A_{2}x + A_{0}y + A_{1}z)((A_{2}^2A_{4} - A_{2}A_{3}A_{4})x + (-A_{2}A_{3}A_{4} + A_{2}A_{4}^2)z + (-A_{2}^2A_{3} + A_{2}^2A_{4})t)(A_{2}x + A_{3}y + A_{4}z + A_{2}t)t(x + y + z + t)\)}

\noindent\parbox[right]{9mm}{\textbf{387:}}\parbox{3mm}{\hfill}%
\parbox[t]{15cm}{\(xyz((A_{0} - A_{4})x + (A_{2} - A_{4})y + (A_{2} - A_{4})z)(A_{0}x + A_{2}y + A_{4}t)t(A_{1}x + A_{2}y + A_{3}z + A_{4}t)(x + y + z + t)\)}

\noindent\parbox[right]{9mm}{\textbf{388:}}\parbox{3mm}{\hfill}%
\parbox[t]{15cm}{\(xyz(A_{0}x + A_{1}y + (-A_{0} - A_{1})z)(x + t)(y + t)(z + t)(A_{0}x + A_{2}y + A_{3}z + A_{4}t)\)}

\noindent\parbox[right]{9mm}{\textbf{389:}}\parbox{3mm}{\hfill}%
\parbox[t]{15cm}{\(xy(A_{0}x + A_{1}y + A_{1}z)zt(A_{0}x + A_{2}y + A_{3}t)(x + y + z + t)(A_{0}x + A_{4}y + A_{1}z + A_{1}t)\)}

\noindent\parbox[right]{9mm}{\textbf{390:}}\parbox{3mm}{\hfill}%
\parbox[t]{15cm}{\(xyz(A_{0}x + A_{2}y + A_{3}z)(A_{1}x + A_{3}z + A_{4}t)(A_{4}x + A_{2}y + A_{3}z + A_{4}t)t(x + y + z + t)\)}

\noindent\parbox[right]{9mm}{\textbf{391:}}\parbox{3mm}{\hfill}%
\parbox[t]{15cm}{\(xyz((A_{2} - A_{4})x + A_{0}y + (A_{3} - A_{4})z)(A_{1}x + A_{2}y + A_{4}t)t(A_{2}x + A_{2}y + A_{3}z + A_{4}t)(x + y + z + t)\)}

\noindent\parbox[right]{9mm}{\textbf{392:}}\parbox{3mm}{\hfill}%
\parbox[t]{15cm}{\(xyz((A_{1} - A_{4})x + (A_{2} - A_{4})y + (A_{2} - A_{4})z)(A_{0}x + A_{2}y + A_{4}t)t(A_{1}x + A_{2}y + A_{3}z + A_{4}t)(x + y + z + t)\)}

\noindent\parbox[right]{9mm}{\textbf{393:}}\parbox{3mm}{\hfill}%
\parbox[t]{15cm}{\(xyz(A_{0}x + A_{2}y + A_{3}z)(A_{4}x + A_{3}z + A_{4}t)(A_{1}x + A_{2}y + A_{3}z + A_{4}t)t(x + y + z + t)\)}

\noindent\parbox[right]{9mm}{\textbf{394:}}\parbox{3mm}{\hfill}%
\parbox[t]{15cm}{\(xyz(A_{0}x + A_{1}y + A_{0}z)(A_{0}x + A_{2}y + A_{4}t)t(A_{2}x + A_{2}y + A_{3}z + A_{4}t)(x + y + z + t)\)}

\noindent\parbox[right]{9mm}{\textbf{395:}}\parbox{3mm}{\hfill}%
\parbox[t]{15cm}{\(xyz(A_{1}x + A_{0}y + A_{1}z)(A_{1}^2A_{4}x + (A_{0}A_{1}A_{2} - A_{0}A_{2}A_{4} + A_{1}A_{2}A_{4})y + A_{1}A_{2}A_{4}t)t(A_{1}^2x + (A_{0}A_{1} - A_{0}A_{4} + A_{1}A_{4})y + A_{1}A_{3}z + A_{1}A_{4}t)(x + y + z + t)\)}

\noindent\parbox[right]{9mm}{\textbf{396:}}\parbox{3mm}{\hfill}%
\parbox[t]{15cm}{\(xyz(A_{1}^2A_{3}x + A_{0}A_{1}A_{3}y + (A_{0}A_{1}A_{4} - A_{0}A_{2}A_{4} + A_{2}A_{3}A_{4})z)(A_{1}^2x + A_{2}A_{4}z + A_{1}A_{2}t)(A_{1}x + A_{3}y + A_{4}z + A_{1}t)t(x + y + z + t)\)}

\noindent\parbox[right]{9mm}{\textbf{397:}}\parbox{3mm}{\hfill}%
\parbox[t]{15cm}{\(xyz(A_{0}x + (A_{2} - A_{4})y + (A_{3} - A_{4})z)(A_{4}x + A_{3}z + A_{4}t)(A_{1}x + A_{2}y + A_{3}z + A_{4}t)t(x + y + z + t)\)}

\noindent\parbox[right]{9mm}{\textbf{398:}}\parbox{3mm}{\hfill}%
\parbox[t]{15cm}{\(xyz(A_{0}x + A_{1}y + A_{1}z)(A_{0}x + A_{2}y + A_{4}t)t(A_{4}x + A_{2}y + A_{3}z + A_{4}t)(x + y + z + t)\)}

\noindent\parbox[right]{9mm}{\textbf{399:}}\parbox{3mm}{\hfill}%
\parbox[t]{15cm}{\(xyz(A_{0}x + A_{0}y + A_{1}z)(A_{0}x + A_{3}y + A_{4}t)t(A_{2}x + A_{3}y + A_{4}z + A_{4}t)(x + y + z + t)\)}

\noindent\parbox[right]{9mm}{\textbf{400:}}\parbox{3mm}{\hfill}%
\parbox[t]{15cm}{\(xyz((A_{1} - A_{4})x + (A_{3} - A_{4})y + A_{0}z)(A_{1}x + A_{3}y + A_{4}t)t(A_{2}x + A_{3}y + A_{4}z + A_{4}t)(x + y + z + t)\)}

\noindent\parbox[right]{9mm}{\textbf{401:}}\parbox{3mm}{\hfill}%
\parbox[t]{15cm}{\(xyz(A_{1}x + A_{2}y + A_{0}z)(A_{1}x + A_{2}y + A_{3}t)t(A_{4}x + A_{2}y + A_{3}z + A_{3}t)(x + y + z + t)\)}

\noindent\parbox[right]{9mm}{\textbf{402:}}\parbox{3mm}{\hfill}%
\parbox[t]{15cm}{\(xyz(A_{1}^3x + A_{0}A_{1}^2y + (A_{0}A_{1}^2 - A_{0}A_{2}A_{4} + A_{1}A_{2}A_{4})z)(A_{1}^2x + A_{2}A_{4}z + A_{1}A_{2}t)(A_{1}x + A_{3}y + A_{4}z + A_{1}t)t(x + y + z + t)\)}

\noindent\parbox[right]{9mm}{\textbf{403:}}\parbox{3mm}{\hfill}%
\parbox[t]{15cm}{\(xy((A_{0}A_{2} - A_{0}A_{4})x + A_{2}A_{4}y + A_{2}A_{4}z)zt(A_{0}x + A_{1}y + A_{2}t)(x + y + z + t)(A_{0}x + A_{3}y + A_{4}z + A_{4}t)\)}

\noindent\parbox[right]{9mm}{\textbf{404:}}\parbox{3mm}{\hfill}%
\parbox[t]{15cm}{\(xyz(A_{2}x + A_{0}y + A_{3}z)(A_{2}^2x + A_{1}A_{2}y + A_{1}A_{4}t)t(A_{2}x + A_{2}y + A_{3}z + A_{4}t)(x + y + z + t)\)}

\noindent\parbox[right]{9mm}{\textbf{405:}}\parbox{3mm}{\hfill}%
\parbox[t]{15cm}{\(xyz(A_{0}x + (A_{2} - A_{4})y + (A_{3} - A_{4})z)(A_{1}x + A_{3}z + A_{4}t)(A_{2}y + A_{3}z + A_{4}t)t(x + y + z + t)\)}

\noindent\parbox[right]{9mm}{\textbf{406:}}\parbox{3mm}{\hfill}%
\parbox[t]{15cm}{\(xyz((-A_{0} + A_{2} + A_{3} - A_{4})x + A_{0}y + A_{1}z)(x + t)(y + t)(z + t)(A_{2}x + A_{3}y + A_{1}z + A_{4}t)\)}

\noindent\parbox[right]{9mm}{\textbf{407:}}\parbox{3mm}{\hfill}%
\parbox[t]{15cm}{\(xyz(A_{0}x + A_{1}y + A_{1}z)(A_{0}x + A_{1}y + A_{4}t)t(A_{2}x + A_{1}y + A_{3}z + A_{4}t)(x + y + z + t)\)}

\noindent\parbox[right]{9mm}{\textbf{408:}}\parbox{3mm}{\hfill}%
\parbox[t]{15cm}{\(xyz(A_{1}x + A_{2}y + A_{0}z)(A_{1}x + A_{2}y + A_{3}t)t(x + y + z + t)((A_{1} - A_{2} + A_{4})x + A_{4}y + (-A_{2} + A_{3} + A_{4})z + (-A_{2} + A_{3} + A_{4})t)\)}

\noindent\parbox[right]{9mm}{\textbf{409:}}\parbox{3mm}{\hfill}%
\parbox[t]{15cm}{\(xyz((A_{1} - A_{4})x + (A_{1} - A_{4})y + (A_{3} - A_{4})z)(A_{0}x + A_{2}y + A_{4}t)t(A_{1}x + A_{2}y + A_{3}z + A_{4}t)(x + y + z + t)\)}

\noindent\parbox[right]{9mm}{\textbf{410:}}\parbox{3mm}{\hfill}%
\parbox[t]{15cm}{\(xyz(A_{0}x + A_{2}y + A_{3}z)(A_{1}x + A_{3}z + A_{4}t)(A_{2}y + A_{3}z + A_{4}t)t(x + y + z + t)\)}

\noindent\parbox[right]{9mm}{\textbf{411:}}\parbox{3mm}{\hfill}%
\parbox[t]{15cm}{\(xyz(A_{0}x + A_{1}y + A_{1}z)(A_{0}x + A_{2}y + A_{4}t)t(A_{2}x + A_{2}y + A_{3}z + A_{4}t)(x + y + z + t)\)}

\noindent\parbox[right]{9mm}{\textbf{412:}}\parbox{3mm}{\hfill}%
\parbox[t]{15cm}{\(xyz((A_{1} - A_{3})x + (A_{2} - A_{3})y + A_{0}z)(A_{1}x + A_{2}y + A_{3}t)t(x + y + z + t)(A_{1}x + A_{2}y + A_{4}z + A_{4}t)\)}

\noindent\parbox[right]{9mm}{\textbf{413:}}\parbox{3mm}{\hfill}%
\parbox[t]{15cm}{\(xy(A_{0}x + A_{1}y + A_{1}z)zt(A_{0}x + A_{2}y + A_{3}t)(x + y + z + t)(A_{4}x + (A_{2} + A_{3})y + A_{3}z + A_{3}t)\)}

\noindent\parbox[right]{9mm}{\textbf{414:}}\parbox{3mm}{\hfill}%
\parbox[t]{15cm}{\(xyz(A_{2}x + A_{0}y + A_{1}z)(x + t)(y + t)(z + t)(A_{2}x + A_{3}y + A_{4}z + (A_{2} + A_{3} + A_{4})t)\)}

\noindent\parbox[right]{9mm}{\textbf{415:}}\parbox{3mm}{\hfill}%
\parbox[t]{15cm}{\(xyz(A_{0}x + (A_{0} - A_{1} + A_{2})y + (A_{0} - A_{1} + A_{3})z)(A_{4}x + A_{3}z + A_{4}t)(A_{1}x + A_{2}y + A_{3}z + A_{4}t)t(x + y + z + t)\)}

\noindent\parbox[right]{9mm}{\textbf{416:}}\parbox{3mm}{\hfill}%
\parbox[t]{15cm}{\(xyz((A_{1} - A_{3})x + (A_{2} - A_{3})y + A_{0}z)(A_{1}x + A_{2}y + A_{3}t)t(x + y + z + t)((A_{1} - A_{3})x + (A_{2} - A_{3})y + A_{4}z + A_{4}t)\)}

\noindent\parbox[right]{9mm}{\textbf{417:}}\parbox{3mm}{\hfill}%
\parbox[t]{15cm}{\(xyz((A_{2} - A_{3})x + A_{0}y + (-A_{3} + A_{4})z)(A_{2}x + A_{2}y + A_{1}t)t(x + y + z + t)(A_{2}x + A_{3}y + A_{4}z + A_{3}t)\)}

\noindent\parbox[right]{9mm}{\textbf{418:}}\parbox{3mm}{\hfill}%
\parbox[t]{15cm}{\(xyz((A_{2} - A_{4})x + A_{0}y + (A_{3} - A_{4})z)((A_{2} - A_{4})x + A_{0}y + A_{1}t)t(x + y + z + t)(A_{2}x + A_{4}y + A_{3}z + A_{4}t)\)}

\noindent\parbox[right]{9mm}{\textbf{419:}}\parbox{3mm}{\hfill}%
\parbox[t]{15cm}{\(xyz(A_{0}x + A_{1}y + A_{1}z)(A_{0}x + A_{2}y + A_{3}t)t(A_{4}x + A_{2}y + A_{3}z + A_{3}t)(x + y + z + t)\)}

\noindent\parbox[right]{9mm}{\textbf{420:}}\parbox{3mm}{\hfill}%
\parbox[t]{15cm}{\(xyz(A_{0}x + A_{1}y + A_{0}z)(A_{0}x + A_{3}y + A_{4}t)t(A_{2}x + A_{3}y + A_{4}z + A_{4}t)(x + y + z + t)\)}

\noindent\parbox[right]{9mm}{\textbf{421:}}\parbox{3mm}{\hfill}%
\parbox[t]{15cm}{\(xyz(A_{0}x + A_{3}y + A_{4}z)(A_{1}x + A_{3}y + A_{5}t)t(A_{2}x + A_{3}y + A_{4}z + A_{5}t)(x + y + z + t)\)}

\noindent\parbox[right]{9mm}{\textbf{422:}}\parbox{3mm}{\hfill}%
\parbox[t]{15cm}{\(xyz(A_{0}x + A_{1}y + A_{2}z)(A_{0}x + A_{3}z + A_{4}t)(A_{4}x + A_{5}y + A_{3}z + A_{4}t)t(x + y + z + t)\)}

\noindent\parbox[right]{9mm}{\textbf{423:}}\parbox{3mm}{\hfill}%
\parbox[t]{15cm}{\(xyz(A_{0}x + A_{1}y + A_{2}z)(A_{0}x + A_{3}z + A_{0}t)(A_{4}x + A_{5}y + A_{3}z + A_{0}t)t(x + y + z + t)\)}

\noindent\parbox[right]{9mm}{\textbf{424:}}\parbox{3mm}{\hfill}%
\parbox[t]{15cm}{\(xyz(A_{0}x + A_{1}y + A_{1}z)(A_{0}x + A_{2}y + A_{3}t)t(A_{4}x + A_{2}y + A_{5}z + A_{3}t)(x + y + z + t)\)}

\noindent\parbox[right]{9mm}{\textbf{425:}}\parbox{3mm}{\hfill}%
\parbox[t]{15cm}{\(xyz(A_{0}x + A_{1}y + A_{0}z)(A_{0}x + A_{2}y + A_{3}t)t(A_{4}x + A_{2}y + A_{5}z + A_{3}t)(x + y + z + t)\)}

\noindent\parbox[right]{9mm}{\textbf{426:}}\parbox{3mm}{\hfill}%
\parbox[t]{15cm}{\(xyz(A_{0}x + A_{1}y + A_{2}z)(A_{3}x + A_{4}y + A_{5}t)t((A_{0} + A_{5})x + A_{4}y + (A_{2} + A_{5})z + A_{5}t)(x + y + z + t)\)}

\noindent\parbox[right]{9mm}{\textbf{427:}}\parbox{3mm}{\hfill}%
\parbox[t]{15cm}{\(xyz(A_{0}x + A_{1}y + A_{2}z)(A_{0}x + A_{3}y + A_{4}t)t(A_{5}x + A_{3}y + A_{4}z + A_{4}t)(x + y + z + t)\)}

\noindent\parbox[right]{9mm}{\textbf{428:}}\parbox{3mm}{\hfill}%
\parbox[t]{15cm}{\(xyz(A_{0}x + A_{1}y + A_{2}z)(A_{0}x + A_{3}y + A_{4}t)t(x + y + z + t)((-A_{0} + A_{5})x + (-A_{4} + A_{5})y + A_{5}z + (-A_{4} + A_{5})t)\)}

\noindent\parbox[right]{9mm}{\textbf{429:}}\parbox{3mm}{\hfill}%
\parbox[t]{15cm}{\(xyz(A_{3}x + A_{0}y + A_{1}z)((-A_{3} + A_{4})x + (A_{4} - A_{5})y + A_{2}t)t(x + y + z + t)(A_{3}x + A_{5}y + A_{4}z + A_{5}t)\)}

\noindent\parbox[right]{9mm}{\textbf{430:}}\parbox{3mm}{\hfill}%
\parbox[t]{15cm}{\(xyz(A_{0}x + A_{1}y + A_{2}z)((A_{0} - A_{2})x + (A_{1} - A_{2})y + A_{3}t)t(x + y + z + t)(A_{0}x + A_{4}y + A_{5}z + A_{4}t)\)}

\noindent\parbox[right]{9mm}{\textbf{431:}}\parbox{3mm}{\hfill}%
\parbox[t]{15cm}{\(xyz((A_{0} + A_{2} - A_{4})x + (A_{0} + A_{3} - A_{4})y + A_{0}z)(A_{1}x + A_{3}y + A_{5}t)t(A_{2}x + A_{3}y + A_{4}z + A_{5}t)(x + y + z + t)\)}

\noindent\parbox[right]{9mm}{\textbf{432:}}\parbox{3mm}{\hfill}%
\parbox[t]{15cm}{\(xyz(A_{2}x + A_{0}y + A_{4}z)(A_{1}x + A_{3}y + A_{5}t)t(A_{2}x + A_{3}y + A_{4}z + A_{5}t)(x + y + z + t)\)}

\noindent\parbox[right]{9mm}{\textbf{433:}}\parbox{3mm}{\hfill}%
\parbox[t]{15cm}{\(xyz(A_{0}x + A_{1}y + A_{2}z)(A_{0}x + A_{3}z + A_{4}t)(A_{5}y + A_{3}z + A_{4}t)t(x + y + z + t)\)}

\noindent\parbox[right]{9mm}{\textbf{434:}}\parbox{3mm}{\hfill}%
\parbox[t]{15cm}{\(xyz(A_{0}x + A_{1}y + A_{2}z)((A_{0} + A_{3})x + (A_{1} + A_{3})y + A_{3}t)t(x + y + z + t)(A_{0}x + A_{4}y + A_{5}z + A_{5}t)\)}

\noindent\parbox[right]{9mm}{\textbf{435:}}\parbox{3mm}{\hfill}%
\parbox[t]{15cm}{\(xyz((A_{3} - A_{5})x + A_{0}y + (A_{4} - A_{5})z)(A_{3}x + A_{1}y + A_{2}t)t(x + y + z + t)(A_{3}x + A_{5}y + A_{4}z + A_{5}t)\)}

\noindent\parbox[right]{9mm}{\textbf{436:}}\parbox{3mm}{\hfill}%
\parbox[t]{15cm}{\(xyz((A_{0} + A_{2} - A_{4})x + (A_{0} + A_{3} - A_{4})y + A_{0}z)((A_{1} + A_{2} - A_{5})x + (A_{1} + A_{3} - A_{5})y + A_{1}t)t(x + y + z + t)(A_{2}x + A_{3}y + A_{4}z + A_{5}t)\)}

\noindent\parbox[right]{9mm}{\textbf{437:}}\parbox{3mm}{\hfill}%
\parbox[t]{15cm}{\(xyz(A_{0}x + A_{1}y + A_{2}z)(A_{0}x + A_{1}y + A_{3}t)t(x + y + z + t)(A_{0}x + A_{4}y + A_{5}z + A_{5}t)\)}

\noindent\parbox[right]{9mm}{\textbf{438:}}\parbox{3mm}{\hfill}%
\parbox[t]{15cm}{\(xyz((A_{2} - A_{5})x + (A_{3} - A_{5})y + A_{0}z)((A_{2} - A_{4})x + (A_{3} - A_{4})y + A_{1}t)t(x + y + z + t)(A_{2}x + A_{3}y + A_{4}z + A_{5}t)\)}

\noindent\parbox[right]{9mm}{\textbf{439:}}\parbox{3mm}{\hfill}%
\parbox[t]{15cm}{\(xyz(A_{0}x + A_{1}y + A_{2}z)(A_{0}x + A_{3}y + A_{4}t)t(A_{5}x + A_{3}y + A_{6}z + A_{4}t)(x + y + z + t)\)}

\noindent\parbox[right]{9mm}{\textbf{440:}}\parbox{3mm}{\hfill}%
\parbox[t]{15cm}{\(xyz(A_{0}x + A_{1}y + A_{2}z)(A_{0}x + A_{3}y + A_{4}t)t(x + y + z + t)((A_{0} + A_{6})x + (A_{1} + A_{6})y + A_{5}z + A_{6}t)\)}

\noindent\parbox[right]{9mm}{\textbf{441:}}\parbox{3mm}{\hfill}%
\parbox[t]{15cm}{\(xyz(A_{0}x + A_{1}y + A_{2}z)((A_{0} + A_{3})x + (A_{1} + A_{3})y + A_{3}t)t(x + y + z + t)(A_{0}x + A_{4}y + A_{5}z + A_{6}t)\)}

\noindent\parbox[right]{9mm}{\textbf{442:}}\parbox{3mm}{\hfill}%
\parbox[t]{15cm}{\(xyz(A_{0}x + A_{1}y + A_{2}z)(A_{0}x + A_{3}y + A_{4}t)t(x + y + z + t)(A_{0}x + A_{5}y + A_{6}z + A_{5}t)\)}

\noindent\parbox[right]{9mm}{\textbf{443:}}\parbox{3mm}{\hfill}%
\parbox[t]{15cm}{\(xyz(A_{2}x + A_{0}y + A_{1}z)(A_{2}x + A_{3}y + A_{4}t)t(x + y + z + t)((A_{2} - A_{4} + A_{6})x + (A_{3} - A_{4} + A_{6})y + A_{5}z + A_{6}t)\)}

\noindent\parbox[right]{9mm}{\textbf{444:}}\parbox{3mm}{\hfill}%
\parbox[t]{15cm}{\(xyz(A_{0}x + A_{1}y + A_{2}z)(A_{0}x + A_{1}y + A_{3}t)t(x + y + z + t)(A_{0}x + A_{4}y + A_{5}z + A_{6}t)\)}

\noindent\parbox[right]{9mm}{\textbf{445:}}\parbox{3mm}{\hfill}%
\parbox[t]{15cm}{\(xyz(A_{0}x + A_{1}y + A_{2}z)(A_{0}x + A_{3}y + A_{4}t)t(x + y + z + t)(A_{0}x + A_{5}y + A_{6}z + A_{6}t)\)}

\noindent\parbox[right]{9mm}{\textbf{446:}}\parbox{3mm}{\hfill}%
\parbox[t]{15cm}{\(xyz(A_{0}x + A_{1}y + A_{2}z)(A_{0}x + A_{3}y + A_{4}t)t(x + y + z + t)(A_{0}x + A_{5}y + A_{6}z + A_{7}t)\)}

\noindent\parbox[right]{9mm}{\textbf{447:}}\parbox{3mm}{\hfill}%
\parbox[t]{15cm}{\(xyz(A_{0}x + A_{1}y + A_{2}z)t(A_{0}x + A_{3}y + A_{4}z + A_{5}t)(A_{6}x + A_{3}y + A_{4}z + A_{7}t)(x + y + z + t)\)}

\noindent\parbox[right]{9mm}{\textbf{448:}}\parbox{3mm}{\hfill}%
\parbox[t]{15cm}{\(xyz(A_{2}x + A_{0}y + A_{1}z)t(A_{2}x + A_{3}y + A_{4}z + A_{5}t)((A_{2} - A_{3} + A_{6})x + A_{6}y + (-A_{3} + A_{4} + A_{6})z + A_{7}t)(x + y + z + t)\)}

\noindent\parbox[right]{9mm}{\textbf{449:}}\parbox{3mm}{\hfill}%
\parbox[t]{15cm}{\(xyz(A_{0}x + A_{1}y + A_{2}z)t(A_{0}x + A_{3}y + A_{4}z + A_{5}t)(A_{0}x + A_{6}y + A_{7}z + A_{8}t)(x + y + z + t)\)}

\noindent\parbox[right]{9mm}{\textbf{450:}}\parbox{3mm}{\hfill}%
\parbox[t]{15cm}{\(xyzt(x + y + z + t)(A_{0}x + A_{1}y + A_{2}z + A_{3}t)(A_{0}x + A_{4}y + A_{5}z + A_{6}t)(A_{0}x + A_{7}y + A_{8}z + A_{9}t)\)}

\noindent\parbox[right]{9mm}{\textbf{451:}}\parbox{3mm}{\hfill}%
\parbox[t]{15cm}{\((x + \frac{\sqrt{-3} - 1}2y)xy(y + t)(x + t)(-x + -\frac{\sqrt{-3} + 1}2y + \frac{\sqrt{-3} - 1)}2z - t)z(z + t)\)}

\noindent\parbox[right]{9mm}{\textbf{452:}}\parbox{3mm}{\hfill}%
\parbox[t]{15cm}{\(xy(\frac{\sqrt{-3} + 3}2x + y + \frac{\sqrt{-3} + 3}2z)zt(\frac32(\sqrt{-3} + 1)x + \frac{\sqrt{-3} + 3}2y + 3t)(y + \frac{\sqrt{-3} + 3}2z + \frac{\sqrt{-3} + 3}2t)(x + y + z + t)\)}

\noindent\parbox[right]{9mm}{\textbf{453:}}\parbox{3mm}{\hfill}%
\parbox[t]{15cm}{\(xyz(x + y + \frac{\sqrt{5} + 1}2z)(\frac{\sqrt{5} + 3}2x + \frac{3\sqrt{5} + 7}2y + \frac{3\sqrt{5} + 7}2t)t(x + y + z + t)(\frac{\sqrt{5} + 3}2x + \frac{3\sqrt{5} + 7}2z + \frac{3\sqrt{5} + 7}2t)\)}

\noindent\parbox[right]{9mm}{\textbf{454:}}\parbox{3mm}{\hfill}%
\parbox[t]{15cm}{\(x(A_{0}x + A_{1}y + A_{1}z)yzt(x + y + z + t)(A_{0}x + A_{0}y + \frac{1}{2}(\sqrt{-3} + 1)A_{0}t)(A_{0}y + \frac{1}{2}(\sqrt{-3} + 1)A_{0}z + A_{0}t)\)}

\noindent\parbox[right]{9mm}{\textbf{455:}}\parbox{3mm}{\hfill}%
\parbox[t]{15cm}{\(xy(A_{1}x + A_{1}y + (A_{1} + A_{2})z)zt(A_{0}A_{2}x + (A_{1}A_{2} + A_{2}^2)y + (A_{0}A_{1} + A_{0}A_{2})t)(A_{0}A_{1}x + A_{1}A_{2}y + (A_{1}A_{2} + A_{2}^2)z + A_{0}A_{1}t)(x + y + z + t)\)}

\subsection{Magma code}\leavevmode

We use the following magma code to compute detailed information on an octic arrangement and the associated double octic.
\begin{itemize}\def\itemsep{2mm}\leftskip=-8mm\setlength\labelsep{15mm}
	\setlength  \itemindent{-5mm}
	\item[] \texttt{IncidenceTable} computes the minimal incidence table of an octic arrangement, defined as a sequence of eight linear polynomials in four variables, an incidence table is presented as a sequence of quadruples of eight digits $\{1,\dots,8\}$,
	\item[] \texttt{InvariantPermutaions} computes all permutations of the set $\{1,\dots,8\}$ which preserve a given incidence table up to permutation (as explained in section~\ref{sec:projtransf}),
	\item[] \texttt{MinimalIncidenceTable} computes minimal incidence table for a given incidence table,
	\item[] \texttt{Singularities} computes singularities of an octic arrangement defining the given incidence table, returns lists of $p_3, p_4^0, p_4^1, p_5^0, p_5^1$ and $p_5^2$ and  $l_2, l_3$ lines (as lists of planes),
	\item[] \texttt{ArrInvariants} return numbers of $p_3, p_4^0, p_4^1, p_5^0, p_5^1, p_5^2$ points and $l_{3}$ lines of an octic arrangement, Hodge numbers $h^{1,2}$, $h^{1,1}$ and the Euler characteristic $\chi$ of the corresponding double octic.
\end{itemize}

\lstset{language=magma,breaklines=false,breakatwhitespace=true,basicstyle=\tiny, comment=[s]{/*}{*/},commentstyle=\tiny\emph, showstringspaces=false, keepspaces=true, breaklines=true,  tabsize=2 }

\begin{lstlisting}
function RecursiveSort(II) local result;
	result:=[];
	for i:=1 to #II do
		result[i]:=Sort(II[i]);
	end for;
	return Sort(result);
end function;

/*
Computes Incidence Table for an octic arrangement
*/
function IncidenceTable(OcticArr);
	IT:=[];
	Mat:=JacobianMatrix(OcticArr);
	for Q in Subsets({1,2,3,4,5,6,7,8},4) do
		if Minor(Mat,SetToSequence(Q),[1,2,3,4]) eq 0 then
			Include(~IT,SetToSequence(Q));
		end if;
	end for;
	return RecursiveSort(IT);
end function;

/*
Computes  invariant permutations for an incidence table
*/
function InvariantPermutations(IT) local symmetries;
	symmetries:=[];
	for x in Sym(8) do
		NewIT:=RecursiveSort(IT^x);
		if NewIT eq IT then
			Include(~symmetries,x);
		end if;
	end for;
	return symmetries;
end function;

/*
Computes Minimal incidence table for an Incidence Table
*/
function MinimalIncidenceTable(IT) local ITMin;
	ITMin:=IT;
	for x in Sym(8) do
		NewIT:=RecursiveSort(IT^x);
		if NewIT lt ITMin then
			ITMin:=NewIT;
		end if;
	end for;
	return ITMin;
end function;

/*
Computes singularities of an octic arrangement
returns: p_3, p_4^0, p_4^1, p_5^0, p_5^1, p_5^2, l_2, l_3
*/
function Singularities(IT);
	fivefold_points:=[SetToSequence(P): P in Subsets({1..8},5)| forall{Q: Q in Subsets(P,4)| \
		SetToSequence(Q) in IT }];
	fourfold_points:=[P: P in IT | forall {Q: Q in fivefold_points| SequenceToSet(P) \
		notsubset SequenceToSet(Q)}];
	triple_points:=[SetToSequence(P): P in Subsets({1..8},3)| forall{Q:Q in IT |P \
		notsubset SequenceToSet(Q)}];
	triple_lines:=[SetToSequence(P): P in Subsets({1..8},3)| forall{Q: Q in \
		[QQ: QQ in Subsets({1..8},4)| P subset QQ]| SetToSequence(Q) in IT }];
	double_lines:=[SetToSequence(i) : i in Subsets({1..8},2)| SetToSequence(i) notin \
		&cat([[[ii[1],ii[2]],[ii[1],ii[3]],[ii[2],ii[3]]]:ii in triple_lines])];
	fourfold_points_zero:=[];
	fourfold_points_one:=[];
	fivefold_points_zero:=[];
	fivefold_points_one:=[];
	fivefold_points_two:=[];
	for P in fivefold_points do
		nn:=#{Q: Q in triple_lines| IsSubsequence(Q,P:Kind:="Setwise")};
		if nn eq 2 then
			Include(~fivefold_points_two,P);
		elif nn eq 1 then
			Include(~fivefold_points_one,P);
		else
			Include(~fivefold_points_zero,P);
		end if;
	end for;
	Sort(~fivefold_points_zero);
	Sort(~fivefold_points_one);
	Sort(~fivefold_points_two);
	Sort(~fourfold_points_zero);
	Sort(~fourfold_points_one);
	Sort(~triple_lines);
	Sort(~triple_points);
	Sort(~double_lines);
	for P in fourfold_points do
		nn:=#{Q: Q in triple_lines| IsSubsequence(Q,P:Kind:="Setwise")};
		if nn eq 1 then
			Include(~fourfold_points_one,P);
		else
			Include(~fourfold_points_zero,P);
		end if;
	end for;
	return [triple_points,fourfold_points_zero,fourfold_points_one,fivefold_points_zero, \
		fivefold_points_one, fivefold_points_two, double_lines, triple_lines];
end function;

/*
Computes topological invariants of adouble octic
returns: #p_3, #p_4^0, #p_4^1, #p_5^0, #p_5^1, #p_5^2, #l_3, h^1,2, h^1,2, \xi
*/
function ArrInvariants(OcticArr) local OcticEq, JF, Ieq;
	SingList:=Singularities(IncidenceTable(OcticArr));
	OcticEq:=&*OcticArr;
	JF:=JacobianIdeal(OcticEq);
	Ieq:=ideal<Parent(OcticEq)|1>;
	for i in &cat(SingList) do
		Ieq:=Ieq meet (JF+Ideal([OcticArr[k]:k in i])^(#i));
	end for;
	deform:=Coefficients(HilbertSeries(JF,9))[9]-Coefficients(HilbertSeries(Ieq,9))[9];
	euler_characteristic:=40+4*#SingList[2]+3*#SingList[3]+16*#SingList[4]+ \ 18*#SingList[5]+20*#SingList[6]+#SingList[8];
	picard_number:=euler_characteristic/2+deform;
	return([#SingList[i]: i in [1,2,3,4,5,6,8]] cat [deform,picard_number, \
		euler_characteristic]);
end function;
\end{lstlisting}

\section*{Acknowledgment}
We are indebted to Christian Meyer sharing with us a list of equations of octic arrangement used to compile the list in \cite{Meyer}. Thanks to Matthias Sch\"utt and Duco van Straten for helpful discussions and explanations.
The first named author is grateful to Max Planck Institute for Mathematics in Bonn for its hospitality and financial support.
This research was partially supported by the National Science Center
grant no. 2020/39/B/ST1/03358 and  by PL--Grid Infrastructure.

\end{document}